\documentclass[12pt]{amsart}
\usepackage{graphicx}
\usepackage{amsmath, amscd, amssymb, amsthm}
\usepackage[subnum]{cases}
\usepackage{changepage}
\usepackage{xcolor}
\usepackage{marginnote}
\usepackage{hyperref}
\usepackage{cleveref}
\usepackage[all]{xy}
\usepackage{tabularx}
\usepackage{ltablex}
\usepackage{longtable}
\usepackage[T1]{fontenc}

\newtheorem{theorem}{Theorem}[section]
\newtheorem{lemma}[theorem]{Lemma}

\newtheorem{proposition}[theorem]{Proposition}
\newtheorem{corollary}[theorem]{Corollary}

\theoremstyle{definition}
\newtheorem{definition}[theorem]{Definition}
\newtheorem{remark}[theorem]{Remark}
\numberwithin{equation}{section}
\newtheorem{example}[theorem]{Example}
\newtheorem{assumption}[theorem]{Assumption}
\newtheorem{setting}[theorem]{Setting}

\usepackage{fancyhdr}
\begin{document}

\normalfont

\title{Period Rings with Big Coefficients and Application II}
\author{Xin Tong}

\maketitle

\begin{abstract}
\rm  We continue our study on the corresponding noncommutative deformation of the relative $p$-adic Hodge structures of Kedlaya-Liu along our previous work. In this paper, we are going to initiate the study of the corresponding descent of pseudocoherent modules carrying large noncommutative coefficients. And also we are going to more systematically study the corresponding noncommutative geometric aspects of noncommutative deformation of Hodge structures, which will definitely also provide the insights not only for noncommutative Iwasawa theory but also for noncommutative analytic geometry. The noncommutative Hodge-Iwasawa theory is now improved along some very well-defined direction (we will expect many well-targeted applications to noncommutative Tamagawa number conjectures from the modern perspectives of Burns-Flach-Fukaya-Kato), while the corresponding Kedlaya-Liu glueing of pseudocoherent Banach modules with certain stability is also generalized to the large noncommutative coefficient case. 
\end{abstract}

\newpage

\tableofcontents

\newpage

\section{Introduction}

\subsection{Noncommutative Topological Pseudocoherence after Illusie-Kedlaya-Liu}

\noindent In our previous work on noncommutative deformation of period rings and application \cite{T1}, \cite{T2} and \cite{T3}, we considered many useful deformation of the sheaves from many deep constructions coming from Kedlaya-Liu. The corresponding local pictures in both commutative setting and noncommutative setting were carefully developed. \\

\indent The globalization of the picture will happen both in the corresponding coefficients and in the corresponding underlying spaces. Both of them require very deep descent results, although in the commutative setting one do not have to modify much from \cite{KL1} and \cite{KL2}. However in the noncommutative setting, things will be very complicated. Since if one would like to really consider the corresponding localization with respect to the noncommutative coefficient Banach algebra, one has to study the corresponding noncommutative toposes. This will be really complicated since the corresponding topological issues are really  hard to manipulate on the level of toposes. That being said, with respect to some of the corresponding application in mind, we could actually do some extension of Kedlaya-Liu glueing to noncommutative setting by keeping track of the noncommutative spatial information in the following sense.\\

\indent Over $\mathbb{Q}_p$ or $\mathbb{F}_p((t))$, in our application in mind we consider the corresponding deformation of period rings in various context and the corresponding period sheaves over certain sites. One could deal with such very big Banach sheaves by considering the corresponding localization concentrated at the corresponding commutative part but release the corresponding deformation components. This will basically be slightly different from Kedlaya-Liu's original extension to Banach context \cite{KL2} from Illusie's original notion of pseudocoherence in \cite{SGAVI}. What is happening is that we will have some stability notion which is only sensible to the commutative spatial part. To be more precise when we consider the product of a commutative Banach adic algebra $A$ with a noncommutative Banach algebra $B$, we will consider just the corresponding stability with respect to the corresponding rational localization from $A$ (not at all for $B$). Therefore when we sheafify the modules with such stability we will have a sheaf locally attached to modules with such stability in the same fashion.\\

\indent In some of the applications one might want to consider the corresponding deformations by more general coefficients, for instance in noncommutative Iwasawa theory one might want to consider the corresponding noncommutative Iwasawa algebras. Therefore in such situation one could deform the structure sheaves by some pro-Banach algebras, motivated by Burns-Flach-Fukaya-Kato \cite{BF1}, \cite{BF2} and \cite{FK1}. On the other hand suppose one considers deformations by some limit of Fr\'echet rings such as the corresponding ring $B_e$ in the theory of Bergers $B$-pairs in \cite{Ber1} and the corresponding various types of Robba rings with respect to some radius $r>0$ or nothing (namely the full Robba rings) in \cite{KL1} and \cite{KL2}. Carrying these LF algebras one can consider some types of mixed-type of Hodge-structures along the work of \cite{CKZ}, \cite{PZ} and \cite{Ked1} in more general sense or in the original context.\\

\indent In this paper, we will also consider the following discussion. First we would like to tackle the corresponding integral aspects of the corresponding story, especially when we have the adic spaces which are not Tate but just analytic (since we are dealing with spaces over $\mathbb{Z}_p$) in the sense of \cite{Ked2}. Certainly we will carry some large coefficients. We will expect that this will have the potential application to integral $p$-adic Hodge theory such as in \cite{BMS} and \cite{BS}. That being said, the corresponding descent in this paper will definitely have its own interests. We work in detail out the corresponding descent of the corresponding pseudocoherent sheaves with some stability carrying the corresponding large coefficients in the corresponding uniform analytic Huber pair situation, while we also translated several results to the corresponding uniform analytic adic Banach rings after \cite{KL2}. \\

\subsection{Main Results}

\indent  The corresponding notions of stably-pseudocoherence is generalized to noncommutative deformed setting from \cite{KL2}. Over affinoids we realized the following theorems around taking global section of pseudocoherent sheaves. Namely with the corresponding notations in \cref{theorem2.14}:

\begin{theorem}\mbox{\bf{(After Kedlaya-Liu \cite[Theorem 2.5.5]{KL2})}} In analytic topology, taking the global section will realize an equivalence between the category of pseudocoherent $\mathcal{O}\widehat{\otimes} B$-sheaves and the category of $B$-stably-pseudocoherent $\mathcal{O}(X)\widehat{\otimes}B$ modules. Here $X$ is an adic affinoid $\mathrm{Spa}(A,A^+)$. 
	
\end{theorem}

\indent In further application, one can also establish the desired parallel results in the \'etale topology. Namely with the corresponding notations in \cref{theorem2.22}:

\begin{theorem}\mbox{\bf{(After Kedlaya-Liu \cite[Theorem 2.5.14]{KL2})}} In \'etale topology, taking the global section will realize an equivalence between the category of pseudocoherent $\mathcal{O}_\text{\'et}\widehat{\otimes} B$-sheaves and the category of $B$-\'etale-stably-pseudocoherent $\mathcal{O}_\text{\'et}(X)\widehat{\otimes}B$ modules. Here $X$ is an adic affinoid $\mathrm{Spa}(A,A^+)$.  
	
\end{theorem}

\indent As we mentioned one could also consider more general deformation such as the corresponding. The first scope of the consideration could be made in order to include the context of pro-Banach families associated to some $p$-adic rational Iwasawa algebra with Banach coefficients $B[[G]]$ attached to some $p$-adic Lie group $G$. Namely with the corresponding notations in \cref{theorem3.15}:

\begin{theorem}\mbox{\bf{(After Kedlaya-Liu \cite[Theorem 2.5.5]{KL2})}} In analytic topology, taking the global section will realize an equivalence between the category of pseudocoherent $\mathcal{O}\widehat{\otimes} B[[G]]$-sheaves and the category of $B[[G]]$-stably-pseudocoherent $\mathcal{O}(X)\widehat{\otimes}B[[G]]$ modules. Here $X$ is an adic affinoid $\mathrm{Spa}(A,A^+)$, and here all the modules are pro-systems over $\mathbb{Q}_p[[G]]$. 
	
\end{theorem}

\indent In further application, one can also establish the desired parallel results in the \'etale topology. Namely with the corresponding notations in \cref{theorem3.23}:

\begin{theorem}\mbox{\bf{(After Kedlaya-Liu \cite[Theorem 2.5.14]{KL2})}} In \'etale topology, taking the global section will realize an equivalence between the category of pseudocoherent $\mathcal{O}_\text{\'et}\widehat{\otimes} B[[G]]$-sheaves and the category of $B[[G]]$-\'etale-stably-pseudocoherent $\mathcal{O}_\text{\'et}(X)\widehat{\otimes}B[[G]]$ modules. Here $X$ is an adic affinoid $\mathrm{Spa}(A,A^+)$, and here all the modules are pro-systems over $\mathbb{Q}_p[[G]]$.	
\end{theorem}

\indent One can also even consider the limit of Fr\'echet coefficients with some further application to some equivariant sheaves over adic Fargues-Fontaine curves. Namely with the corresponding notations in \cref{theorem4.14}:

\begin{theorem}\mbox{\bf{(After Kedlaya-Liu \cite[Theorem 2.5.5]{KL2})}} In analytic topology, taking the global section will realize an equivalence between the category of pseudocoherent $\mathcal{O}\widehat{\otimes} B$-sheaves and the category of $B$-stably-pseudocoherent $\mathcal{O}(X)\widehat{\otimes}B$ modules. Here $X$ is an adic affinoid $\mathrm{Spa}(A,A^+)$, and here $B$ is injective limit of Banach algebras $\varinjlim_h B_h$. 
	
\end{theorem}

\indent In further application, one can also establish the desired parallel results in the \'etale topology. Namely with the corresponding notations in \cref{theorem4.22}:

\begin{theorem}\mbox{\bf{(After Kedlaya-Liu \cite[Theorem 2.5.14]{KL2})}} In \'etale topology, taking the global section will realize an equivalence between the category of pseudocoherent $\mathcal{O}_\text{\'et}\widehat{\otimes} B$-sheaves and the category of $B$-\'etale-stably-pseudocoherent $\mathcal{O}_\text{\'et}(X)\widehat{\otimes}B$ modules. Here $X$ is an adic affinoid $\mathrm{Spa}(A,A^+)$, and here $B$ is injective limit of Banach algebras $\varinjlim_h B_h$.  
	
\end{theorem}

\indent One can also consider the corresponding families of affinoids in some coherent way for the commutative parts, namely this is the notion of quasi-Stein spaces from \cite[Chapter 2.6]{KL2}. Then with the corresponding notations in \cref{theorem2.31}, \cref{theorem3.32} and \cref{theorem4.31}:

\begin{theorem}\mbox{\bf{(After Kedlaya-Liu \cite[Proposition 2.6.17]{KL2})}} Carrying Banach, pro-Banach or limit of Fr\'echet coefficient $B$ one could have the finiteness of the global sections of $B$-stably-pseudocoherent sheaves over some quasi-Stein space $X$ as long as one has $m$-uniform covering as in \cite[Proposition 2.6.17]{KL2}.  
	
\end{theorem}

\indent The application will be encoded in our last chapter, where we focused on finding relationship between the corresponding Frobenius equivariant sheaves over Fargues-Fontaine space:

\begin{align}
\bigcup \mathrm{Spa}(\widetilde{\Pi}_R^{[r_1,r_2]},\widetilde{\Pi}_R^{[r_1,r_2],+})	
\end{align}
and the corresponding Frobenius equivariant modules over the global sections:
\begin{align}
\widetilde{\Pi}_R^{[r_1,r_2]},\forall 0<s<r.	
\end{align}
in glueing fashion. The sheaves and modules here carrying some big Banach coefficients (noncommutative). And we also considered some very interesting pro-Banach situation and LF situation, see \cref{theorem5.18} and \cref{theorem5.24}.\\

\subsection{Further Consideration}

\indent One thing we have observed (although have not written so) is that suppose we apply our results to some locally noetherian spaces, then we will have no stability issue, and we will have not just pseudoflatness but even more. Instead, this will immediately mean that we are going to be more algebraic throughout. The corresponding story is very interesting in the situation of rigid analytic spaces.\\ 

\indent The discussion in the previous paragraph in fact implies that in the noetherian situation, one should be able to extend the discussion around $\infty$-glueing along the style of Kedlaya-Liu by applying Bambozzi-Kremnizer spectrum \cite{BK1} and Clausen-Scholze space \cite{CS} to the corresponding coherent sheaves which are interesting to us where the corresponding highly nontrivial stability and nonflatness issues in \cite{KL2} disappear. Again note that we are still carrying some large coefficients.\\

\newpage

\section{Foundations on Noncommutative Descent for Adic Spectra in Banach Case}

\subsection{Noncommutative Pseudocoherence in Analytic Topology}

\indent We now establish some foundations in the noncommutative setting on glueing noncommutative pseudocoherent modules after \cite[Chapter 2]{KL2}. But we remind the readers that we will fix the base space:

\begin{setting} \label{setting2.1}
Consider a corresponding sheafy Banach adic uniform algebra $(A,A^+)$ over $\mathbb{Q}_p$ or $\mathbb{F}_p((t))$, we consider the base space $\mathrm{Spa}(A,A^+)$ as in \cite[Chapter 2]{KL2}. And now we will consider a further noncommutative Banach algebra $(B,B^+)$ over $\mathbb{Q}_p$ or $\mathbb{F}_p((t))$. 	
\end{setting}

\begin{definition} \mbox{\bf{(After Kedlaya-Liu \cite[Definition 2.4.1]{KL2})}}
For any left $A\widehat{\otimes}B$-module $M$, we call it $m$-$B$-stably pseudocoherent if we have that it is $m$-$B$-pseudocoherent, complete for the natural topology and for any morphism $A\rightarrow A'$ which is the corresponding rational localization, the base change of $M$ to $A'\widehat{\otimes} B$ is complete for the natural topology as the corresponding left $A'\widehat{\otimes} B$-module.	As in \cite[Definition 2.4.1]{KL2} we call that the corresponding left $A\widehat{\otimes}B$-module $M$ just $B$-stably pseudocoherent if we have that it is simply just $\infty$-$B$-stably pseudocoherent.
\end{definition}

\begin{remark}
Throughout the whole paper, it is obvious we need to use the corresponding noncommutative version of the corresponding pseudocoherent modules on the algebraic level due to Illusie along the development of Grothendieck's SGAVI \cite{SGAVI}, but the main difference is just that we consider the corresponding modules with action happening on one side.
\end{remark}

\indent Since our spaces are commutative, so the rational localization is not flat which is exactly the same situation as in \cite{KL1} and \cite{KL2}, therefore in our situation we need to define something weaker than the corresponding flatness which will be sufficient for us to apply in the following development: 

\begin{definition}\mbox{\bf{(After Kedlaya-Liu \cite[Definition 2.4.4]{KL2})}}
For any Banach left $A\widehat{\otimes}B$-module $M$, we call it is $m$-$B$-pseudoflat if for any right $A\widehat{\otimes}B$-module $M'$ $m$-$B$-stably pseudocoherent we have $\mathrm{Tor}_1^{A\widehat{\otimes}B}(M',M)=0$. For any Banach right $A\widehat{\otimes}B$-module $M$, we call it is $m$-$B$-pseudoflat if for any left $A\widehat{\otimes}B$-module $M'$ $m$-$B$-stably pseudocoherent we have $\mathrm{Tor}_1^{A\widehat{\otimes}B}(M,M')=0$.
\end{definition}

\begin{definition}\mbox{\bf{(After Kedlaya-Liu \cite[Definition 2.4.6]{KL2})}} We consider the corresponding notion of the corresponding pro-projective module. We define over $A$ the corresponding $B$-pro-projective module $M$ to be a corresponding left $A\widehat{\otimes}B$-module such that one could find a filtered sequence of projectors such that in the sense of taking the prolimit we have the corresponding projector will converge to any chosen element in $M$. Here we assume the module is complete with respect to natural topology, and we assume the projectors are $A\widehat{\otimes}B$-linear and we assume that the corresponding image of the projectors are modules which are also finitely generated and projective.
	
\end{definition}

\begin{lemma}\mbox{\bf{(After Kedlaya-Liu \cite[Lemma 2.4.7]{KL2})}}
Suppose we have over the space $\mathrm{Spa}(A,A^+)$ a corresponding $B$-pro-projective left module $M$. And suppose that we have a $2$-$B$-pseudocoherent right module $C$ over $A$. And we assume that $C$ is complete with respect to the natural topology. Then we have that the corresponding product $C{\otimes}_{A\widehat{\otimes}B}M$ is then complete under the corresponding natural topology. And moreover we have that in our situation:
\begin{displaymath}
\mathrm{Tor}_1^{A\widehat{\otimes}B}(C,M)=0.	
\end{displaymath}

\end{lemma}

\begin{proof}
This will be just a noncommutative version of the corresponding  \cite[Lemma 2.4.7]{KL2}. We adapt the corresponding argument in \cite[Lemma 2.4.7]{KL2} to our situation. What we are going to consider in this case is then first consider the following presentation:
\[
\xymatrix@C+0pc@R+0pc{
M'  \ar[r]\ar[r]\ar[r] &(A\widehat{\otimes}B)^k \ar[r]\ar[r]\ar[r] & M\ar[r]\ar[r]\ar[r] &0,
}
\]
where the left module $M'$ is finitely presented. Then we consider the following commutative diagram:
\[
\xymatrix@C+0pc@R+3pc{
&C\otimes_{A\widehat{\otimes}B}M'  \ar[r]\ar[r]\ar[r]  \ar[d]\ar[d]\ar[d] &C\otimes_{A\widehat{\otimes}B}(A\widehat{\otimes}B)^k \ar[r]\ar[r]\ar[r] \ar[d]\ar[d]\ar[d] & C\otimes_{A\widehat{\otimes}B}M\ar[r]\ar[r]\ar[r] \ar[d]\ar[d]\ar[d] &0,\\
0 \ar[r]\ar[r]\ar[r] &C\widehat{\otimes}_{A\widehat{\otimes}B}M'  \ar[r]\ar[r]\ar[r] &C\widehat{\otimes}_{A\widehat{\otimes}B}(A\widehat{\otimes}B)^k \ar[r]\ar[r]\ar[r] & C\widehat{\otimes}_{A\widehat{\otimes}B}M\ar[r]\ar[r]\ar[r] &0.
}
\]\\
The first row is exact as in \cite[Lemma 2.4.7]{KL2}, and we have the corresponding second row is also exact since we have that by hypothesis the corresponding module $M$ is $B$-pro-projective.  
Then as in \cite[Lemma 2.4.7]{KL2} the results follow.	
\end{proof}

\begin{proposition} \mbox{\bf{(After Kedlaya-Liu \cite[Corollary 2.4.8]{KL2})}}
Over the space $\mathrm{Spa}(A,A^+)$, we have any $B$-pro-projective left module is $2$-$B$-pseudoflat.
\end{proposition}

\begin{proof}
This is a direct consequence of the previous lemma.	
\end{proof}

\begin{lemma} \mbox{\bf{(After Kedlaya-Liu \cite[Corollary 2.4.9]{KL2})}}
Keep the notation above, suppose we are working over the adic space $
\mathrm{Spa}(A,A^+)$. Suppose now the left $A\widehat{\otimes}B$-module is finitely generated, then we have that the following natural maps:
\begin{align}
A\{T\}\widehat{\otimes}B\otimes_{A\widehat{\otimes}B} M	\rightarrow M\{T\},\\\
A\{T,T^{-1}\}\widehat{\otimes}B \otimes_{A\widehat{\otimes}B} M 	\rightarrow M\{T,T^{-1}\}	
\end{align}
are surjective. Suppose now the left $A\widehat{\otimes}B$-module is finitely presented, then we have that the following natural maps:
\begin{align}
A\{T\}\widehat{\otimes}B\otimes_{A\widehat{\otimes}B} M	\rightarrow M\{T\},\\\
A\{T,T^{-1}\}\widehat{\otimes}B \otimes_{A\widehat{\otimes}B} M 	\rightarrow M\{T,T^{-1}\}	
\end{align}
are bijective. Here all the modules are assumed to be compete with respect to the natural topology in our context.

\end{lemma}

\begin{proof}
This is the corresponding consequence of the previous lemma.
\end{proof}

\begin{proposition} \mbox{\bf{(After Kedlaya-Liu \cite[Lemma 2.4.12]{KL2})}}
We keep the corresponding notations in \cite[Lemma 2.4.10]{KL2}. Then in our current sense we have that the corresponding morphism $A\rightarrow B_2$ is then $2$-$B$-pseudoflat. 
\end{proposition}

\begin{proof}
The corresponding proof could be made parallel to the corresponding proof of \cite[Lemma 2.4.12]{KL2}. We need to keep track the corresponding coefficient $B$ in our current context. To be more precise we consider the following commutative diagram from the corresponding chosen exact sequence (which is just the analog of the corresponding one in \cite[Lemma 2.4.12]{KL2}, namely as below we choose arbitrary short exact sequence with $P$ as a left $A\widehat{\otimes}B$-module which is $2$-$B$-stably pseudocoherent):
\[
\xymatrix@C+0pc@R+0pc{
0   \ar[r]\ar[r]\ar[r] &M \ar[r]\ar[r]\ar[r] &N \ar[r]\ar[r]\ar[r] & P \ar[r]\ar[r]\ar[r] &0,
}
\]
which then induces the following corresponding big commutative diagram:
\[
\xymatrix@C+0pc@R+4pc{
& &0 \ar[d]\ar[d]\ar[d] &0 \ar[d]\ar[d]\ar[d] &,\\
0   \ar[r]\ar[r]\ar[r]  &{A\{T\}\widehat{\otimes}B}\otimes_{A\widehat{\otimes}B} M \ar[r]\ar[r]\ar[r] \ar[d]^{1-fT}\ar[d]\ar[d] &{A\{T\}\widehat{\otimes}B}\otimes_{A\widehat{\otimes}B} N \ar[r]\ar[r]\ar[r] \ar[d]^{1-fT}\ar[d]\ar[d] &{A\{T\}\widehat{\otimes}B} \otimes_{A\widehat{\otimes}B}P \ar[r]\ar[r]\ar[r] \ar[d]^{1-fT}\ar[d]\ar[d] &0,\\
0   \ar[r]\ar[r]\ar[r] &{A\{T\}\widehat{\otimes}B}\otimes_{A\widehat{\otimes}B}M  \ar[r]\ar[r]\ar[r] \ar[d]\ar[d]\ar[d] &{A\{T\}\widehat{\otimes}B} \otimes_{A\widehat{\otimes}B}N\ar[r]\ar[r]\ar[r] \ar[d]\ar[d]\ar[d] & {A\{T\}\widehat{\otimes}B}\otimes_{A\widehat{\otimes}B}P \ar[r]\ar[r]\ar[r] \ar[d]\ar[d]\ar[d] &0,\\
0  \ar[r]^?\ar[r]\ar[r] &{B_2\widehat{\otimes}B}\otimes_{A\widehat{\otimes}B} M \ar[r]\ar[r]\ar[r] \ar[d]\ar[d]\ar[d] &{B_2\widehat{\otimes}B}\otimes_{A\widehat{\otimes}B}N \ar[r]\ar[r]\ar[r] \ar[d]\ar[d]\ar[d] & {B_2\widehat{\otimes}B}\otimes_{A\widehat{\otimes}B} P \ar[r]\ar[r]\ar[r] \ar[d]\ar[d]\ar[d] &0,\\
&0&0&0
}
\]
with the notations in \cite[Lemma 2.4.10]{KL2} where we actually have the corresponding exactness along the horizontal direction at the corner around ${B_2\widehat{\otimes}B}\otimes_{A\widehat{\otimes}B} M $ marked with $?$, by diagram chasing.
\end{proof}

\begin{proposition} \mbox{\bf{(After Kedlaya-Liu \cite[Lemma 2.4.13]{KL2})}} \label{proposition2.9}
We keep the corresponding notations in \cite[Lemma 2.4.10]{KL2}. Then in our current sense we have that the corresponding morphism $A\rightarrow B_1$ is then $2$-$B$-pseudoflat. And in our current sense we have that the corresponding morphism $A\rightarrow B_{12}$ is then $2$-$B$-pseudoflat. 
\end{proposition}

\begin{proof}
The statement for $A\rightarrow B_{12}$ could be proved by consider the composition $A\rightarrow B_2\rightarrow B_{12}$ where the $2$-$B$-pseudoflatness could be proved as in \cite[Lemma 2.4.13]{KL2} by inverting the corresponding variables. For $A\rightarrow B_{1}$,
the corresponding proof could be made parallel to the corresponding proof of \cite[Lemma 2.4.13]{KL2}. To be more precise we consider the following commutative diagram from the corresponding chosen exact sequence (which is just the analog of the corresponding one in \cite[Remark 2.4.5]{KL2}, namely as below we choose arbitrary short exact sequence with $P$ as a left $A\widehat{\otimes}B$-module which is $2$-$B$-stably pseudocoherent):
\[
\xymatrix@C+0pc@R+0pc{
0   \ar[r]\ar[r]\ar[r] &M \ar[r]\ar[r]\ar[r] &N \ar[r]\ar[r]\ar[r] & P \ar[r]\ar[r]\ar[r] &0,
}
\]
which then induces the following corresponding big commutative diagram:
\[\tiny
\xymatrix@C+0pc@R+5pc{
& &0 \ar[d]\ar[d]\ar[d] &0 \ar[d]\ar[d]\ar[d] &,\\
0   \ar[r]\ar[r]\ar[r]  &M \ar[r]\ar[r]\ar[r] \ar[d]\ar[d]\ar[d] &N \ar[r]\ar[r]\ar[r] \ar[d]\ar[d]\ar[d] & P \ar[r]\ar[r]\ar[r] \ar[d]\ar[d]\ar[d] &0,\\
0   \ar[r]^?\ar[r]\ar[r] &({B_1\widehat{\otimes}B}\bigoplus {B_2\widehat{\otimes}B}) \otimes_{A\widehat{\otimes}B}M  \ar[r]\ar[r]\ar[r] \ar[d]\ar[d]\ar[d] &({B_1\widehat{\otimes}B}\bigoplus {B_2\widehat{\otimes}B})\otimes_{A\widehat{\otimes}B}N  \ar[r]\ar[r]\ar[r] \ar[d]\ar[d]\ar[d] &({B_1\widehat{\otimes}B}\bigoplus {B_2\widehat{\otimes}B})\otimes_{A\widehat{\otimes}B}P\ar[r]\ar[r]\ar[r] \ar[d]\ar[d]\ar[d] &0,\\
0  \ar[r]\ar[r]\ar[r] &{B_{12}\widehat{\otimes}B}\otimes_{A\widehat{\otimes}B}M \ar[r]\ar[r]\ar[r] \ar[d]\ar[d]\ar[d] &{B_{12}\widehat{\otimes}B}\otimes_{A\widehat{\otimes}B}N \ar[r]\ar[r]\ar[r] \ar[d]\ar[d]\ar[d] &{B_{12}\widehat{\otimes}B} \otimes_{A\widehat{\otimes}B} P \ar[r]\ar[r]\ar[r] \ar[d]\ar[d]\ar[d] &0,\\
&0&0&0
}
\]
with the notations in \cite[Lemma 2.4.10]{KL2} where we actually have the corresponding exactness along the horizontal direction at the corner around $M\otimes_{A\widehat{\otimes}B} ({B_1\widehat{\otimes}B}\bigoplus {B_2\widehat{\otimes}B})$ marked with $?$, by diagram chasing.
\end{proof}

\indent After these foundational results, as in \cite{KL2} we have the following proposition which is the corresponding noncommutative generalization of the corresponding result established in \cite[Theorem 2.4.15]{KL2}.

\begin{proposition} \mbox{\bf{(After Kedlaya-Liu \cite[Theorem 2.4.15]{KL2})}} \label{proposition2.10}
In our current context we have that for any rational localization $\mathrm{Spa}(A',A^{',+})\rightarrow \mathrm{Spa}(A,A^+)$ we have that along this base change the corresponding $B$-stably pseudocoherence is preserved.\\
	
\end{proposition}

\indent Then we have the corresponding Tate's acyclicity in the noncommutative deformed setting:

\begin{theorem}\mbox{\bf{(After Kedlaya-Liu \cite[Theorem 2.5.1]{KL2})}} \label{theorem2.11} Now suppose we have in our corresponding \cref{setting2.1} a corresponding $B$-stably pseudocoherent module $M$. Then we consider the corresponding assignment such that for any $U\subset \mathrm{Spa}(A,A^+)$ we define $\widetilde{M}(U)$ as in the following:
\begin{align}
\widetilde{M}(U):=\varprojlim_{\mathrm{Spa}(S,S^+)\subset U,\mathrm{rational}} S\widehat{\otimes}B\otimes_{A\widehat{\otimes}B}M.	
\end{align}
Then we have that for any $\mathfrak{B}$ which is a rational covering of $U=\mathrm{Spa}(S,S^+)\subset \mathrm{Spa}(A,A^+)$ (certainly this $U$ is also assumed to be rational) we have that the vanishing of the following two cohomology groups:
\begin{align}
H^i(U,\widetilde{M}), \check{H}^i(U:\mathfrak{B},\widetilde{M})
\end{align}
for any $i>0$. When concentrating at the degree zero we have:
\begin{align}
H^0(U,\widetilde{M})=S\widehat{\otimes}B\otimes_{A\widehat{\otimes}B}M, \check{H}^0(U:\mathfrak{B},\widetilde{M})=S\widehat{\otimes}B\otimes_{A\widehat{\otimes}B}M.
\end{align}
	
\end{theorem}

\begin{proof}
By \cite[Propositions 2.4.20-2.4.21]{KL1}, we could then finish proof as in \cite[Theorem 2.5.1]{KL2} by using our previous \cref{proposition2.10} and \cref{proposition2.9} as above.	
\end{proof}

\indent We now consider the corresponding noncommutative deformed version of Kiehl's glueing properties for stably pseudocoherent modules after \cite{KL2}:

\begin{definition} \mbox{\bf{(After Kedlaya-Liu \cite[Definition 2.5.3]{KL2})}}
Consider in our context (over $\mathrm{Spa}(A,A^+)$) the corresponding sheaves $\mathcal{O}\widehat{\otimes}B$, we will then define the corresponding pseudocoherent sheaves over $\mathrm{Spa}(A,A^+)$ to be those locally defined by attaching stably-pseudocoherent modules over the section. 
\end{definition}

\begin{lemma} \mbox{\bf{(After Kedlaya-Liu \cite[Lemma 2.5.4]{KL2})}}\label{lemma2.13}
	Consider the corresponding notations in \cite[Lemma 2.4.10]{KL2}, we have the corresponding morphism $A\rightarrow B_1\bigoplus B_2$. Then we have that this morphism is an descent morphism effective for the corresponding $B$-stably pseodocoherent Banach modules. 
\end{lemma}

\begin{proof}
This is a $B$-relative version of the \cite[Lemma 2.5.4]{KL2}. We adapt the corresponding proof to our situation. First consider the corresponding notations on the commutative adic space $\mathrm{Spa}(A,A^+)$ as in \cite[Lemma 2.4.10]{KL2}. And we have a coving $\{U_1,U_2\}$. Then consider a descent datum for this cover of $B$-stably pseudocoherent sheaves. By \cref{theorem2.11} we could realize this as a corresponding pseudocoherent sheaf over $\mathcal{O}\widehat{\otimes}B$. We denote this by $\mathcal{M}$. Then by (the proof of) \cite[Lemma 6.82]{T2}, we have that this is finitely generated and we have the corresponding vanishing of $\check{H}^i(\mathrm{Spa}(A,A^+),\mathfrak{B};\mathcal{M})$ for $i>0$. Then by \cref{theorem2.11} again we have that the vanishing of $H^i(U_j;\mathcal{M})$ for $j\in \{1,2\},i>0$. Then as in \cite[Lemma 2.5.4]{KL2} we have the corresponding vanishing of $H^i(\mathrm{Spa}(A,A^+);\mathcal{M})$ for $i>0$. Then as in \cite[Lemma 2.5.4]{KL2} we choose a corresponding covering of some finite free object $\mathcal{F}$ and take its kernel $\mathcal{K}$ which is actually pseudocoherent by the corresponding \cref{proposition2.10}. Then forming the following diagram:
\[\tiny
\xymatrix@C+0pc@R+6pc{
0 \ar[r]\ar[r]\ar[r] & (B_p\widehat{\otimes}B)\otimes \mathcal{K}(\mathrm{Spa}(A,A^+))  \ar[r]\ar[r]\ar[r] \ar[d]\ar[d]\ar[d] &(B_p\widehat{\otimes}B)\otimes \mathcal{F}(\mathrm{Spa}(A,A^+))\ar[r]\ar[r]\ar[r] \ar[d]\ar[d]\ar[d] & (B_p\widehat{\otimes}B)\otimes \mathcal{M}(\mathrm{Spa}(A,A^+))\ar[r]\ar[r]\ar[r]  \ar[d]\ar[d]\ar[d]&0,\\
0 \ar[r]\ar[r]\ar[r] &\mathcal{K}(U_p)   \ar[r]\ar[r]\ar[r] &\mathcal{F}(U_p) \ar[r]\ar[r]\ar[r] & \mathcal{M}(U_p) \ar[r]\ar[r]\ar[r] &0,\\
}
\]
where we have $p=1,2$. The first row is exact by $B$-pseudoflatness as in \cref{proposition2.10} and the second is row is exact as well by \cref{theorem2.11}. The middle vertical arrow is an isomorphism while the rightmost vertical arrow is surjective and the leftmost vertical arrow is surjective. The surjectivity just here comes from actually \cite[Lemma 6.82]{T2}. This will imply that the rightmost vertical arrow is injective as well. Then we have the corresponding isomorphism $\mathcal{M}\overset{\sim}{\rightarrow} \widetilde{\mathcal{M}(\mathrm{Spa}(A,A^+))}$. Then we consider the following diagram:

\[ \tiny
\xymatrix@C+0pc@R+6pc{
 &(C\widehat{\otimes}B)\otimes\mathcal{K}(\mathrm{Spa}(A,A^+))   \ar[r]\ar[r]\ar[r] \ar[d]\ar[d]\ar[d] &(C\widehat{\otimes}B)\otimes\mathcal{F}(\mathrm{Spa}(A,A^+)) \ar[r]\ar[r]\ar[r] \ar[d]\ar[d]\ar[d] & (C\widehat{\otimes}B)\otimes\mathcal{M}(\mathrm{Spa}(A,A^+))\ar[r]\ar[r]\ar[r]  \ar[d]\ar[d]\ar[d]&0,\\
0 \ar[r]\ar[r]\ar[r] &\mathcal{K}(\mathrm{Spa}(C,C^+))   \ar[r]\ar[r]\ar[r] &\mathcal{F}(\mathrm{Spa}(C,C^+)) \ar[r]\ar[r]\ar[r] & \mathcal{M}(\mathrm{Spa}(C,C^+)) \ar[r]\ar[r]\ar[r] &0,\\
}
\]
$C$ is any rational localization of $A$. Here the first row is exact by $B$-pseudoflatness and the second is row is exact as well. The middle vertical arrow is an isomorphism while the rightmost vertical arrow is surjective as in \cite[Lemma 2.5.4]{KL2}. Then we apply the same argument to the leftmost vertical arrow (with a new diagram), we will have the corresponding surjectivity of this leftmost vertical arrow. This will imply that the rightmost vertical arrow is injective as well. Then we have the corresponding isomorphism for the rightmost vertical arrow.
\end{proof}

\begin{theorem}\mbox{\bf{(After Kedlaya-Liu \cite[Theorem 2.5.5]{KL2})}} \label{theorem2.14}
Taking global section will realize the equivalence between the following two categories: A. The category of all the pseudocoherent $\mathcal{O}\widehat{\otimes}B$-sheaves; B. The category of all the $B$-stably pseudocoherent modules over $A\widehat{\otimes}B$. 	
\end{theorem}

\begin{proof}
See \cite[Theorem 2.5.5]{KL2}. We need to still apply \cite[Proposition 2.4.20]{KL1}, as long as one considers instead in our situation \cref{lemma2.13} and \cref{theorem2.11}.\\	
\end{proof}

\subsection{Noncommutative Pseudocoherence in \'Etale Topology}

\indent We consider the extension of the corresponding discussion in the previous subsection to the \'etale topology. At this moment we consider the following same context on the geometric level as in \cite{KL2}:

\begin{setting} \mbox{\bf{(After Kedlaya-Liu \cite[Hypothesis 2.5.8]{KL2})}} \label{setting2.15}
As in the previous subsection we consider now the corresponding geometric setting namely an adic space $\mathrm{Spa}(A,A^+)$ where $A$ is assumed to be sheafy. Then we consider $B$ as above. And we consider the corresponding \'etale site $\mathrm{Spa}(A,A^+)_\text{\'et}$. And by keeping the corresponding setting in the geometry as in \cite[Hypothesis 2.5.8]{KL2}, we assume that there is a stable basis $\mathfrak{B}$ containing the space $\mathrm{Spa}(A,A^+)$ itself.	
\end{setting}

\begin{definition} \mbox{\bf{(After Kedlaya-Liu \cite[Definition 2.5.9]{KL2})}}
For any left $A\widehat{\otimes}B$-module $M$, we call it $m$-$B$-\'etale stably pseudocoherent with respect to $\mathfrak{B}$ if we have that it is $m$-$B$-pseudocoherent, complete for the natural topology and for any rational localization $A\rightarrow A'$, the base change of $M$ to $A'\widehat{\otimes} B$ is complete for the natural topology as the corresponding left $A'\widehat{\otimes} B$-module.	As in \cite[Definition 2.4.1]{KL2} we call that the corresponding left $A\widehat{\otimes}B$-module $M$ just $B$-stably pseudocoherent if we have that it is simply just $\infty$-$B$-stably pseudocoherent.
\end{definition}

\begin{definition}\mbox{\bf{(After Kedlaya-Liu \cite[Below Definition 2.5.9]{KL2})}}
For any Banach left $A\widehat{\otimes}B$-module $M$, we call it $m$-$B$-\'etale-pseudoflat if for any right $A\widehat{\otimes}B$-module $M'$ $m$-$B$-\'etale-stably pseudocoherent we have $\mathrm{Tor}_1^{A\widehat{\otimes}B}(M',M)=0$. For any Banach right $A\widehat{\otimes}B$-module $M$, we call it is $m$-$B$-\'etale-pseudoflat if for any left $A\widehat{\otimes}B$-module $M'$ $m$-$B$-\'etale-stably pseudocoherent we have $\mathrm{Tor}_1^{A\widehat{\otimes}B}(M,M')=0$.
\end{definition}

\indent The following proposition holds in our current setting.

\begin{proposition} \mbox{\bf{(After Kedlaya-Liu \cite[Lemma 2.5.10]{KL2})}} \label{proposition2.18}
One can actually find another basis $\mathfrak{C}$ in $\mathfrak{B}$ such that any morphism in $\mathfrak{C}$ could be $2$-$B$-\'etale-pseudoflat with respect to either $\mathfrak{C}$ or $\mathfrak{B}$.	
\end{proposition}

\begin{proof}
We follow the proof of \cite[Lemma 2.5.10]{KL2} in our current $B$-relative situation, however not that much needs to be adapted by relying on the proof of \cite[Lemma 2.5.10]{KL2}. To be more precise the corresponding selection of the new basis $\mathfrak{C}$ comes from including all the morphisms made up of some composition of rational localization and the corresponding finite \'etale ones. For these we have shown above the corresponding 2-$B$-\'etale pseudoflatness. Then the rest will be pure geometric for the analytic spaces which are just the same as the situation of \cite[Lemma 2.5.10]{KL2}, therefore we just omit the corresponding argument, see \cite[Lemma 2.5.10]{KL2}.	
\end{proof}

\indent Then we have the corresponding Tate's acyclicity in the noncommutative deformed setting in \'etale topology (here fix $\mathfrak{C}$ as above):

\begin{theorem}\mbox{\bf{(After Kedlaya-Liu \cite[Theorem 2.5.11]{KL2})}} \label{theorem2.19} Now suppose we have in our corresponding \cref{setting2.1} a corresponding $B$-\'etale-stably pseudocoherent module $M$. Then we consider the corresponding assignment such that for any $U\subset \mathrm{Spa}(A,A^+)$ we define $\widetilde{M}(U)$ as in the following:
\begin{align}
\widetilde{M}(U):=\varprojlim_{\mathrm{Spa}(S,S^+)\subset U,\in \mathfrak{C}} S\widehat{\otimes}B\otimes_{A\widehat{\otimes}B}M.	
\end{align}
Then we have that for any $\mathfrak{B}$ which is a covering of $U=\mathrm{Spa}(S,S^+)\subset \mathrm{Spa}(A,A^+)$ (certainly this $U$ is also assumed to be in $\mathfrak{C}$, and we assume this covering is formed by using the corresponding members in $\mathfrak{C}$) we have that the vanishing of the following two cohomology groups:
\begin{align}
H^i(U,\widetilde{M}), \check{H}^i(U,\mathfrak{B},\widetilde{M})
\end{align}
for any $i>0$. When concentrating at the degree zero we have:
\begin{align}
H^0(U,\widetilde{M})=S\widehat{\otimes}B\otimes_{A\widehat{\otimes}B}M, \check{H}^0(U,\mathfrak{B},\widetilde{M})=S\widehat{\otimes}B\otimes_{A\widehat{\otimes}B}M.
\end{align}
	
\end{theorem}

\begin{proof}
By \cite[Proposition 8.2.21]{KL1}, we could then finish proof as in \cite[Theorem 2.5.11]{KL2} by using our previous \cref{proposition2.18} and the corresponding faithfully flat descent as in the situation of \cite[Theorem 2.5.11]{KL2}.	
\end{proof}

\indent We now consider the corresponding noncommutative deformed version of Kiehl's glueing property for stably pseudocoherent modules after \cite[Definition 2.5.12]{KL2}:

\begin{definition} \mbox{\bf{(After Kedlaya-Liu \cite[Definition 2.5.12]{KL2})}}
Consider in our context (over $\mathrm{Spa}(A,A^+)_{\text{\'et}}$) the corresponding sheaves $\mathcal{O}_{\mathrm{Spa}(A,A^+)_{\text{\'et}}}\widehat{\otimes}B$, we will then define the corresponding pseudocoherent sheaves over $\mathrm{Spa}(A,A^+)_{\text{\'et}}$ to be those locally defined by attaching \'etale-stably-pseudocoherent modules over the section. 
\end{definition}

\begin{lemma} \mbox{\bf{(After Kedlaya-Liu \cite[Lemma 2.5.13]{KL2})}} \label{lemma2.21}
	Consider the corresponding notations in \cite[Lemma 2.4.10]{KL2}, we have the corresponding morphism $A\rightarrow B_1\bigoplus B_2$. Then we have that this morphism is an descent morphism effective for the corresponding $B$-\'etale-stably pseodocoherent Banach modules. 
\end{lemma}

\begin{proof}
This is a $B$-relative version of the \cite[Lemma 2.5.13]{KL2}. We adapt the corresponding proof to our situation. First consider the corresponding notations on the commutative adic space $\mathrm{Spa}(A,A^+)$ as in \cite[Lemma 2.4.10]{KL2}. And we have a coving $\{U_1,U_2\}$. Then consider a descent datum for this cover of $B$-\'etale-stably pseudocoherent sheaves. By \cref{theorem2.19} we could realize this as a corresponding pseudocoherent sheaf over $\mathcal{O}_{\text{\'et}}\widehat{\otimes}B$. We denote this by $\mathcal{M}$. Then by consider the corresponding results in the rational localization situation above we have that the vanishing of $H^i(U_j;\mathcal{M})$ for $j\in \{1,2\},i>0$. And note that we have the corresponding equality on the degree zero namely $H^i_\text{\'et}(U_j;\mathcal{M})=H^i(U_j;\mathcal{M})$ for just $i$ which is zero. Then as in \cite[Lemma 2.5.13]{KL2} we have the corresponding vanishing of $H^i(\mathrm{Spa}(A,A^+);\mathcal{M})$ for $i>0$. Namely this is due to the fact that the covering is already one in the analytic topology. Then as in \cite[Lemma 2.5.13]{KL2} we choose a corresponding covering of some finite free object $\mathcal{F}$ and take its kernel $\mathcal{K}$ which gives rise to the corresponding exact sequence:
\[\tiny
\xymatrix@C+0pc@R+6pc{
0 \ar[r]\ar[r]\ar[r] &\mathcal{K}(\mathrm{Spa}(A,A^+))   \ar[r]\ar[r]\ar[r] &\mathcal{F}(\mathrm{Spa}(A,A^+)) \ar[r]\ar[r]\ar[r] & \mathcal{M}(\mathrm{Spa}(A,A^+)) \ar[r]\ar[r]\ar[r] &0,
}
\]
where we have $p=1,2$. We have already the corresponding stably-pseudocoherence of the global section. As in \cite[Lemma 2.5.13]{KL2} we only need to show that it is stable under \'etale morphism. Then we consider the following diagram:

\[ \tiny
\xymatrix@C+0pc@R+6pc{
 &(C\widehat{\otimes}B)\otimes\mathcal{K}(\mathrm{Spa}(A,A^+))   \ar[r]\ar[r]\ar[r] \ar[d]\ar[d]\ar[d] &(C\widehat{\otimes}B)\otimes\mathcal{F}(\mathrm{Spa}(A,A^+)) \ar[r]\ar[r]\ar[r] \ar[d]\ar[d]\ar[d] & (C\widehat{\otimes}B)\otimes\mathcal{M}(\mathrm{Spa}(A,A^+))\ar[r]\ar[r]\ar[r]  \ar[d]\ar[d]\ar[d]&0,\\
0 \ar[r]\ar[r]\ar[r] &\mathcal{K}(\mathrm{Spa}(C,C^+))   \ar[r]\ar[r]\ar[r] &\mathcal{F}(\mathrm{Spa}(C,C^+)) \ar[r]\ar[r]\ar[r] & \mathcal{M}(\mathrm{Spa}(C,C^+)) \ar[r]\ar[r]\ar[r] &0,\\
}
\]
$A\rightarrow C$ is some member in $\mathfrak{C}$. Here the first row is exact by $B$-\'etale-pseudoflatness and the second is row is exact as well. The middle vertical arrow is an isomorphism while the rightmost vertical arrow is surjective as in \cite[Lemma 2.5.13]{KL2}. Then we apply the same argument to the leftmost vertical arrow (with a new diagram), we will have the corresponding surjectivity of this leftmost vertical arrow. This will imply that the rightmost vertical arrow is injective as well. Then we have the corresponding isomorphism for the rightmost vertical arrow.
\end{proof}

\begin{theorem}\mbox{\bf{(After Kedlaya-Liu \cite[Theorem 2.5.14]{KL2})}} \label{theorem2.22}
Taking global section will realize the equivalence between the following two categories: A. The category of all the pseudocoherent $\mathcal{O}_{\mathrm{Spa}(A,A^+)_{\text{\'et}}}\widehat{\otimes}B$-sheaves; B. The category of all the $B$-\'etale-stably pseudocoherent modules over $A\widehat{\otimes}B$. 	
\end{theorem}

\begin{proof}
See \cite[Theorem 2.5.14]{KL2}. We need to still apply \cite[Theorem 8.2.22]{KL1}, as long as one considers instead in our situation, and \cite[Tag 03OD]{SP}.\\	
\end{proof}

\subsection{Noncommutative Deformation over Quasi-Stein Spaces}

\indent In this subsection we consider the corresponding noncommutative deformation over the corresponding context in \cite[Chapter 2.6]{KL2}.

\begin{setting}
Let $X$ be a corresponding quasi-Stein adic space over $\mathbb{Q}_p$ or $\mathbb{F}_p((t))$ in the sense of \cite[Definition 2.6.2]{KL2}. Recall that what is happening is that $X$ could be written as the corresponding direct limit of affinoids $X:=\varinjlim_i X_i$.\\ 	
\end{setting}

\begin{lemma} \mbox{\bf{(After Kedlaya-Liu \cite[Lemma 2.6.3]{KL2})}} \label{lemma2.24}
We now consider the corresponding rings $A_i:=\mathcal{O}_{X_i}$ for all $i=0,1,...$, and in our current situation we consider the corresponding rings $A_i\widehat{\otimes}B$ (over $\mathbb{Q}_p$ or $\mathbb{F}_p((t))$) for all $i=0,1,...$. And in our situation we consider the corresponding modules $M_i$ over $A_i\widehat{\otimes}B$ for all $i=0,1,...$ with the same requirement as in \cite[Lemma 2.6.3]{KL2} (namely those complete with respect to the natural topology). Suppose that we have bounded surjective map from $f_i:A_{i}\widehat{\otimes}B\otimes_{A_{i+1}\widehat{\otimes}B} M_{i+1}\rightarrow M_i,i=0,1,...$. Then we have first the density of the corresponding image of $\varprojlim_i M_i$ in each $M_i$ for any $i=0,1,2,...$. And we have as well the corresponding vanishing of $R^1\varprojlim_i M_i$.
\end{lemma}

\begin{proof}
This is the $B$-relative version of the result in \cite[Lemma 2.6.3]{KL2}. For the first statement we just choose sequence of Banach norms on all the corresponding modules for all $i=0,1,...$ such that we have $\|f_i(x_{i+1})\|_i\leq 1/2\|x_{i+1}\|_{i+1}$ for any $x_{i+1}\in M_{i+1}$. Then for any $x_i\in M_i$ and any $\delta>0$, we now consider for any $j\geq 1$ the corresponding $x_{i+j}$ such that we have $\|x_{i+j}-f_{i+j+1}(x_{i+j+1})\|_{i+j}\leq \delta$. Then the sequence $x_{i+j+k},k=0,1,...$ will converge to some well-defined $y_{i+j}$ with in our situation the corresponding $y_{i}=f_{i}(y_{i+1})$. We then have $\|x_i-y_i\|_i\leq \delta$. This will prove the first statement. For the second statement as in \cite[Lemma 2.6.3]{KL2} we form the product $M_0\times M_1\times M_2\times...$ and the consider the induced map $F$ from $M_{i+1}\rightarrow M_i$, and consider the corresponding cokernel of the map $1-F$ since this is just the corresponding limit we are considering. Then to show that the cokernel is zero we just look at the corresponding cokernel of the corresponding map on the corresponding completed direct summand which will project to the original one. But then we will have $\|f_i(v)\|_i\leq 1/2 \|v\|_i$, which produces an inverse to $1-F$ which will basically finish the proof for the second statement. 	
\end{proof}

\begin{proposition} \mbox{\bf{(After Kedlaya-Liu \cite[Lemma 2.6.4]{KL2})}} In the same situation as above, suppose we have that the corresponding modules $M_i$ are basically $B$-stably pseudocoherent over the rings $A_i\widehat{\otimes}B$ for all $i=0,1,...$. Now we consider the situation where $f_i:A_i\widehat{\otimes}B\otimes_{A_{i+1}\widehat{\otimes}B}M_{i+1}\rightarrow M_i$ is an isomorphism. Then the conclusion in our situation is then that the corresponding projection from $\varprojlim M_i$ to $M_i$ for each $i=0,1,2,...$ is an isomorphism.

\end{proposition}

\begin{proof}
This is a $B$-relative version of the \cite[Lemma 2.6.4]{KL2}. We adapt the argument to our situation as in the following. First we choose some finite free covering $T$ of the limit $M$ such that we have for each $i$ the corresponding map $T_i\rightarrow M_i$ is surjective. Then we consider the index $j\geq i$ and set the kernel of the map from $T_j$ to $M_j$ to be $S_j$. By the direct analog of \cite[Lemma 2.5.6]{KL2} we have that $A_i\widehat{\otimes}B\otimes S_j\overset{}{\rightarrow}S_i$ realizes the isomorphism, and we have that the corresponding surjectivity of the corresponding map from $\varprojlim_i S_i$ projecting to the $S_i$. Then one could finish the proof by 5-lemma to the following commutative diagram as in \cite[Lemma 2.6.4]{KL2}:
\[ \tiny
\xymatrix@C+0pc@R+6pc{
 &(A_i\widehat{\otimes}B)\otimes \varprojlim_i S_i \ar[r]\ar[r]\ar[r] \ar[d]\ar[d]\ar[d] &(A_i\widehat{\otimes}B)\otimes \varprojlim_i F_i \ar[r]\ar[r]\ar[r] \ar[d]\ar[d]\ar[d] &(A_i\widehat{\otimes}B)\otimes \varprojlim_i M_i\ar[r]\ar[r]\ar[r]  \ar[d]\ar[d]\ar[d]&0,\\
0 \ar[r]\ar[r]\ar[r] &S_i   \ar[r]\ar[r]\ar[r] &F_i \ar[r]\ar[r]\ar[r] &M_i \ar[r]\ar[r]\ar[r] &0.\\
}
\]
 
\end{proof}

\begin{proposition}  \mbox{\bf{(After Kedlaya-Liu \cite[Theorem 2.6.5]{KL2})}}
For any quasi-compact adic affinoid space of $X$ which is denoted by $Y$, we have that the map $\mathcal{M}(X)\rightarrow \mathcal{M}(Y)$ is surjective for any $B$-stably pseudocoherent sheaf $\mathcal{M}$ over the sheaf $\mathcal{O}_X\widehat{\otimes}B$.	
\end{proposition}

\begin{proof}
This is just the corresponding corollary of the previous proposition.	
\end{proof}

\begin{proposition}  \mbox{\bf{(After Kedlaya-Liu \cite[Theorem 2.6.5]{KL2})}}
We have that the stalk $\mathcal{M}_x$ is generated over the stalk $\mathcal{O}_{X,x}$ for any $x\in X$ by $M(X)$, for any $B$-stably pseudocoherent sheaf $\mathcal{M}$ over the sheaf $\mathcal{O}_X\widehat{\otimes}B$.	
\end{proposition}

\begin{proof}
This is just the corresponding corollary of the proposition before the previous proposition.	
\end{proof}

\begin{proposition}  \mbox{\bf{(After Kedlaya-Liu \cite[Theorem 2.6.5]{KL2})}} \label{proposition2.28}
For any quasi-compact adic affinoid space of $X$ which is denoted by $Y$, we have that the corresponding vanishing of the corresponding sheaf cohomology groups $H^k(X,\mathcal{M})$ of $\mathcal{M}$ for higher $k>0$, for any $B$-stably pseudocoherent sheaf $\mathcal{M}$ over the sheaf $\mathcal{O}_X\widehat{\otimes}B$.	
\end{proposition}

\begin{proof}
We follow the idea of the proof of \cite[Theorem 2.6.5]{KL2} by comparing this to the corresponding \v{C}ech cohomology with some covering $\mathfrak{X}=\{X_1,...,X_N,...\}$:
\begin{align}
\breve{H}^k(X,\mathfrak{X}=\{X_1,...,X_N,...\};\mathcal{M})=H^k(X,\mathcal{M}),\\
\breve{H}^k(X_i,\mathfrak{X}=\{X_1,...,X_i\};\mathcal{M})=H^k(X_i,\mathcal{M})=0,k\geq 1.
\end{align}
Here we have applied the corresponding \cite[Tag 01EW]{SP}. Now we consider the situation where $k>1$:
\begin{align}
\breve{H}^{k-1}(X_{j+1},\mathfrak{X}=\{X_1,...,X_{j+1}\};\mathcal{M})	\rightarrow \breve{H}^{k-1}(X_j,\mathfrak{X}=\{X_1,...,X_{j}\};\mathcal{M})\rightarrow 0,
\end{align}
which induces the following isomorphism by \cite[2.6 Hilfssatz]{Kie1}:
\begin{align}
\varprojlim_{j\rightarrow \infty} \breve{H}^{k-1}(X_j,\mathfrak{X}=\{X_1,...,X_{j}\};\mathcal{M})	\overset{\sim}{\rightarrow} \breve{H}^k(X,\mathfrak{X}=\{X_1,...,X_N,...\};\mathcal{M}). 
\end{align}
Then we have the corresponding results for the index $k>1$. For $k=1$, one can as in \cite[Theorem 2.6.5]{KL2} relate this to the corresponding $R^1\varprojlim_i$, which will finishes the proof in the same fashion.\\

\end{proof}

\begin{corollary}  \mbox{\bf{(After Kedlaya-Liu \cite[Corollary 2.6.6]{KL2})}}
The corresponding functor from the corresponding $B$-deformed pseudocoherent sheaves over $X$ to the corresponding $B$-stably by taking the corresponding global section is an exact functor.
\end{corollary}

\begin{corollary}  \mbox{\bf{(After Kedlaya-Liu \cite[Corollary 2.6.8]{KL2})}}
Consider a particular $\mathcal{O}_X\widehat{\otimes}B$-pseudocoherent sheaf $\mathcal{M}$ which is finite locally free throughout the whole space $X$. Then we have that the global section $\mathcal{M}(X)$ as $\mathcal{O}_X(X)\widehat{\otimes}B$ left module admits the corresponding structures of finite projective structure if and only if we have the corresponding global section is finitely generated.
\end{corollary}

\begin{proof}
As in \cite[Corollary 2.6.8]{KL2} one could find some global splitting through the local splittings.
\end{proof}

\indent We now consider the following $B$-relative analog of \cite[Proposition 2.6.17]{KL2}:

\begin{theorem} \mbox{\bf{(After Kedlaya-Liu \cite[Proposition 2.6.17]{KL2})}} \label{theorem2.31}
Consider the following two statements for a particular $\mathcal{O}_X\widehat{\otimes}B$-pseudocoherent sheaf $\mathcal{M}$. First is that one can find finite many generators (the number is up to some uniform fixed integer $n\geq 0$) for each section of $\mathcal{M}(X_i)$ for each $i=1,2,...$. The second statement is that the global section $\mathcal{M}(X)$ is just finitely generated. Then in our situation the two statement is equivalent if we have that the corresponding space $X$ admits an $m$-uniform covering in the exact same sense of \cite[Proposition 2.6.17]{KL2}. 	
\end{theorem}

\begin{proof}
One direction is obvious, the other direct could be proved in the same way as in \cite[Proposition 2.6.17]{KL2} where essentially the information on $X$ does not change at all. Namely as in \cite[Proposition 2.6.17]{KL2} we could basically consider one single subcovering $\{Y_u\}_{u\in U}$ indexed by $U$ of the covering which is $m$-uniform. For any $u\in U$ we consider the smallest $i$ such that $X_i\bigcap Y_u$, and we denote this by $i(u)\geq 0$. Then we form:
\begin{align}
V_{u}:=X_{i(u)}\bigcup_{v\in U,v\leq u} Y_{v}.	
\end{align}
Then as in \cite[Proposition 2.6.17]{KL2} we can find some $x_u,u\in U$ (which are denoted by the general form of $x_i$ in \cite[Proposition 2.6.17]{KL2}) in the global section of $\mathcal{O}_X(X)$ such that it restricts to some topologically nilpotent onto the space $V_u$ and it restricts to some inverse of some topologically nilpotent to $Y_u$. Then we can build up a corresponding generating set for the global section by the following approximation process just as in \cite[Proposition 2.6.17]{KL2}. To be more precise what we have consider is to modify the corresponding generator for each $Y_u$ (for instance let them be denoted by $y_1,...,y_u$) to be $y_{u,1},...,y_{u,n}$ ($n$ is the corresponding uniform integer in the statement of the theorem) for all $u\in U$. It is achieved through induction, when we have $y_{u,1},...,y_{u,n}$ then we will set that to be $0$ otherwise in the situation where there exists some predecessor $u'$ of $u$ in the corresponding set $U$ we then set this to be $y_{u',1},...,y_{u',n}$. And we set as in \cite[Proposition 2.6.17]{KL2}:
\begin{displaymath}
y_{u,j}:=	y_{u-1,j}+x_u^cy_u
\end{displaymath}
by lifting the corresponding power $c$ to be as large as possible. This will guarantee the convergence of:
\begin{align}
\lim_{u\rightarrow \infty}\{y_{u,1},...,y_{u,n}\},	
\end{align}
which gives the set of the corresponding global generators desired.
\end{proof}

\newpage

\section{Foundations on Noncommutative Descent for Adic Spectra in Pro-Banach Case}

\subsection{Noncommutative Pseudocoherence in Analytic Topology}

\indent We now establish some foundations in the noncommutative setting on glueing noncommutative pseudocoherent modules after \cite[Chapter 2]{KL2}. However the corresponding difference we will make here from the previous section is that we will consider some further Iwasawa deformation similar to the sense of Fukaya-Kato's adic rings. We first consider the simpler situation where we will deform over $B[[G]]$ where $B$ as before comes from the following beginning setting up and $G$ is a given $p$-adic Lie group:

\begin{setting} \label{setting3.1}
Consider a corresponding sheafy Banach adic uniform algebra $(A,A^+)$ over $\mathbb{Q}_p$, we consider the base space $\mathrm{Spa}(A,A^+)$. And now we will consider a further noncommutative Banach algebra $(B,B^+)$ over $\mathbb{Q}_p$. 	
\end{setting}

\begin{setting}
In what follow when we mean a module $M$ over $A\widehat{\otimes}B[[G]]$ we will mean a pro-system $\{M_E\}_{E\subset G}$ of left module over $A\widehat{\otimes}B[[G]]$. And the corresponding definition of pseudocoherence is also defined on the adic system $\{M_E\}_{E\subset G}$ namely for each member in the system. We abuse the notation a bit when we write the notation $[[G]]/E$ which will really mean that $\mathbb{Q}_p[[G/E]]$. Here we have considered a corresponding representation of the algebra $B[[G]]$ by $\varprojlim_{E} B[[G/E]]$.	
\end{setting}

\begin{definition} \mbox{\bf{(After Kedlaya-Liu \cite[Definition 2.4.1]{KL2})}}
For any left $A\widehat{\otimes}B[[G]]$-module $M$, we call it $m$-$B[[G]]$-stably pseudocoherent if we have that it is $m$-$B[[G]]$-pseudocoherent, complete for the natural topology and for any morphism $A\rightarrow A'$ which is the corresponding rational localization, the base change of $M$ to $A'\widehat{\otimes} B[[G]]$ is complete for the natural topology as the corresponding left $A'\widehat{\otimes} B[[G]]$-module (namely each member in the family is still remained to be complete after the corresponding rational localization). As in \cite[Definition 2.4.1]{KL2} we call that the corresponding left $A\widehat{\otimes}B[[G]]$-module $M$ just $B[[G]]$-stably pseudocoherent if we have that it is simply just $\infty$-$B[[G]]$-stably pseudocoherent.
\end{definition}

\indent Since our spaces are commutative, so the rational localization is not flat, therefore in our situation we need to define something weaker than the corresponding flatness which will be sufficient for us to apply in the following development: 

\begin{definition}\mbox{\bf{(After Kedlaya-Liu \cite[Definition 2.4.4]{KL2})}}
For any Banach left $A\widehat{\otimes}B[[G]]$-module $M$, we call it is $m$-$B[[G]]$-pseudoflat if for any right $A\widehat{\otimes}B[[G]]$-module $M'$ $m$-$B[[G]]$-stably pseudocoherent we have $\mathrm{Tor}_1^{A\widehat{\otimes}B[[G]]}(M',M)=0$ (namely $\{\mathrm{Tor}_1^{A\widehat{\otimes}B[[G]]/E}(M'_E,M_E)\}_{E\subset G}=\{0\}_{E\subset G}$). For any Banach right $A\widehat{\otimes}B[[G]]$-module $M$, we call it is $m$-$B[[G]]$-pseudoflat if for any left $A\widehat{\otimes}B[[G]]$-module $M'$ $m$-$B[[G]]$-stably pseudocoherent we have $\mathrm{Tor}_1^{A\widehat{\otimes}B[[G]]}(M,M')=0$.
\end{definition}

\begin{definition}\mbox{\bf{(After Kedlaya-Liu \cite[Definition 2.4.6]{KL2})}} We consider the corresponding notion of the corresponding pro-projective modules. We define over $A$ the corresponding $B[[G]]$-pro-projective module $M$ to be a corresponding left $A\widehat{\otimes}B[[G]]$-module such that one could find a filtered sequence of projectors such that in the sense of taking the prolimit we have the corresponding projector will converge to any chosen element in $M$. Here we assume the module is complete with respect to natural topology, and we assume the projectors are $A\widehat{\otimes}B[[G]]$-linear and we assume that the corresponding image of the projectors are modules which are also finitely generated and projective.
	
\end{definition}

\begin{lemma}\mbox{\bf{(After Kedlaya-Liu \cite[Lemma 2.4.7]{KL2})}}
Suppose we have over the space $\mathrm{Spa}(A,A^+)$ a corresponding $B[[G]]$-pro-projective left module $M$. And suppose that we have a $2$-$B[[G]]$-pseudoflat right module $C$ over $A$. And here we assume that $C$ is complete. Then we have that the corresponding product $C{\otimes}_{A\widehat{\otimes}B[[G]]}M$ is then complete under the corresponding natural topology. And moreover we have that in our situation:
\begin{displaymath}
\mathrm{Tor}_1^{A\widehat{\otimes}B[[G]]}(C,M)=0.	
\end{displaymath}

\end{lemma}

\begin{proof}
This will be just a noncommutative version of the corresponding  \cite[Lemma 2.4.7]{KL2}. We adapt the corresponding argument in \cite[Lemma 2.4.7]{KL2} to our situation. We just consider each member in the families for the adic modules. What we are going to consider in this case is then first to consider the following presentation:
\[
\xymatrix@C+0pc@R+0pc{
M'  \ar[r]\ar[r]\ar[r] &(A\widehat{\otimes}B[[G]]/E)^k \ar[r]\ar[r]\ar[r] & M\ar[r]\ar[r]\ar[r] &0,
}
\]
where the left module $M'$ is finitely presented. Then we consider the following commutative diagram:
\[
\xymatrix@C+0pc@R+3pc{
&C\otimes_{A\widehat{\otimes}B[[G]]/E}M'  \ar[r]\ar[r]\ar[r]  \ar[d]\ar[d]\ar[d] &C\otimes_{A\widehat{\otimes}B[[G]]/E}(A\widehat{\otimes}B[[G]]/E)^k \ar[r]\ar[r]\ar[r] \ar[d]\ar[d]\ar[d] & C\otimes_{A\widehat{\otimes}B[[G]]/E}M\ar[r]\ar[r]\ar[r] \ar[d]\ar[d]\ar[d] &0,\\
0 \ar[r]\ar[r]\ar[r] &C\widehat{\otimes}_{A\widehat{\otimes}B[[G]]/E}M'  \ar[r]\ar[r]\ar[r] &C\widehat{\otimes}_{A\widehat{\otimes}B[[G]]/E}(A\widehat{\otimes}B[[G]]/E)^k \ar[r]\ar[r]\ar[r] & C\widehat{\otimes}_{A\widehat{\otimes}B[[G]]/E}M\ar[r]\ar[r]\ar[r] &0.
}
\]\\
The first row is exact as in \cite[Lemma 2.4.7]{KL2}, and we have the corresponding second row is also exact since we have that by hypothesis the corresponding module $M$ is $[[G]]/E$-pro-projective.  
Then as in \cite[Lemma 2.4.7]{KL2} the results follow.	
\end{proof}

\begin{proposition} \mbox{\bf{(After Kedlaya-Liu \cite[Corollary 2.4.8]{KL2})}}
Over the space $\mathrm{Spa}(A,A^+)$, we have any $B[[G]]$-pro-projective left module is $2$-$B[[G]]$-pseudoflat.
\end{proposition}

\begin{proof}
This is a direct consequence of the previous lemma.	
\end{proof}

\begin{lemma} \mbox{\bf{(After Kedlaya-Liu \cite[Corollary 2.4.9]{KL2})}}
Keep the notation above, suppose we are working over the adic space $
\mathrm{Spa}(A,A^+)$. Suppose now the left $A\widehat{\otimes}B[[G]]$-module is finitely generated (again note that we are talking about a family of objects over the corresponding quotients), then we have that the following natural maps:
\begin{align}
A\{T\}\widehat{\otimes}B[[G]]\otimes_{A\widehat{\otimes}B[[G]]}M 	\rightarrow M\{T\},\\\
A\{T,T^{-1}\}\widehat{\otimes}B[[G]] \otimes_{A\widehat{\otimes}B[[G]]} M	\rightarrow M\{T,T^{-1}\}	
\end{align}
are surjective. Suppose now the left $A\widehat{\otimes}B$-module is finitely presented, then we have that the following natural maps:
\begin{align}
A\{T\}\widehat{\otimes}B[[G]]\otimes_{A\widehat{\otimes}B[[G]]}M 	\rightarrow M\{T\},\\\
A\{T,T^{-1}\}\widehat{\otimes}B[[G]] \otimes_{A\widehat{\otimes}B[[G]]} M	\rightarrow M\{T,T^{-1}\}	
\end{align}
are bijective. All modules here are assumed to be complete.

\end{lemma}

\begin{proof}
This is the corresponding consequence of the previous lemma.
\end{proof}

\begin{proposition} \mbox{\bf{(After Kedlaya-Liu \cite[Lemma 2.4.12]{KL2})}}
We keep the corresponding notations in \cite[Lemma 2.4.10]{KL2}. Then in our current sense we have that the corresponding morphism $A\rightarrow B_2$ is then $2$-$B[[G]]$-pseudoflat. 
\end{proposition}

\begin{proof}
The corresponding proof could be made parallel to the corresponding proof of \cite[Lemma 2.4.12]{KL2}. To be more precise we consider the following commutative diagram from the corresponding chosen exact sequence (which is just the analog of the corresponding one in \cite[Lemma 2.4.12]{KL2}, namely as below we choose arbitrary short exact sequence with $P$ as a left $A\widehat{\otimes}B[[G]]/E$-module which is $2$-$B[[G]]/E$-stably pseudocoherent):
\[
\xymatrix@C+0pc@R+0pc{
0   \ar[r]\ar[r]\ar[r] &M \ar[r]\ar[r]\ar[r] &N \ar[r]\ar[r]\ar[r] & P \ar[r]\ar[r]\ar[r] &0,
}
\]
which then induces the following corresponding big commutative diagram:
\[\tiny
\xymatrix@C+0pc@R+6pc{
& &0 \ar[d]\ar[d]\ar[d] &0 \ar[d]\ar[d]\ar[d] &,\\
0   \ar[r]\ar[r]\ar[r]  &{A\{T\}\widehat{\otimes}B[[G]]/E}\otimes_{A\widehat{\otimes}B[[G]]/E} M \ar[r]\ar[r]\ar[r] \ar[d]^{1-fT}\ar[d]\ar[d] &{A\{T\}\widehat{\otimes}B[[G]]/E}\otimes_{A\widehat{\otimes}B[[G]]/E} N \ar[r]\ar[r]\ar[r] \ar[d]^{1-fT}\ar[d]\ar[d] &{A\{T\}\widehat{\otimes}B[[G]]/E} \otimes_{A\widehat{\otimes}B[[G]]/E}P \ar[r]\ar[r]\ar[r] \ar[d]^{1-fT}\ar[d]\ar[d] &0,\\
0   \ar[r]\ar[r]\ar[r] &{A\{T\}\widehat{\otimes}B[[G]]/E}\otimes_{A\widehat{\otimes}B[[G]]/E}M  \ar[r]\ar[r]\ar[r] \ar[d]\ar[d]\ar[d] &{A\{T\}\widehat{\otimes}B[[G]]/E} \otimes_{A\widehat{\otimes}B[[G]]/E}N\ar[r]\ar[r]\ar[r] \ar[d]\ar[d]\ar[d] & {A\{T\}\widehat{\otimes}B[[G]]/E}\otimes_{A\widehat{\otimes}B[[G]]/E}P \ar[r]\ar[r]\ar[r] \ar[d]\ar[d]\ar[d] &0,\\
0  \ar[r]^?\ar[r]\ar[r] &{B_2\widehat{\otimes}B[[G]]/E}\otimes_{A\widehat{\otimes}B[[G]]/E} M \ar[r]\ar[r]\ar[r] \ar[d]\ar[d]\ar[d] &{B_2\widehat{\otimes}B[[G]]/E}\otimes_{A\widehat{\otimes}B[[G]]/E}N \ar[r]\ar[r]\ar[r] \ar[d]\ar[d]\ar[d] & {B_2\widehat{\otimes}B[[G]]/E}\otimes_{A\widehat{\otimes}B[[G]]/E} P \ar[r]\ar[r]\ar[r] \ar[d]\ar[d]\ar[d] &0,\\
&0&0&0
}
\]
with the notations in \cite[Lemma 2.4.10]{KL2} where we actually have the corresponding exactness along the horizontal direction at the corner around $M\otimes_{A\widehat{\otimes}B[[G]]/E} {B_2\widehat{\otimes}B[[G]]/E}$ marked with $?$, by diagram chasing.
\end{proof}

\begin{proposition} \mbox{\bf{(After Kedlaya-Liu \cite[Lemma 2.4.13]{KL2})}} \label{proposition3.10}
We keep the corresponding notations in \cite[Lemma 2.4.10]{KL2}. Then in our current sense we have that the corresponding morphism $A\rightarrow B_1$ is then $2$-$B[[G]]$-pseudoflat. And in our current sense we have that the corresponding morphism $A\rightarrow B_{1,2}$ is then $2$-$B[[G]]$-pseudoflat. 
\end{proposition}

\begin{proof}
The statement for $A\rightarrow B_{1,2}$ could be proved by consider the composition $A\rightarrow B_2\rightarrow B_{12}$ where the $2$-$B[[G]]$-pseudoflatness could be proved as in \cite[Lemma 2.4.13]{KL2} by inverting the corresponding variable. For $A\rightarrow B_{1}$,
the corresponding proof could be made parallel to the corresponding proof of \cite[Lemma 2.4.13]{KL2}. To be more precise we consider the following commutative diagram from the corresponding chosen exact sequence (which is just the analog of the corresponding one in \cite[Remark 2.4.5]{KL2}, namely as below we choose arbitrary short exact sequence with $P$ as a left $A\widehat{\otimes}B[[G]]/E$-module which is $2$-$B[[G]]/E$-stably pseudocoherent):
\[
\xymatrix@C+0pc@R+0pc{
0   \ar[r]\ar[r]\ar[r] &M \ar[r]\ar[r]\ar[r] &N \ar[r]\ar[r]\ar[r] & P \ar[r]\ar[r]\ar[r] &0,
}
\]
which then induces the following corresponding big commutative diagram:
\[\tiny
\xymatrix@C+0pc@R+6pc{
& &0 \ar[d]\ar[d]\ar[d] &0 \ar[d]\ar[d]\ar[d] &,\\
0   \ar[r]\ar[r]\ar[r]  &M \ar[r]\ar[r]\ar[r] \ar[d]\ar[d]\ar[d] &N \ar[r]\ar[r]\ar[r] \ar[d]\ar[d]\ar[d] & P \ar[r]\ar[r]\ar[r] \ar[d]\ar[d]\ar[d] &0,\\
0   \ar[r]^?\ar[r]\ar[r] &({B_1\widehat{\otimes}\frac{B[[G]]}{E}}\bigoplus {B_2\widehat{\otimes}\frac{B[[G]]}{E}}) \otimes M  \ar[r]\ar[r]\ar[r] \ar[d]\ar[d]\ar[d] &({B_1\widehat{\otimes}\frac{B[[G]]}{E}}\bigoplus {B_2\widehat{\otimes}\frac{B[[G]]}{E}})\otimes N  \ar[r]\ar[r]\ar[r] \ar[d]\ar[d]\ar[d] &({B_1\widehat{\otimes}\frac{B[[G]]}{E}}\bigoplus {B_2\widehat{\otimes}\frac{B[[G]]}{E}})\otimes P\ar[r]\ar[r]\ar[r] \ar[d]\ar[d]\ar[d] &0,\\
0  \ar[r]\ar[r]\ar[r] &{B_{12}\widehat{\otimes}\frac{B[[G]]}{E}}\otimes M \ar[r]\ar[r]\ar[r] \ar[d]\ar[d]\ar[d] &{B_{12}\widehat{\otimes}\frac{B[[G]]}{E}}\otimes N \ar[r]\ar[r]\ar[r] \ar[d]\ar[d]\ar[d] &{B_{12}\widehat{\otimes}\frac{B[[G]]}{E}} \otimes P \ar[r]\ar[r]\ar[r] \ar[d]\ar[d]\ar[d] &0,\\
&0&0&0
}
\]
with the notations in \cite[Lemma 2.4.10]{KL2} where we actually have the corresponding exactness along the horizontal direction at the corner around $M\otimes_{A\widehat{\otimes}B[[G]]/E} ({B_1\widehat{\otimes}B[[G]]/E}\bigoplus {B_2\widehat{\otimes}B[[G]]/E})$ marked with $?$, by diagram chasing.
\end{proof}

\indent After these foundational results as in \cite{KL2} we have the following proposition which is the corresponding noncommutative generalization of the corresponding result established in \cite[Theorem 2.4.15]{KL2}.

\begin{proposition} \mbox{\bf{(After Kedlaya-Liu \cite[Theorem 2.4.15]{KL2})}} \label{proposition3.11}
In our current context we have that for any rational localization $\mathrm{Spa}(A',A^{',+})\rightarrow \mathrm{Spa}(A,A^+)$ we have that along this base change the corresponding $B[[G]]$-stably pseudocoherence is preserved.	
\end{proposition}

\indent Then we have the corresponding Tate's acyclicity in the noncommutative deformed setting:

\begin{theorem}\mbox{\bf{(After Kedlaya-Liu \cite[Theorem 2.5.1]{KL2})}} \label{theorem3.12} Now suppose we have in our corresponding \cref{setting2.1} a corresponding $B[[G]]$-stably pseudocoherent module $M$. Then we consider the corresponding assignment such that for any $U\subset \mathrm{Spa}(A,A^+)$ we define $\widetilde{M}(U)$ as in the following:
\begin{align}
\widetilde{M}(U):=\varprojlim_{\mathrm{Spa}(S,S^+)\subset U,\mathrm{rational}} S\widehat{\otimes}B[[G]]\otimes_{A\widehat{\otimes}B[[G]]}M.	
\end{align}
Then we have that for any $\mathfrak{B}$ which is a rational covering of $U=\mathrm{Spa}(S,S^+)\subset \mathrm{Spa}(A,A^+)$ (certainly this $U$ is also assumed to be rational) we have that the vanishing of the following two cohomology groups:
\begin{align}
H^i(U,\widetilde{M}), \check{H}^i(U:\mathfrak{B},\widetilde{M})
\end{align}
for any $i>0$. When concentrating at the degree zero we have:
\begin{align}
H^0(U,\widetilde{M})=S\widehat{\otimes}B[[G]]\otimes_{A\widehat{\otimes}B[[G]]}M, \check{H}^0(U:\mathfrak{B},\widetilde{M})=S\widehat{\otimes}B[[G]]\otimes_{A\widehat{\otimes}B[[G]]}M.
\end{align}
	
\end{theorem}

\begin{proof}
By \cite[Propositions 2.4.20-2.4.21]{KL1}, we could then finish proof as in \cite[Theorem 2.5.1]{KL2} by using our previous \cref{proposition3.11} and \cref{proposition3.10} as above.	
\end{proof}

\indent We now consider the corresponding noncommutative deformed version of Kiehl's glueing properties for stably pseudocoherent modules after \cite{KL2}:

\begin{definition} \mbox{\bf{(After Kedlaya-Liu \cite[Definition 2.5.3]{KL2})}}
Consider in our context (over $\mathrm{Spa}(A,A^+)$) the corresponding sheaves $\mathcal{O}\widehat{\otimes}B[[G]]$, we will then define the corresponding pseudocoherent sheaves over $\mathrm{Spa}(A,A^+)$ to be those locally defined by attaching stably-pseudocoherent modules over the section. 
\end{definition}

\begin{lemma} \mbox{\bf{(After Kedlaya-Liu \cite[Lemma 2.5.4]{KL2})}}\label{lemma3.14}
	Consider the corresponding notations in \cite[Lemma 2.4.10]{KL2}, we have the corresponding morphism $A\rightarrow B_1\bigoplus B_2$. Then we have that this morphism is an descent morphism effective for the corresponding $B[[G]]$-stably pseodocoherent Banach modules. 
\end{lemma}

\begin{proof}
This is a $B[[G]]$-relative version of the \cite[Lemma 2.5.4]{KL2}. We adapt the corresponding proof to our situation. First consider the corresponding notations on the commutative adic space $\mathrm{Spa}(A,A^+)$ as in \cite[Lemma 2.4.10]{KL2}. And we have a coving $\{U_1,U_2\}$. Then consider a descent datum for this cover of $B[[G]]$-stably pseudocoherent sheaves. By \cref{theorem3.12} we could realize this as a corresponding pseudocoherent sheaf over $\mathcal{O}\widehat{\otimes}B[[G]]$. We denote this by $\mathcal{M}$. Then by (the proof of) \cite[Lemma 6.82]{T2}, we have that this is finitely generated and we have the corresponding vanishing of $\check{H}^i(\mathrm{Spa}(A,A^+),\mathfrak{B};\mathcal{M})$ for $i>0$. Then by \cref{theorem3.12} again we have that the vanishing of $H^i(U_j;\mathcal{M})$ for $j\in \{1,2\},i>0$. Then as in \cite[Lemma 2.5.4]{KL2} we have the corresponding vanishing of $H^i(\mathrm{Spa}(A,A^+);\mathcal{M})$ for $i>0$. Then as in \cite[Lemma 2.5.4]{KL2} we choose a corresponding covering of some finite free object $\mathcal{F}$ and take its kernel $\mathcal{K}$ which is actually pseudocoherent by the corresponding \cref{proposition3.11}. Then forming the following diagram:
\[\tiny
\xymatrix@C+0pc@R+6pc{
0 \ar[r]\ar[r]\ar[r] & (B_p\widehat{\otimes}B[[G]]/E)\otimes \mathcal{K}(\mathrm{Spa}(A,A^+))  \ar[r]\ar[r]\ar[r] \ar[d]\ar[d]\ar[d] &(B_p\widehat{\otimes}B[[G]]/E)\otimes \mathcal{F}(\mathrm{Spa}(A,A^+))\ar[r]\ar[r]\ar[r] \ar[d]\ar[d]\ar[d] & (B_p\widehat{\otimes}B[[G]]/E)\otimes \mathcal{M}(\mathrm{Spa}(A,A^+))\ar[r]\ar[r]\ar[r]  \ar[d]\ar[d]\ar[d]&0,\\
0 \ar[r]\ar[r]\ar[r] &\mathcal{K}(U_p)   \ar[r]\ar[r]\ar[r] &\mathcal{F}(U_p) \ar[r]\ar[r]\ar[r] & \mathcal{M}(U_p) \ar[r]\ar[r]\ar[r] &0,\\
}
\]
where we have $p=1,2$. The first row is exact by $B[[G]]/E$-pseudoflatness as in \cref{proposition2.10} and the second is row is exact as well by \cref{theorem2.11}. The middle vertical arrow is an isomorphism while the rightmost vertical arrow is surjective and the leftmost vertical arrow is surjective. The surjectivity just here comes from actually \cite[Lemma 6.82]{T2}. This will imply that the rightmost vertical arrow is injective as well. Then we have the corresponding isomorphism $\mathcal{M}\overset{\sim}{\rightarrow} \widetilde{\mathcal{M}(\mathrm{Spa}(A,A^+))}$. Then we consider the following diagram:

\[ \tiny
\xymatrix@C+0pc@R+6pc{
 &(C\widehat{\otimes}B[[G]]/E)\otimes\mathcal{K}(\mathrm{Spa}(A,A^+))   \ar[r]\ar[r]\ar[r] \ar[d]\ar[d]\ar[d] &(C\widehat{\otimes}B[[G]]/E)\otimes\mathcal{F}(\mathrm{Spa}(A,A^+)) \ar[r]\ar[r]\ar[r] \ar[d]\ar[d]\ar[d] & (C\widehat{\otimes}B[[G]]/E)\otimes\mathcal{M}(\mathrm{Spa}(A,A^+))\ar[r]\ar[r]\ar[r]  \ar[d]\ar[d]\ar[d]&0,\\
0 \ar[r]\ar[r]\ar[r] &\mathcal{K}(\mathrm{Spa}(C,C^+))   \ar[r]\ar[r]\ar[r] &\mathcal{F}(\mathrm{Spa}(C,C^+)) \ar[r]\ar[r]\ar[r] & \mathcal{M}(\mathrm{Spa}(C,C^+)) \ar[r]\ar[r]\ar[r] &0,\\
}
\]
$C$ is any rational localization of $A$. Here the first row is exact by $B[[G]]/E$-pseudoflatness and the second is row is exact as well. The middle vertical arrow is an isomorphism while the rightmost vertical arrow is surjective as in \cite[Lemma 2.5.4]{KL2}. Then we apply the same argument to the leftmost vertical arrow (with a new diagram), we will have the corresponding surjectivity of this leftmost vertical arrow. This will imply that the rightmost vertical arrow is injective as well. Then we have the corresponding isomorphism for the rightmost vertical arrow.
\end{proof}

\begin{theorem}\mbox{\bf{(After Kedlaya-Liu \cite[Theorem 2.5.5]{KL2})}} \label{theorem3.15}
Taking global section will realize the equivalence between the following two categories: A. The category of all the pseudocoherent $\mathcal{O}\widehat{\otimes}B[[G]]$-sheaves; B. The category of all the $B[[G]]$-stably pseudocoherent modules over $A\widehat{\otimes}B[[G]]$. 	
\end{theorem}

\begin{proof}
See \cite[Theorem 2.5.5]{KL2}. We need to still apply \cite[Proposition 2.4.20]{KL1}, as long as one considers instead in our situation \cref{lemma3.14} and \cref{theorem3.12}.\\	
\end{proof}

\subsection{Noncommutative Pseudocoherence in \'Etale Topology}

\indent We consider the extension of the corresponding discussion in the previous subsection to the \'etale topology. At this moment we consider the following same context on the geometric level as in \cite{KL2}:

\begin{setting} \mbox{\bf{(After Kedlaya-Liu \cite[Hypothesis 2.5.8]{KL2})}} \label{setting3.16}
As in the previous subsection we consider now the corresponding geometric setting namely an adic space $\mathrm{Spa}(A,A^+)$ where $A$ is assumed to be sheafy. Then we consider $B[[G]]$ as above. And we consider the corresponding \'etale site $\mathrm{Spa}(A,A^+)_\text{\'et}$. And by keeping the corresponding setting in the geometry as in \cite[Hypothesis 2.5.8]{KL2}, we assume that there is a stable basis $\mathfrak{B}$ containing the space $\mathrm{Spa}(A,A^+)$ itself.	
\end{setting}

\begin{definition} \mbox{\bf{(After Kedlaya-Liu \cite[Definition 2.5.9]{KL2})}}
For any left $A\widehat{\otimes}B[[G]]$-module $M$, we call it $m$-$B[[G]]$-\'etale stably pseudocoherent with respect to $\mathfrak{B}$ if we have that it is $m$-$B[[G]]$-pseudocoherent, complete for the natural topology and for any rational localization $A\rightarrow A'$, the base change of $M$ to $A'\widehat{\otimes} B[[G]]$ is complete for the natural topology as the corresponding left $A'\widehat{\otimes} B[[G]]$-module.	As in \cite[Definition 2.4.1]{KL2} we call that the corresponding left $A\widehat{\otimes}B[[G]]$-module $M$ just $B[[G]]$-stably pseudocoherent if we have that it is simply just $\infty$-$B[[G]]$-stably pseudocoherent.
\end{definition}

\begin{definition}\mbox{\bf{(After Kedlaya-Liu \cite[Below Definition 2.5.9]{KL2})}}
For any Banach left $A\widehat{\otimes}B[[G]]$-module $M$, we call it is $m$-$B[[G]]$-\'etale-pseudoflat if for any right $A\widehat{\otimes}B[[G]]$-module $M'$ $m$-$B[[G]]$-\'etale-stably pseudocoherent we have $\mathrm{Tor}_1^{A\widehat{\otimes}B[[G]]}(M',M)=0$. For any Banach right $A\widehat{\otimes}B[[G]]$-module $M$, we call it is $m$-$B[[G]]$-\'etale-pseudoflat if for any left $A\widehat{\otimes}B[[G]]$-module $M'$ $m$-$B[[G]]$-\'etale-stably pseudocoherent we have $\mathrm{Tor}_1^{A\widehat{\otimes}B[[G]]}(M,M')=0$.
\end{definition}

\indent The following proposition holds in our current setting.

\begin{proposition} \mbox{\bf{(After Kedlaya-Liu \cite[Lemma 2.5.10]{KL2})}} \label{proposition3.19}
One can actually find another basis $\mathfrak{C}$ in $\mathfrak{B}$ such that any morphism in $\mathfrak{C}$ could be $2$-$B[[G]]$-\'etale-pseudoflat with respect to either $\mathfrak{C}$ or $\mathfrak{B}$.	
\end{proposition}

\begin{proof}
We follow the proof of \cite[Lemma 2.5.10]{KL2} in our current $B[[G]]$-relative situation, however not that much needs to proof by relying on the proof of \cite[Lemma 2.5.10]{KL2}. To be more precise the corresponding selection of the new basis $\mathfrak{C}$ comes from including all the morphism made up of some composition of rational localization and the corresponding finite \'etale ones. For these we have shown above the corresponding 2-$B[[G]]$-\'etale pseudoflatness. Then the rest will be pure geometric for the analytic spaces which are just the same as the situation of \cite[Lemma 2.5.10]{KL2}, therefore we just omit the corresponding argument, see \cite[Lemma 2.5.10]{KL2}.	
\end{proof}

\indent Then we have the corresponding Tate's acyclicity in the noncommutative deformed setting in \'etale topology (here fix $\mathfrak{C}$ as above):

\begin{theorem}\mbox{\bf{(After Kedlaya-Liu \cite[Theorem 2.5.11]{KL2})}} \label{theorem3.20} Now suppose we have in our corresponding \cref{setting3.1} a corresponding $B[[G]]$-\'etale-stably pseudocoherent module $M$. Then we consider the corresponding assignment such that for any $U\subset \mathrm{Spa}(A,A^+)$ we define $\widetilde{M}(U)$ as in the following:
\begin{align}
\widetilde{M}(U):=\varprojlim_{\mathrm{Spa}(S,S^+)\subset U,\in \mathfrak{C}} S\widehat{\otimes}B[[G]]\otimes_{A\widehat{\otimes}B[[G]]}M.	
\end{align}
Then we have that for any $\mathfrak{B}$ which is a covering of $U=\mathrm{Spa}(S,S^+)\subset \mathrm{Spa}(A,A^+)$ (certainly this $U$ is also assumed to be in $\mathfrak{C}$, and we assume this covering is formed by using the corresponding members in $\mathfrak{C}$) we have that the vanishing of the following two cohomology groups:
\begin{align}
H^i(U,\widetilde{M}), \check{H}^i(U:\mathfrak{B},\widetilde{M})
\end{align}
for any $i>0$. When concentrating at the degree zero we have:
\begin{align}
H^0(U,\widetilde{M})=S\widehat{\otimes}B[[G]]\otimes_{A\widehat{\otimes}B[[G]]}M, \check{H}^0(U:\mathfrak{B},\widetilde{M})=S\widehat{\otimes}B[[G]]\otimes_{A\widehat{\otimes}B[[G]]}M.
\end{align}
	
\end{theorem}

\begin{proof}
By \cite[Proposition 8.2.21]{KL1}, we could then finish proof as in \cite[Theorem 2.5.11]{KL2} by using our previous \cref{proposition3.19} and the corresponding faithfully flat descent as in the situation of \cite[Theorem 2.5.11]{KL2}.	
\end{proof}

\indent We now consider the corresponding noncommutative deformed version of Kiehl's glueing properties for stably pseudocoherent modules after \cite{KL2}:

\begin{definition} \mbox{\bf{(After Kedlaya-Liu \cite[Definition 2.5.12]{KL2})}}
Consider in our context (over $\mathrm{Spa}(A,A^+)_{\text{\'et}}$) the corresponding sheaves $\mathcal{O}_{\mathrm{Spa}(A,A^+)_{\text{\'et}}}\widehat{\otimes}B[[G]]$, we will then define the corresponding pseudocoherent sheaves over $\mathrm{Spa}(A,A^+)_{\text{\'et}}$ to be those locally defined by attaching \'etale-stably-pseudocoherent modules over the section. 
\end{definition}

\begin{lemma} \mbox{\bf{(After Kedlaya-Liu \cite[Lemma 2.5.13]{KL2})}}
	Consider the corresponding notations in \cite[Lemma 2.4.10]{KL2}, we have the corresponding morphism $A\rightarrow B_1\bigoplus B_2$. Then we have that this morphism is an descent morphism effective for the corresponding $B[[G]]$-\'etale-stably pseodocoherent Banach modules. 
\end{lemma}

\begin{proof}
This is a $B[[G]]$-relative version of the \cite[Lemma 2.5.13]{KL2}. We adapt the corresponding proof to our situation. First consider the corresponding notations on the commutative adic space $\mathrm{Spa}(A,A^+)$ as in \cite[Lemma 2.4.10]{KL2}. And we have a coving $\{U_1,U_2\}$. Then consider a descent datum for this cover of $B[[G]]$-\'etale-stably pseudocoherent sheaves. By \cref{theorem2.19} we could realize this as a corresponding pseudocoherent sheaf over $\mathcal{O}\widehat{\otimes}B[[G]]$. We denote this by $\mathcal{M}$. Then by consider the corresponding results in the rational localization situation above we have that the vanishing of $H^i(U_j;\mathcal{M})$ for $j\in \{1,2\},i>0$. Then as in \cite[Lemma 2.5.13]{KL2} we have the corresponding vanishing of $H^i(\mathrm{Spa}(A,A^+);\mathcal{M})$ for $i>0$. Then as in \cite[Lemma 2.5.13]{KL2} we choose a corresponding covering of some finite free object $\mathcal{F}$ and take its kernel $\mathcal{K}$ which gives rise to the corresponding exact sequence:
\[\tiny
\xymatrix@C+0pc@R+6pc{
0 \ar[r]\ar[r]\ar[r] &\mathcal{K}(\mathrm{Spa}(A,A^+))   \ar[r]\ar[r]\ar[r] &\mathcal{F}(\mathrm{Spa}(A,A^+)) \ar[r]\ar[r]\ar[r] & \mathcal{M}(\mathrm{Spa}(A,A^+)) \ar[r]\ar[r]\ar[r] &0,
}
\]
where we have $p=1,2$. We have already the corresponding stably-pseudocoherence of the global section. As in \cite[Lemma 2.5.13]{KL2} we only need to show that it is stable under \'etale morphism. Then we consider the following diagram:

\[ \tiny
\xymatrix@C+0pc@R+6pc{
 &(C\widehat{\otimes}B[[G]]/E)\otimes\mathcal{K}(\mathrm{Spa}(A,A^+))   \ar[r]\ar[r]\ar[r] \ar[d]\ar[d]\ar[d] &(C\widehat{\otimes}B[[G]]/E)\otimes\mathcal{F}(\mathrm{Spa}(A,A^+)) \ar[r]\ar[r]\ar[r] \ar[d]\ar[d]\ar[d] & (C\widehat{\otimes}B[[G]]/E)\otimes\mathcal{M}(\mathrm{Spa}(A,A^+))\ar[r]\ar[r]\ar[r]  \ar[d]\ar[d]\ar[d]&0,\\
0 \ar[r]\ar[r]\ar[r] &\mathcal{K}(\mathrm{Spa}(C,C^+))   \ar[r]\ar[r]\ar[r] &\mathcal{F}(\mathrm{Spa}(C,C^+)) \ar[r]\ar[r]\ar[r] & \mathcal{M}(\mathrm{Spa}(C,C^+)) \ar[r]\ar[r]\ar[r] &0,\\
}
\]
$A\rightarrow C$ is some member in $\mathfrak{C}$. Here the first row is exact by $B[[G]]/E$-\'etale-pseudoflatness and the second is row is exact as well. The middle vertical arrow is an isomorphism while the rightmost vertical arrow is surjective as in \cite[Lemma 2.5.13]{KL2}. Then we apply the same argument to the leftmost vertical arrow (with a new diagram), we will have the corresponding surjectivity of this leftmost vertical arrow. This will imply that the rightmost vertical arrow is injective as well. Then we have the corresponding isomorphism for the rightmost vertical arrow.
\end{proof}

\begin{theorem}\mbox{\bf{(After Kedlaya-Liu \cite[Theorem 2.5.14]{KL2})}} \label{theorem3.23}
Taking global section will realize the equivalence between the following two categories: A. The category of all the pseudocoherent $\mathcal{O}_{\mathrm{Spa}(A,A^+)_{\text{\'et}}}\widehat{\otimes}B[[G]]$-sheaves; B. The category of all the $B[[G]]$-\'etale-stably pseudocoherent modules over $A\widehat{\otimes}B[[G]]$. 	
\end{theorem}

\begin{proof}
See \cite[Theorem 2.5.14]{KL2}. We need to still apply \cite[Theorem 8.2.22]{KL1}, as long as one considers instead in our situation, and \cite[Tag 03OD]{SP}.\\	
\end{proof}

\subsection{Noncommutative Deformation over Quasi-Stein Spaces}

\indent In this subsection we consider the corresponding noncommutative deformation over the corresponding context in \cite[Chapter 2.6]{KL2}.

\begin{setting}
Let $X$ be a corresponding quasi-Stein adic space over $\mathbb{Q}_p$  in the sense of \cite[Definition 2.6.2]{KL2}. Recall that what is happening is that $X$ could be written as the corresponding direct limit of affinoids $X:=\varinjlim_i X_i$. 	
\end{setting}

\begin{lemma} \mbox{\bf{(After Kedlaya-Liu \cite[Lemma 2.6.3]{KL2})}}
We now consider the corresponding rings $A_i:=\mathcal{O}_{X_i}$ for all $i=0,1,...$, and in our current situation we consider the corresponding rings $A_i\widehat{\otimes}B[[G]]$ (over $\mathbb{Q}_p$) for all $i=0,1,...$. And in our situation we consider the corresponding modules $M_i$ over $A_i\widehat{\otimes}B[[G]]$ for all $i=0,1,...$ with the same requirement as in \cite[Lemma 2.6.3]{KL2} (namely those complete with respect to the natural topology). Suppose that we have bounded surjective map from $f_i:A_{i}\widehat{\otimes}B[[G]]\otimes_{A_{i+1}\widehat{\otimes}B[[G]]} M_{i+1}\rightarrow M_i,i=0,1,...$. Then we have first the density of the corresponding image of $\varprojlim_i M_i$ in each $M_i$ for any $i=0,1,2,...$. And we have as well the corresponding vanishing of $R^1\varprojlim_i M_i$.
\end{lemma}

\begin{proof}
This is the $B[[G]]$-relative version of the result in \cite[Lemma 2.6.3]{KL2}. One needs to consider the corresponding truncation through some specific open subgroup $E$. Then see \cref{lemma2.24}.	
\end{proof}

\begin{proposition} \mbox{\bf{(After Kedlaya-Liu \cite[Lemma 2.6.4]{KL2})}} In the same situation as above, suppose we have that the corresponding modules $M_i$ are basically $B[[G]]$-stably pseudocoherent over the rings $A_i$ for all $i=0,1,...$. Now we consider the situation where $f_i:A_i\widehat{\otimes}B[[G]]\otimes_{A_{i+1}\widehat{\otimes}B[[G]]}M_{i+1}\rightarrow M_i$ is an isomorphism. Then the conclusion in our situation is then that the corresponding projection from $\varprojlim M_i$ to $M_i$ for each $i=0,1,2,...$ is an isomorphism.

\end{proposition}

\begin{proof}
This is a $B[[G]]$-relative version of the \cite[Lemma 2.6.4]{KL2}. We adapt the argument to our situation as in the following. First we choose some finite free covering $T$ of the limit $M$ such that we have for each $i$ the corresponding map $T_i\rightarrow M_i$ is surjective. Then we consider the index $j\geq i$ and set the kernel of the map from $T_j$ to $M_j$ to be $S_j$. By the direct analog of \cite[Lemma 2.5.6]{KL2} we have that $A_i\widehat{\otimes}B[[G]]/E\otimes S_j\overset{}{\rightarrow}S_i$ realizes the isomorphism, and we have that the corresponding surjectivity of the corresponding map from $\varprojlim_i S_i$ projecting to the $S_i$. Then one could finish the proof by 5-lemma to the following commutative diagram as in \cite[Lemma 2.6.4]{KL2}:
\[ \tiny
\xymatrix@C+0pc@R+6pc{
 &(A_i\widehat{\otimes}B[[G]]/E)\otimes \varprojlim_i S_i \ar[r]\ar[r]\ar[r] \ar[d]\ar[d]\ar[d] &(A_i\widehat{\otimes}B[[G]]/E)\otimes \varprojlim_i F_i \ar[r]\ar[r]\ar[r] \ar[d]\ar[d]\ar[d] &(A_i\widehat{\otimes}B[[G]]/E)\otimes \varprojlim_i M_i\ar[r]\ar[r]\ar[r]  \ar[d]\ar[d]\ar[d]&0,\\
0 \ar[r]\ar[r]\ar[r] &S_i   \ar[r]\ar[r]\ar[r] &F_i \ar[r]\ar[r]\ar[r] &M_i \ar[r]\ar[r]\ar[r] &0.\\
}
\]
 
\end{proof}

\begin{proposition}  \mbox{\bf{(After Kedlaya-Liu \cite[Theorem 2.6.5]{KL2})}}
For any quasi-compact adic affinoid space of $X$ which is denoted by $Y$, we have that the map $\mathcal{M}(X)\rightarrow \mathcal{M}(Y)$ is surjective for any $B[[G]]$-stably pseudocoherent sheaf $\mathcal{M}$ over the sheaf $\mathcal{O}_X\widehat{\otimes}B[[G]]$.	
\end{proposition}

\begin{proof}
This is just the corresponding corollary of the previous proposition.	
\end{proof}

\begin{proposition}  \mbox{\bf{(After Kedlaya-Liu \cite[Theorem 2.6.5]{KL2})}}
We have that the stalk $\mathcal{M}_x$ is finitely generated over the stalk $\mathcal{O}_{X,x}$ for any $x\in X$ by $M(X)$, for any $B[[G]]$-stably pseudocoherent sheaf $\mathcal{M}$ over the sheaf $\mathcal{O}_X\widehat{\otimes}B[[G]]$.	
\end{proposition}

\begin{proof}
This is just the corresponding corollary of the proposition before the previous proposition.	
\end{proof}

\begin{proposition}  \mbox{\bf{(After Kedlaya-Liu \cite[Theorem 2.6.5]{KL2})}}
For any quasi-compact adic affinoid space of $X$ which is denoted by $Y$, we have that the corresponding vanishing of the corresponding sheaf cohomology groups $H^k(X,\mathcal{M})$ of $\mathcal{M}$ for higher $k>0$, for any $B[[G]]$-stably pseudocoherent sheaf $\mathcal{M}$ over the sheaf $\mathcal{O}_X\widehat{\otimes}B[[G]]$.	
\end{proposition}

\begin{proof}
See the proof of \cref{proposition2.28}. Again in our situation we only have to work over some quotient.

\end{proof}

\begin{corollary}  \mbox{\bf{(After Kedlaya-Liu \cite[Corollary 2.6.6]{KL2})}}
The corresponding functor from the corresponding $B[[G]]$-deformed pseudocoherent sheaves over $X$ to the corresponding $B[[G]]$-stably by taking the corresponding global section is an exact functor.
\end{corollary}

\begin{corollary}  \mbox{\bf{(After Kedlaya-Liu \cite[Corollary 2.6.8]{KL2})}}
Consider a particular $\mathcal{O}_X\widehat{\otimes}B[[G]]$-pseudocoherent sheaf $\mathcal{M}$ which is finite locally free throughout the whole space $X$. Then we have that the global section $\mathcal{M}(X)$ as $\mathcal{O}_X(X)\widehat{\otimes}B[[G]]$ left module admits the corresponding structures of finite projective structure if and only if we have the corresponding global section is finitely generated.
\end{corollary}

\begin{proof}
As in \cite[Corollary 2.6.8]{KL2} one could find some global splitting through the local splittings.
\end{proof}

\indent We now consider the following $B[[G]]$-relative analog of \cite[Proposition 2.6.17]{KL2}:

\begin{theorem} \mbox{\bf{(After Kedlaya-Liu \cite[Proposition 2.6.17]{KL2})}} \label{theorem3.32}
Consider the following two statements for a particular $\mathcal{O}_X\widehat{\otimes}B[[G]]$-pseudocoherent sheaf $\mathcal{M}$. First is that one can find finite many generators (the number is up to some uniform fixed integer $n\geq 0$) for each section of $\mathcal{M}(X_i)$ for each $i=1,2,...$. The second statement is that the global section $\mathcal{M}(X)$ is just finitely generated. Then in our situation the two statement is equivalent if we have that the corresponding space $X$ admits a $m$-uniformed covering in the exact same sense of \cite[Proposition 2.6.17]{KL2}. 	
\end{theorem}

\begin{proof}
See the proof of \cref{theorem2.31}.
\end{proof}

\newpage

\section{Foundations on Noncommutative Descent for Adic Spectra in Limit of Fr\'echet Case}

\subsection{Noncommutative Pseudocoherence in Analytic Topology}

\indent We now establish some foundations in the noncommutative setting on glueing noncommutative pseudocoherent modules after \cite[Chapter 2]{KL2}. But we remind the readers that we will fix the base space:

\begin{setting} \label{setting4.1}
Consider a corresponding sheafy Banach adic uniform algebra $(A,A^+)$ over $\mathbb{Q}_p$ or $\mathbb{F}_p((t))$, we consider the base space $\mathrm{Spa}(A,A^+)$. And now we will consider a further noncommutative Banach algebra $(B,B^+)$ over $\mathbb{Q}_p$ or $\mathbb{F}_p((t))$ which could be written as the following injective limit:
\begin{displaymath}
B=\varinjlim_{h} B_h,	
\end{displaymath}
where each $B_h$ is Banach (or more general Fr\'echet although we do not consider such generality). 	
\end{setting}

\begin{example}
There are many interesting models in our current context, for instance the period ring $B_e$ as considered by Berger in \cite{Ber1} and the corresponding Robba rings in the full setting as in \cite{KL2} (note that the latter is really indeed ind-Fr\'echet).	
\end{example}

\begin{definition} \mbox{\bf{(After Kedlaya-Liu \cite[Definition 2.4.1]{KL2})}}
For any left $A\widehat{\otimes}B$-module $M$, we call it $m$-$B$-stably pseudocoherent if we have that it is $m$-$B$-pseudocoherent, complete for the natural topology and for any morphism $A\rightarrow A'$ which is the corresponding rational localization, the base change of $M$ to $A'\widehat{\otimes} B$ is complete for the natural LF topology as the corresponding left $A'\widehat{\otimes} B$-module.	As in \cite[Definition 2.4.1]{KL2} we call that the corresponding left $A\widehat{\otimes}B$-module $M$ just $B$-stably pseudocoherent if we have that it is simply just $\infty$-$B$-stably pseudocoherent.
\end{definition}

\begin{definition}\mbox{\bf{(After Kedlaya-Liu \cite[Definition 2.4.4]{KL2})}}
For any Banach left $A\widehat{\otimes}B$-module $M$, we call it is $m$-$B$-pseudoflat if for any right $A\widehat{\otimes}B$-module $M'$ $m$-$B$-stably pseudocoherent we have $\mathrm{Tor}_1^{A\widehat{\otimes}B}(M',M)=0$. For any Banach right $A\widehat{\otimes}B$-module $M$, we call it is $m$-$B$-pseudoflat if for any left $A\widehat{\otimes}B$-module $M'$ $m$-$B$-stably pseudocoherent we have $\mathrm{Tor}_1^{A\widehat{\otimes}B}(M,M')=0$.
\end{definition}

\begin{definition}\mbox{\bf{(After Kedlaya-Liu \cite[Definition 2.4.6]{KL2})}} We consider the corresponding notion of the corresponding pro-projective module. We define over $A$ the corresponding $B$-pro-projective module $M$ to be a corresponding left $A\widehat{\otimes}B$-module such that one could find a filtered sequence of projectors such that in the sense of taking the prolimit we have the corresponding projector will converge to any chosen element in $M$. Here we assume the module is complete with respect to natural topology, and we assume the projectors are $A\widehat{\otimes}\varinjlim_h B_h$-linear and we assume that the corresponding image of the projectors are modules which are also finitely generated and projective.
	
\end{definition}

\begin{lemma}\mbox{\bf{(After Kedlaya-Liu \cite[Lemma 2.4.7]{KL2})}}
Suppose we have over the space $\mathrm{Spa}(A,A^+)$ a corresponding $B$-pro-projective left module $M$. And suppose that we have a $2$-$B$-pseudoflat right module $C$ over $A$. And we assume that $C$ is complete for the LF topology. Then we have that the corresponding product $C{\otimes}_{A\widehat{\otimes}B}M$ is then complete under the corresponding natural topology. And moreover we have that in our situation:
\begin{displaymath}
\mathrm{Tor}_1^{A\widehat{\otimes}B}(C,M)=0.	
\end{displaymath}

\end{lemma}

\begin{proof}
This will be just a noncommutative and LF version of the corresponding  \cite[Lemma 2.4.7]{KL2}. We adapt the corresponding argument in \cite[Lemma 2.4.7]{KL2} to our situation. What we are going to consider in this case is then first consider the following presentation:
\[
\xymatrix@C+0pc@R+0pc{
M'  \ar[r]\ar[r]\ar[r] &(A\widehat{\otimes}\varinjlim_h B_h)^k \ar[r]\ar[r]\ar[r] & M\ar[r]\ar[r]\ar[r] &0,
}
\]
where the left module $M'$ is finitely presented. Then we consider the following commutative diagram:
\[
\xymatrix@C+0pc@R+3pc{
&C\otimes_{A\widehat{\otimes}\varinjlim_h B_h}M'  \ar[r]\ar[r]\ar[r]  \ar[d]\ar[d]\ar[d] &C\otimes_{A\widehat{\otimes}\varinjlim_h B_h}(A\widehat{\otimes}\varinjlim_h B_h)^k \ar[r]\ar[r]\ar[r] \ar[d]\ar[d]\ar[d] & C\otimes_{A\widehat{\otimes}\varinjlim_h B_h}M\ar[r]\ar[r]\ar[r] \ar[d]\ar[d]\ar[d] &0,\\
0 \ar[r]\ar[r]\ar[r] &C\widehat{\otimes}_{A\widehat{\otimes}\varinjlim_h B_h}M'  \ar[r]\ar[r]\ar[r] &C\widehat{\otimes}_{A\widehat{\otimes}B}(A\widehat{\otimes}\varinjlim_h B_h)^k \ar[r]\ar[r]\ar[r] & C\widehat{\otimes}_{A\widehat{\otimes}\varinjlim_h B_h}M\ar[r]\ar[r]\ar[r] &0.
}
\]\\
The first row is exact as in \cite[Lemma 2.4.7]{KL2}, and we have the corresponding second row is also exact since we have that by hypothesis the corresponding module $M$ is $B$-pro-projective.  
Then as in \cite[Lemma 2.4.7]{KL2} the results follow.	
\end{proof}

\begin{proposition} \mbox{\bf{(After Kedlaya-Liu \cite[Corollary 2.4.8]{KL2})}}
Over the space $\mathrm{Spa}(A,A^+)$, we have any $B$-pro-projective left module is $2$-$B$-pseudoflat.
\end{proposition}

\begin{proof}
This is a direct consequence of the previous lemma.	
\end{proof}

\begin{lemma} \mbox{\bf{(After Kedlaya-Liu \cite[Corollary 2.4.9]{KL2})}}
Keep the notation above, suppose we are working over the adic space $
\mathrm{Spa}(A,A^+)$. Suppose now the left $A\widehat{\otimes}B$-module is finitely generated, then we have that the following natural maps:
\begin{align}
A\{T\}\widehat{\otimes}B\otimes_{A\widehat{\otimes}B}M 	\rightarrow M\{T\},\\\
A\{T,T^{-1}\}\widehat{\otimes}B \otimes_{A\widehat{\otimes}B} M 	\rightarrow M\{T,T^{-1}\}	
\end{align}
are surjective. Suppose now the left $A\widehat{\otimes}B$-module is finitely presented, then we have that the following natural maps:
\begin{align}
A\{T\}\widehat{\otimes}B\otimes_{A\widehat{\otimes}B}M 	\rightarrow M\{T\},\\\
A\{T,T^{-1}\}\widehat{\otimes}B \otimes_{A\widehat{\otimes}B} M 	\rightarrow M\{T,T^{-1}\}	
\end{align}
are bijective. Here all the modules are assumed to be complete for the LF topology.

\end{lemma}

\begin{proof}
This is the corresponding consequence of the previous lemma.
\end{proof}

\begin{proposition} \mbox{\bf{(After Kedlaya-Liu \cite[Lemma 2.4.12]{KL2})}}
We keep the corresponding notations in \cite[Lemma 2.4.10]{KL2}. Then in our current sense we have that the corresponding morphism $A\rightarrow B_2$ is then $2$-$B$-pseudoflat. 
\end{proposition}

\begin{proof}
The corresponding proof could be made parallel to the corresponding proof of \cite[Lemma 2.4.12]{KL2}. To be more precise we consider the following commutative diagram from the corresponding chosen exact sequence (which is just the analog of the corresponding one in \cite[Lemma 2.4.12]{KL2}, namely as below we choose arbitrary short exact sequence with $P$ as a left $A\widehat{\otimes}\varinjlim_h B_h$-module which is $2$-$\varinjlim_h B_h$-stably pseudocoherent):
\[
\xymatrix@C+0pc@R+0pc{
0   \ar[r]\ar[r]\ar[r] &M \ar[r]\ar[r]\ar[r] &N \ar[r]\ar[r]\ar[r] & P \ar[r]\ar[r]\ar[r] &0,
}
\]
which then induces the following corresponding big commutative diagram:
\[\tiny
\xymatrix@C+0pc@R+4pc{
& &0 \ar[d]\ar[d]\ar[d] &0 \ar[d]\ar[d]\ar[d] &,\\
0   \ar[r]\ar[r]\ar[r]  &{A\{T\}\widehat{\otimes}\varinjlim_h B_h}\otimes_{A\widehat{\otimes}\varinjlim_h B_h} M \ar[r]\ar[r]\ar[r] \ar[d]^{1-fT}\ar[d]\ar[d] &{A\{T\}\widehat{\otimes}\varinjlim_h B_h}\otimes_{A\widehat{\otimes}\varinjlim_h B_h} N \ar[r]\ar[r]\ar[r] \ar[d]^{1-fT}\ar[d]\ar[d] &{A\{T\}\widehat{\otimes}\varinjlim_h B_h} \otimes_{A\widehat{\otimes}\varinjlim_h B_h}P \ar[r]\ar[r]\ar[r] \ar[d]^{1-fT}\ar[d]\ar[d] &0,\\
0   \ar[r]\ar[r]\ar[r] &{A\{T\}\widehat{\otimes}\varinjlim_h B_h}\otimes_{A\widehat{\otimes}\varinjlim_h B_h}M  \ar[r]\ar[r]\ar[r] \ar[d]\ar[d]\ar[d] &{A\{T\}\widehat{\otimes}\varinjlim_h B_h} \otimes_{A\widehat{\otimes}\varinjlim_h B_h}N\ar[r]\ar[r]\ar[r] \ar[d]\ar[d]\ar[d] & {A\{T\}\widehat{\otimes}\varinjlim_h B_h}\otimes_{A\widehat{\otimes}\varinjlim_h B_h}P \ar[r]\ar[r]\ar[r] \ar[d]\ar[d]\ar[d] &0,\\
0  \ar[r]^?\ar[r]\ar[r] &{B_2\widehat{\otimes}\varinjlim_h B_h}\otimes_{A\widehat{\otimes}\varinjlim_h B_h} M \ar[r]\ar[r]\ar[r] \ar[d]\ar[d]\ar[d] &{B_2\widehat{\otimes}\varinjlim_h B_h}\otimes_{A\widehat{\otimes}\varinjlim_h B_h}N \ar[r]\ar[r]\ar[r] \ar[d]\ar[d]\ar[d] & {B_2\widehat{\otimes}\varinjlim_h B_h}\otimes_{A\widehat{\otimes}\varinjlim_h B_h} P \ar[r]\ar[r]\ar[r] \ar[d]\ar[d]\ar[d] &0,\\
&0&0&0
}
\]
with the notations in \cite[Lemma 2.4.10]{KL2} where we actually have the corresponding exactness along the horizontal direction at the corner around $M\otimes_{A\widehat{\otimes}\varinjlim_h B_h} {B_2\widehat{\otimes}\varinjlim_h B_h}$ marked with $?$, by diagram chasing.
\end{proof}

\begin{proposition} \mbox{\bf{(After Kedlaya-Liu \cite[Lemma 2.4.13]{KL2})}} \label{proposition4.9}
We keep the corresponding notations in \cite[Lemma 2.4.10]{KL2}. Then in our current sense we have that the corresponding morphism $A\rightarrow B_1$ is then $2$-$B$-pseudoflat. And in our current sense we have that the corresponding morphism $A\rightarrow B_{1,2}$ is then $2$-$B$-pseudoflat. 
\end{proposition}

\begin{proof}
The statement for $A\rightarrow B_{1,2}$ could be proved by consider the composition $A\rightarrow B_2\rightarrow B_{12}$ where the $2$-$B$-pseudoflatness could be proved as in \cite[Lemma 2.4.13]{KL2} by inverting the corresponding variable. For $A\rightarrow B_{1}$,
the corresponding proof could be made parallel to the corresponding proof of \cite[Lemma 2.4.13]{KL2}. To be more precise we consider the following commutative diagram from the corresponding chosen exact sequence (which is just the analog of the corresponding one in \cite[Remark 2.4.5]{KL2}, namely as below we choose arbitrary short exact sequence with $P$ as a left $A\widehat{\otimes}B$-module which is $2$-$B$-stably pseudocoherent):
\[
\xymatrix@C+0pc@R+0pc{
0   \ar[r]\ar[r]\ar[r] &M \ar[r]\ar[r]\ar[r] &N \ar[r]\ar[r]\ar[r] & P \ar[r]\ar[r]\ar[r] &0,
}
\]
which then induces the following corresponding big commutative diagram:
\[\tiny
\xymatrix@C+0pc@R+5pc{
& &0 \ar[d]\ar[d]\ar[d] &0 \ar[d]\ar[d]\ar[d] &,\\
0   \ar[r]\ar[r]\ar[r]  &M \ar[r]\ar[r]\ar[r] \ar[d]\ar[d]\ar[d] &N \ar[r]\ar[r]\ar[r] \ar[d]\ar[d]\ar[d] & P \ar[r]\ar[r]\ar[r] \ar[d]\ar[d]\ar[d] &0,\\
0   \ar[r]^?\ar[r]\ar[r] &({B_1\widehat{\otimes}B_\infty}\bigoplus {B_2\widehat{\otimes}B_\infty}) \otimes_{A\widehat{\otimes}B_\infty}M  \ar[r]\ar[r]\ar[r] \ar[d]\ar[d]\ar[d] &({B_1\widehat{\otimes}B_\infty}\bigoplus {B_2\widehat{\otimes}B_\infty})\otimes_{A\widehat{\otimes}B_\infty}N  \ar[r]\ar[r]\ar[r] \ar[d]\ar[d]\ar[d] &({B_1\widehat{\otimes}B_\infty}\bigoplus {B_2\widehat{\otimes}B_\infty})\otimes_{A\widehat{\otimes}B_\infty}P\ar[r]\ar[r]\ar[r] \ar[d]\ar[d]\ar[d] &0,\\
0  \ar[r]\ar[r]\ar[r] &{B_{12}\widehat{\otimes}B_\infty}\otimes_{A\widehat{\otimes}B_\infty}M \ar[r]\ar[r]\ar[r] \ar[d]\ar[d]\ar[d] &{B_{12}\widehat{\otimes}B_\infty}\otimes_{A\widehat{\otimes}B_\infty}N \ar[r]\ar[r]\ar[r] \ar[d]\ar[d]\ar[d] &{B_{12}\widehat{\otimes}B_\infty} \otimes_{A\widehat{\otimes}B_\infty} P \ar[r]\ar[r]\ar[r] \ar[d]\ar[d]\ar[d] &0,\\
&0&0&0
}
\]
with the notations in \cite[Lemma 2.4.10]{KL2} where we actually have the corresponding exactness along the horizontal direction at the corner around $M\otimes_{A\widehat{\otimes}\varinjlim_h B_h} ({B_1\widehat{\otimes}\varinjlim_h B_h}\bigoplus {B_2\widehat{\otimes}\varinjlim_h B_h})$ marked with $?$, by diagram chasing. Here $B_\infty:=\varinjlim_h B_h$.
\end{proof}

\indent After these foundational results as in \cite{KL2} we have the following proposition which is the corresponding noncommutative LF generalization of the corresponding result established in \cite[Theorem 2.4.15]{KL2}.

\begin{proposition} \mbox{\bf{(After Kedlaya-Liu \cite[Theorem 2.4.15]{KL2})}} \label{proposition4.10}
In our current context we have that for any rational localization $\mathrm{Spa}(A',A^{',+})\rightarrow \mathrm{Spa}(A,A^+)$ we have that along this base change the corresponding $B$-stably pseudocoherence is preserved.
	
\end{proposition}

\indent Then we have the corresponding Tate's acyclicity in the noncommutative deformed setting:

\begin{theorem}\mbox{\bf{(After Kedlaya-Liu \cite[Theorem 2.5.1]{KL2})}} \label{theorem4.11} Now suppose we have in our corresponding \cref{setting4.1} a corresponding $B$-stably pseudocoherent module $M$. Then we consider the corresponding assignment such that for any $U\subset \mathrm{Spa}(A,A^+)$ we define $\widetilde{M}(U)$ as in the following:
\begin{align}
\widetilde{M}(U):=\varprojlim_{\mathrm{Spa}(S,S^+)\subset U,\mathrm{rational}} S\widehat{\otimes}B\otimes_{A\widehat{\otimes}B}M.	
\end{align}
Then we have that for any $\mathfrak{B}$ which is a rational covering of $U=\mathrm{Spa}(S,S^+)\subset \mathrm{Spa}(A,A^+)$ (certainly this $U$ is also assumed to be rational) we have that the vanishing of the following two cohomology groups:
\begin{align}
H^i(U,\widetilde{M}), \check{H}^i(U:\mathfrak{B},\widetilde{M})
\end{align}
for any $i>0$. When concentrating at the degree zero we have:
\begin{align}
H^0(U,\widetilde{M})=S\widehat{\otimes}B\otimes_{A\widehat{\otimes}B}M, \check{H}^0(U:\mathfrak{B},\widetilde{M})=S\widehat{\otimes}B\otimes_{A\widehat{\otimes}B}M.
\end{align}
	
\end{theorem}

\begin{proof}
By \cite[Propositions 2.4.20-2.4.21]{KL1}, we could then finish proof as in \cite[Theorem 2.5.1]{KL2} by using our previous \cref{proposition4.10} and \cref{proposition4.9} as above.	
\end{proof}

\indent We now consider the corresponding noncommutative LF deformed version of Kiehl's glueing properties for stably pseudocoherent modules after \cite{KL2}:

\begin{definition} \mbox{\bf{(After Kedlaya-Liu \cite[Definition 2.5.3]{KL2})}}
Consider in our context (over $\mathrm{Spa}(A,A^+)$) the corresponding sheaves $\mathcal{O}\widehat{\otimes}B$, we will then define the corresponding pseudocoherent sheaves over $\mathrm{Spa}(A,A^+)$ to be those locally defined by attaching stably-pseudocoherent modules over the section. 
\end{definition}

\begin{lemma} \mbox{\bf{(After Kedlaya-Liu \cite[Lemma 2.5.4]{KL2})}}\label{lemma4.13}
	Consider the corresponding notations in \cite[Lemma 2.4.10]{KL2}, we have the corresponding morphism $A\rightarrow B_1\bigoplus B_2$. Then we have that this morphism is an descent morphism effective for the corresponding $B$-stably pseodocoherent Banach modules. 
\end{lemma}

\begin{proof}
See the proof of \cref{lemma2.13}.
\end{proof}

\begin{theorem}\mbox{\bf{(After Kedlaya-Liu \cite[Theorem 2.5.5]{KL2})}} \label{theorem4.14}
Taking global section will realize the equivalence between the following two categories: A. The category of all the pseudocoherent $\mathcal{O}\widehat{\otimes}B$-sheaves; B. The category of all the $B$-stably pseudocoherent modules over $A\widehat{\otimes}B$. 	
\end{theorem}

\begin{proof}
See \cite[Theorem 2.5.5]{KL2}. We need to still apply \cite[Proposition 2.4.20]{KL1}, as long as one considers instead in our situation \cref{lemma4.13} and \cref{theorem4.11}.\\	
\end{proof}

\subsection{Noncommutative Pseudocoherence in \'Etale Topology}

\indent We consider the extension of the corresponding discussion in the previous subsection to the \'etale topology. At this moment we consider the following same context on the geometric level as in \cite{KL2}:

\begin{setting} \mbox{\bf{(After Kedlaya-Liu \cite[Hypothesis 2.5.8]{KL2})}} \label{setting4.15}
As in the previous subsection we consider now the corresponding geometric setting namely an adic space $\mathrm{Spa}(A,A^+)$ where $A$ is assumed to be sheafy. Then we consider $B:=\varinjlim_h B_h$ as above. And we consider the corresponding \'etale site $\mathrm{Spa}(A,A^+)_\text{\'et}$. And by keeping the corresponding setting in the geometry as in \cite[Hypothesis 2.5.8]{KL2}, we assume that there is a stable basis $\mathfrak{B}$ containing the space $\mathrm{Spa}(A,A^+)$ itself.	
\end{setting}

\begin{definition} \mbox{\bf{(After Kedlaya-Liu \cite[Definition 2.5.9]{KL2})}}
For any left $A\widehat{\otimes}B$-module $M$, we call it $m$-$B$-\'etale stably pseudocoherent with respect to $\mathfrak{B}$ if we have that it is $m$-$B$-pseudocoherent, complete for the natural topology and for any rational localization $A\rightarrow A'$, the base change of $M$ to $A'\widehat{\otimes} B$ is complete for the natural LF topology as the corresponding left $A'\widehat{\otimes} B$-module.	As in \cite[Definition 2.4.1]{KL2} we call that the corresponding left $A\widehat{\otimes}B$-module $M$ just $B$-stably pseudocoherent if we have that it is simply just $\infty$-$B$-stably pseudocoherent.
\end{definition}

\begin{definition}\mbox{\bf{(After Kedlaya-Liu \cite[Below Definition 2.5.9]{KL2})}}
For any Banach left $A\widehat{\otimes}B$-module $M$, we call it is $m$-$B$-\'etale-pseudoflat if for any right $A\widehat{\otimes}B$-module $M'$ $m$-$B$-\'etale-stably pseudocoherent we have $\mathrm{Tor}_1^{A\widehat{\otimes}B}(M',M)=0$. For any Banach right $A\widehat{\otimes}B$-module $M$, we call it is $m$-$B$-\'etale-pseudoflat if for any left $A\widehat{\otimes}B$-module $M'$ $m$-$B$-\'etale-stably pseudocoherent we have $\mathrm{Tor}_1^{A\widehat{\otimes}B}(M,M')=0$.
\end{definition}

\indent The following proposition holds in our current setting.

\begin{proposition} \mbox{\bf{(After Kedlaya-Liu \cite[Lemma 2.5.10]{KL2})}} \label{proposition4.18}
One can actually find another basis $\mathfrak{C}$ in $\mathfrak{B}$ such that any morphism in $\mathfrak{C}$ could be $2$-$B$-\'etale-pseudoflat with respect to either $\mathfrak{C}$ or $\mathfrak{B}$.	
\end{proposition}

\begin{proof}
We follow the proof of \cite[Lemma 2.5.10]{KL2} in our current $B$-relative situation, however not that much needs to proof by relying on the proof of \cite[Lemma 2.5.10]{KL2}. To be more precise the corresponding selection of the new basis $\mathfrak{C}$ comes from including all the morphism made up of some composition of rational localization and the corresponding finite \'etale ones. For these we have shown above the corresponding 2-$B$-\'etale pseudoflatness. Then the rest will be pure geometric for the analytic spaces which are just the same as the situation of \cite[Lemma 2.5.10]{KL2}, therefore we just omit the corresponding argument, see \cite[Lemma 2.5.10]{KL2}.	
\end{proof}

\indent Then we have the corresponding Tate's acyclicity in the noncommutative deformed setting in \'etale topology (here fix $\mathfrak{C}$ as above):

\begin{theorem}\mbox{\bf{(After Kedlaya-Liu \cite[Theorem 2.5.11]{KL2})}} \label{theorem4.19} Now suppose we have in our corresponding \cref{setting4.1} a corresponding $B$-\'etale-stably pseudocoherent module $M$. Then we consider the corresponding assignment such that for any $U\subset \mathrm{Spa}(A,A^+)$ we define $\widetilde{M}(U)$ as in the following:
\begin{align}
\widetilde{M}(U):=\varprojlim_{\mathrm{Spa}(S,S^+)\subset U,\in \mathfrak{C}} S\widehat{\otimes}B\otimes_{A\widehat{\otimes}B}M.	
\end{align}
Then we have that for any $\mathfrak{B}$ which is a covering of $U=\mathrm{Spa}(S,S^+)\subset \mathrm{Spa}(A,A^+)$ (certainly this $U$ is also assumed to be in $\mathfrak{C}$, and we assume this covering is formed by using the corresponding members in $\mathfrak{C}$) we have that the vanishing of the following two cohomology groups:
\begin{align}
H^i(U,\widetilde{M}), \check{H}^i(U:\mathfrak{B},\widetilde{M})
\end{align}
for any $i>0$. When concentrating at the degree zero we have:
\begin{align}
H^0(U,\widetilde{M})=S\widehat{\otimes}B\otimes_{A\widehat{\otimes}B}M, \check{H}^0(U:\mathfrak{B},\widetilde{M})=S\widehat{\otimes}B\otimes_{A\widehat{\otimes}B}M.
\end{align}
	
\end{theorem}

\begin{proof}
By \cite[Proposition 8.2.21]{KL1}, we could then finish proof as in \cite[Theorem 2.5.11]{KL2} by using our previous \cref{proposition4.18} and the corresponding faithfully flat descent as in the situation of \cite[Theorem 2.5.11]{KL2}.	
\end{proof}

\indent We now consider the corresponding noncommutative LF deformed version of Kiehl's glueing properties for stably pseudocoherent modules after \cite{KL2}:

\begin{definition} \mbox{\bf{(After Kedlaya-Liu \cite[Definition 2.5.12]{KL2})}}
Consider in our context (over $\mathrm{Spa}(A,A^+)_{\text{\'et}}$) the corresponding sheaves $\mathcal{O}_{\mathrm{Spa}(A,A^+)_{\text{\'et}}}\widehat{\otimes}B$, we will then define the corresponding pseudocoherent sheaves over $\mathrm{Spa}(A,A^+)_{\text{\'et}}$ to be those locally defined by attaching \'etale-stably-pseudocoherent modules over the section. 
\end{definition}

\begin{lemma} \mbox{\bf{(After Kedlaya-Liu \cite[Lemma 2.5.13]{KL2})}}
	Consider the corresponding notations in \cite[Lemma 2.4.10]{KL2}, we have the corresponding morphism $A\rightarrow B_1\bigoplus B_2$. Then we have that this morphism is an descent morphism effective for the corresponding $B$-\'etale-stably pseodocoherent Banach modules. 
\end{lemma}

\begin{proof}
See the proof of \cref{lemma2.21}.

\end{proof}

\begin{theorem}\mbox{\bf{(After Kedlaya-Liu \cite[Theorem 2.5.14]{KL2})}} \label{theorem4.22}
Taking global section will realize the equivalence between the following two categories: A. The category of all the pseudocoherent $\mathcal{O}_{\mathrm{Spa}(A,A^+)_{\text{\'et}}}\widehat{\otimes}B$-sheaves; B. The category of all the $B$-\'etale-stably pseudocoherent modules over $A\widehat{\otimes}B$. 	
\end{theorem}

\begin{proof}
See \cite[Theorem 2.5.14]{KL2}. We need to still apply \cite[Theorem 8.2.22]{KL1}, as long as one considers instead in our situation, and \cite[Tag 03OD]{SP}.\\	
\end{proof}

\subsection{Noncommutative Deformation over Quasi-Stein Spaces}

\indent In this subsection we consider the corresponding noncommutative deformation over the corresponding context in \cite[Chapter 2.6]{KL2}.

\begin{setting}
Let $X$ be a corresponding quasi-Stein adic space over $\mathbb{Q}_p$ or $\mathbb{F}_p((t))$ in the sense of \cite[Definition 2.6.2]{KL2}. Recall that what is happening is that $X$ could be written as the corresponding direct limit of affinoids $X:=\varinjlim_i X_i$. 	
\end{setting}

\begin{lemma} \mbox{\bf{(After Kedlaya-Liu \cite[Lemma 2.6.3]{KL2})}}
We now consider the corresponding rings $A_i:=\mathcal{O}_{X_i}$ for all $i=0,1,...$, and in our current situation we consider the corresponding rings $A_i\widehat{\otimes}B$ (over $\mathbb{Q}_p$ or $\mathbb{F}_p((t))$) for all $i=0,1,...$. And in our situation we consider the corresponding modules $M_i$ over $A_i\widehat{\otimes}B$ for all $i=0,1,...$ with the same requirement as in \cite[Lemma 2.6.3]{KL2} (namely those complete with respect to the natural topology which is LF topology in our current situation). Suppose that we have bounded surjective map from $f_i:A_{i}\widehat{\otimes}B\otimes_{A_{i+1}\widehat{\otimes}B} M_{i+1}\rightarrow M_i,i=0,1,...$. Then we have first the density of the corresponding image of $\varprojlim_i M_i$ in each $M_i$ for any $i=0,1,2,...$. And we have as well the corresponding vanishing of $R^1\varprojlim_i M_i$.
\end{lemma}

\begin{proof}
This is the $\varinjlim_h B_h$-relative version of the result in \cite[Lemma 2.6.3]{KL2}. We need to then use the corresponding family of Banach norms parametrized by the index $h$. For the first statement we just choose sequence of Banach norms on all the corresponding modules for all $i=0,1,...$ and $h$ such that we have $\|f_i(x_{i+1})\|^h_i\leq 1/2\|x_{i+1}\|^h_{i+1}$ for any $x_{i+1}\in M_{i+1}$. Then for any $x_i\in M_i$ and any $\delta>0$, we now consider for any $j\geq 1$ and $\forall h$ the corresponding $x_{i+j}$ such that we have $\|x_{i+j}-f_{i+j+1}(x_{i+j+1})\|^h_{i+j}\leq \delta$. Then the sequence $x_{i+j+k},k=0,1,...$ will converge to some well-defined $y_{i+j}$ with in our situation the corresponding $y_{i}=f_{i}(y_{i+1})$. We then have $\|x_i-y_i\|^h_i\leq \delta$. This will prove the first statement. For the second statement as in \cite[Lemma 2.6.3]{KL2} we form the product $M_0\times M_1\times M_2\times...$ and the consider the induced map $F$ from $M_{i+1}\rightarrow M_i$, and consider the corresponding cokernel of the map $1-F$ since this is just the corresponding limit we are considering. Then to show that the cokernel is zero we just look at the corresponding cokernel of the corresponding map on the corresponding completed direct summand which will project to the original one. But then we will have $\|f_i(v)\|^h_i\leq 1/2 \|v\|^h_i$, which produces an inverse to $1-F$ which will basically finish the proof for the second statement. 	
\end{proof}

\begin{proposition} \mbox{\bf{(After Kedlaya-Liu \cite[Lemma 2.6.4]{KL2})}} In the same situation as above, suppose we have that the corresponding modules $M_i$ are basically $B$-stably pseudocoherent over the rings $A_i$ for all $i=0,1,...$. Now we consider the situation where $f_i:A_i\widehat{\otimes}B\otimes_{A_{i+1}\widehat{\otimes}B}M_{i+1}\rightarrow M_i$ is an isomorphism. Then the conclusion in our situation is then that the corresponding projection from $\varprojlim M_i$ to $M_i$ for each $i=0,1,2,...$.

\end{proposition}

\begin{proof}
This is a $\varinjlim_h B_h$-relative version of the \cite[Lemma 2.6.4]{KL2}. We adapt the argument to our situation as in the following. First we choose some finite free covering $T$ of the limit $M$ such that we have for each $i$ the corresponding map $T_i\rightarrow M_i$ is surjective. Then we consider the index $j\geq i$ and set the kernel of the map from $T_j$ to $M_j$ to be $S_j$. By the direct analog of \cite[Lemma 2.5.6]{KL2} we have that $A_i\widehat{\otimes}\varinjlim_h B_h\otimes S_j\overset{}{\rightarrow}S_i$ realizes the isomorphism, and we have that the corresponding surjectivity of the corresponding map from $\varprojlim_i S_i$ projecting to the $S_i$. Then one could finish the proof by 5-lemma to the following commutative diagram as in \cite[Lemma 2.6.4]{KL2}:
\[ \tiny
\xymatrix@C+0pc@R+6pc{
 &(A_i\widehat{\otimes}\varinjlim_h B_h)\otimes \varprojlim_i S_i \ar[r]\ar[r]\ar[r] \ar[d]\ar[d]\ar[d] &(A_i\widehat{\otimes}\varinjlim_h B_h)\otimes \varprojlim_i F_i \ar[r]\ar[r]\ar[r] \ar[d]\ar[d]\ar[d] &(A_i\widehat{\otimes}\varinjlim_h B_h)\otimes \varprojlim_i M_i\ar[r]\ar[r]\ar[r]  \ar[d]\ar[d]\ar[d]&0,\\
0 \ar[r]\ar[r]\ar[r] &S_i   \ar[r]\ar[r]\ar[r] &F_i \ar[r]\ar[r]\ar[r] &M_i \ar[r]\ar[r]\ar[r] &0.\\
}
\]
 
\end{proof}

\begin{proposition}  \mbox{\bf{(After Kedlaya-Liu \cite[Theorem 2.6.5]{KL2})}}
For any quasi-compact adic affinoid space of $X$ which is denoted by $Y$, we have that the map $\mathcal{M}(X)\rightarrow \mathcal{M}(Y)$ is surjective for any $B$-stably pseudocoherent sheaf $\mathcal{M}$ over the sheaf $\mathcal{O}_X\widehat{\otimes}B$.	
\end{proposition}

\begin{proof}
This is just the corresponding corollary of the previous proposition.	
\end{proof}

\begin{proposition}  \mbox{\bf{(After Kedlaya-Liu \cite[Theorem 2.6.5]{KL2})}}
We have that the stalk $\mathcal{M}_x$ is finitely generated over the stalk $\mathcal{O}_{X,x}$ for any $x\in X$ by $M(X)$, for any $B$-stably pseudocoherent sheaf $\mathcal{M}$ over the sheaf $\mathcal{O}_X\widehat{\otimes}B$.	
\end{proposition}

\begin{proof}
This is just the corresponding corollary of the proposition before the previous proposition.	
\end{proof}

\begin{proposition}  \mbox{\bf{(After Kedlaya-Liu \cite[Theorem 2.6.5]{KL2})}}
For any quasi-compact adic affinoid space of $X$ which is denoted by $Y$, we have that the corresponding vanishing of the corresponding sheaf cohomology groups $H^k(X,\mathcal{M})$ of $\mathcal{M}$ for higher $k>0$, for any $B$-stably pseudocoherent sheaf $\mathcal{M}$ over the sheaf $\mathcal{O}_X\widehat{\otimes}B$.	
\end{proposition}

\begin{proof}
See the proof of \cref{proposition2.28}.
\end{proof}

\begin{corollary}  \mbox{\bf{(After Kedlaya-Liu \cite[Corollary 2.6.6]{KL2})}}
The corresponding functor from the corresponding $B$-deformed pseudocoherent sheaves over $X$ to the corresponding $B$-stably by taking the corresponding global section is an exact functor.
\end{corollary}

\begin{corollary}  \mbox{\bf{(After Kedlaya-Liu \cite[Corollary 2.6.8]{KL2})}}
Consider a particular $\mathcal{O}_X\widehat{\otimes}B$-pseudocoherent sheaf $\mathcal{M}$ which is finite locally free throughout the whole space $X$. Then we have that the global section $\mathcal{M}(X)$ as $\mathcal{O}_X(X)\widehat{\otimes}B$ left module admits the corresponding structures of finite projective structure if and only if we have the corresponding global section is finitely generated.
\end{corollary}

\begin{proof}
As in \cite[Corollary 2.6.8]{KL2} one could find some global splitting through the local splittings.
\end{proof}

\indent We now consider the following $B$-relative analog of \cite[Proposition 2.6.17]{KL2}:

\begin{theorem} \mbox{\bf{(After Kedlaya-Liu \cite[Proposition 2.6.17]{KL2})}} \label{theorem4.31}
Consider the following two statements for a particular $\mathcal{O}_X\widehat{\otimes}B$-pseudocoherent sheaf $\mathcal{M}$. First is that one can find finite many generators (the number is up to some uniform fixed integer $n\geq 0$) for each section of $\mathcal{M}(X_i)$ for each $i=1,2,...$. The second statement is that the global section $\mathcal{M}(X)$ is just finitely generated. Then in our situation the two statement is equivalent if we have that the corresponding space $X$ admits a $m$-uniformed covering in the exact same sense of \cite[Proposition 2.6.17]{KL2}. 	
\end{theorem}

\begin{proof}
See \cref{theorem2.31}.
\end{proof}

\newpage

\section{Descent over Analytic Huber Pairs over $\mathbb{Z}_p$}

\subsection{Noncommutative Deformation over Analytic Huber Pairs in the Analytic Topology}

\noindent We now study the corresponding glueing of stably pseudocoherent sheaves after our previous work \cite{T3}. However since we currently are working in the corresponding framework of \cite{Ked2}, we will consider the corresponding analytic Huber pair in \cite{Ked2}.

\begin{setting}
Now we consider the corresponding Huber pair taking the general form of $(A,A^+)$ as in \cite[Definition 1.1.2]{Ked2}, namely it is uniform analytic. We now assume we are going to work over some base $(V,V^+)=(\mathbb{Z}_p,\mathbb{Z}_p)$. And we will assume that we have another topological ring $Z$ which is assumed to satisfy the following condition: for any exact sequence of uniform analytic Huber rings
\[
\xymatrix@C+0pc@R+0pc{
0 \ar[r] \ar[r] \ar[r] &\Gamma_1 \ar[r] \ar[r] \ar[r] &\Gamma_2 \ar[r] \ar[r] \ar[r] &\Gamma_3 \ar[r] \ar[r] \ar[r] &0,
}
\] 
we have that the following is alway exact:
\[
\xymatrix@C+0pc@R+0pc{
0 \ar[r] \ar[r] \ar[r] &\Gamma_1\widehat{\otimes} Z \ar[r] \ar[r] \ar[r] &\Gamma_2\widehat{\otimes} Z \ar[r] \ar[r] \ar[r] &\Gamma_3\widehat{\otimes} Z \ar[r] \ar[r] \ar[r] &0.
}
\] 
This is achieved for instance when we have that the map $C\rightarrow Z$ splits in the category of all the topological modules.
\end{setting}

\begin{setting}
We will maintain the corresponding assumption in \cite[Hypothesis 1.7.1]{Ked2}. To be more precise we will have a map $(A,A^+)\rightarrow (B,B^+)$ which is a corresponding rational localization. And recall the corresponding complex in \cite{Ked2}:
\[
\xymatrix@C+0pc@R+0pc{
0 \ar[r] \ar[r] \ar[r] &B \ar[r] \ar[r] \ar[r] &B\left\{\frac{f}{g}\right\}\bigoplus B\left\{\frac{g}{f}\right\} \ar[r] \ar[r] \ar[r] &B\left\{\frac{f}{g},\frac{g}{f}\right\} \ar[r] \ar[r] \ar[r] &0.
}
\]	
\end{setting}

\begin{setting}
We will need to consider the corresponding nice rational localizations in the sense of \cite[Definition 1.9.1]{Ked2} where such localizations are defined to be those composites of the corresponding rational localizations in the Laurent or the balanced situation. For the topological modules, we will always assume that the modules are left modules.
\end{setting}

\begin{definition}\mbox{\bf{(After Kedlaya \cite[Definition 1.9.1]{Ked2})}}
Over $A$, we define a corresponding $Z$-stably-pseudocoherent module $M$ to be a pseudocoherent module $M$ over $A\widehat{\otimes}Z$ which is complete with respect to the natural topology. And for any rational localization $A\rightarrow A'$ with respect to $A$, the completeness still holds. Similar we can define the corresponding $m$-$Z$-stably-pseudocoherent modules in the similar way where we just consider the corresponding $m$-pseudocoherent modules on the algebraic level. Over $A$, we define a corresponding $Z$-nice-stably-pseudocoherent module $M$ to be a pseudocoherent module $M$ over $A\widehat{\otimes}Z$ which is complete with respect to the natural topology. And for any nice rational localization $A\rightarrow A'$ with respect to $A$, the completeness still holds. Similar we can define the corresponding $m$-$Z$-nice-stably-pseudocoherent modules in the similar way where we just consider the corresponding $m$-pseudocoherent modules on the algebraic level.
\end{definition}


\begin{lemma} \mbox{\bf{(After Kedlaya \cite[Lemma 1.9.3]{Ked2})}} \label{lemma2.5}
Now assume that we are in the corresponding assumption of \cite[1.7.1]{Ked2}. Also consider the corresponding complex in \cite[1.6.15.1]{Ked2}:
\[
\xymatrix@C+0pc@R+0pc{
0 \ar[r] \ar[r] \ar[r] &B \ar[r] \ar[r] \ar[r] &B\left\{\frac{f}{g}\right\}\bigoplus B\left\{\frac{g}{f}\right\} \ar[r] \ar[r] \ar[r] &B\left\{\frac{f}{g},\frac{g}{f}\right\} \ar[r] \ar[r] \ar[r] &0.
}
\]	
And we assume now that this is exact. Now take $g$ to be $1-f$ or $1$. Here the corresponding elements $f,g$ come from $B$. Now over $B\widehat{\otimes}Z$ suppose we have a module $M$ which is assumed now to be finitely presented and complete with respect to the natural topology over $B\widehat{\otimes}Z$. Then we have that the group $\mathrm{Tor}_1(B\widehat{\otimes}Z\left\{\frac{g}{f}\right\},M)$ is zero. 	
\end{lemma}

\begin{proof}
This is a $Z$-relative version of the corresponding \cite[Lemma 1.9.3]{Ked2}. We briefly recall the argument being adapted in our situation. As in \cite[Lemma 1.9.3]{Ked2} we can reduce ourselves to the situation where $g$ is just $1$. Then regard the corresponding ring $B\widehat{\otimes}Z\left\{\frac{g}{f}\right\}$ as $B\widehat{\otimes}Z\left\{T\right\}/(1-fT)$. Then we consider the corresponding in the ring $B\widehat{\otimes}Z[[T]]$ the formal inverse of $1-fT$ namely the series $1+fT+f^2T^2+...$. This will force multiplication by this element to be injective. This just finishes the proof as in \cite[Lemma 1.9.3]{Ked2}.
\end{proof}

\begin{lemma} \mbox{\bf{(After Kedlaya \cite[Lemma 1.9.4]{Ked2})}} \label{lemma2.6}
Now assume that we are in the corresponding assumption of \cite[1.7.1]{Ked2}. Also consider the corresponding complex in \cite[1.6.15.1]{Ked2}:
\[
\xymatrix@C+0pc@R+0pc{
0 \ar[r] \ar[r] \ar[r] &B \ar[r] \ar[r] \ar[r] &B\left\{\frac{f}{g}\right\}\bigoplus B\left\{\frac{g}{f}\right\} \ar[r] \ar[r] \ar[r] &B\left\{\frac{f}{g},\frac{g}{f}\right\} \ar[r] \ar[r] \ar[r] &0.
}
\]	
And we assume now that this is exact. Now take $g$ to be $1-f$ or $1$. Here the corresponding element $f,g$ come from $B$. Now over $B\widehat{\otimes}Z$ suppose we have a module $M$ which is assumed now to be finitely presented and complete with respect to the natural topology over $B\widehat{\otimes}Z$, and furthermore the corresponding completeness is stable under any corresponding nice rational localization $B\rightarrow C$. Then we have that the group $\mathrm{Tor}_1(B\widehat{\otimes}Z\left\{\frac{f}{g}\right\},M)$ is zero and the group $\mathrm{Tor}_1(B\widehat{\otimes}Z\left\{\frac{f}{g},\frac{g}{f}\right\},M)$ is zero as well. 	
\end{lemma}

\begin{proof}
Choose a corresponding presentation for $M$ taking the form of:
\[
\xymatrix@C+0pc@R+0pc{
0 \ar[r] \ar[r] \ar[r] &X_1 \ar[r] \ar[r] \ar[r] &X_2\ar[r] \ar[r] \ar[r] &M \ar[r] \ar[r] \ar[r] &0.
}
\]
Here the module $X_2$ is a finite free left module. Then we have by the previous lemma the following exact sequence:
\[
\xymatrix@C+0pc@R+0pc{
0 \ar[r] \ar[r] \ar[r] &B\widehat{\otimes}Z\left\{\frac{g}{f}\right\}\otimes_{B\widehat{\otimes}Z}X_1 \ar[r] \ar[r] \ar[r] &B\widehat{\otimes}Z\left\{\frac{g}{f}\right\}\otimes_{B\widehat{\otimes}Z}X_2 \ar[r] \ar[r] \ar[r] &B\widehat{\otimes}Z\left\{\frac{g}{f}\right\}\otimes_{B\widehat{\otimes}Z}M \ar[r] \ar[r] \ar[r] &0.
}
\]	
Now apply the previous lemma again as in the proof of \cite[Lemma 1.9.4]{Ked2} we have the following exact sequence:
\[
\xymatrix@C+0pc@R+0pc{
0 \ar[r] \ar[r] \ar[r] &B\widehat{\otimes}Z\left\{\frac{f}{g},\frac{g}{f}\right\}\otimes_{B\widehat{\otimes}Z}X_1 \ar[r] \ar[r] \ar[r] &B\widehat{\otimes}Z\left\{\frac{f}{g},\frac{g}{f}\right\}\otimes_{B\widehat{\otimes}Z}X_2 \ar[r] \ar[r] \ar[r] &B\widehat{\otimes}Z\left\{\frac{f}{g},\frac{g}{f}\right\}\otimes_{B\widehat{\otimes}Z}M \ar[r] \ar[r] \ar[r] &0.
}
\]	
Then to tackle the situation for $\frac{f}{g}$, we just further take the tensor product with the corresponding sequence as in the following:
\[
\xymatrix@C+0pc@R+0pc{
0 \ar[r] \ar[r] \ar[r] &B\widehat{\otimes}Z \ar[r] \ar[r] \ar[r] &B\widehat{\otimes}Z\left\{\frac{f}{g}\right\}\bigoplus B\widehat{\otimes}Z\left\{\frac{g}{f}\right\} \ar[r] \ar[r] \ar[r] &B\widehat{\otimes}Z\left\{\frac{f}{g},\frac{g}{f}\right\} \ar[r] \ar[r] \ar[r] &0.
}
\]
\end{proof}

\indent Then we have the following corollary which is the corresponding analog of \cite[Corollary 1.9.5]{Ked2}:

\begin{corollary}\mbox{\bf{(After Kedlaya \cite[Corollary 1.9.5]{Ked2})}}
Working over a uniform analytic Huber pair $(A,A^+)$ which is sheafy, suppose we consider a finitely generated $A\widehat{\otimes}Z$ module complete with respect to the natural topology. Then we have that for any nice rational localization $A\rightarrow B$ in the sense of \cite[Definition 1.9.1]{Ked2} then tensoring with the corresponding $Z$ we have the vanishing of $\mathrm{Tor}_1(B,M)$.
	
\end{corollary}

\begin{proof}
By the previous lemmas.	
\end{proof}

\begin{corollary}\mbox{\bf{(After Kedlaya \cite[Corollary 1.9.6]{Ked2})}}
Working over a uniform analytic Huber pair $(A,A^+)$ which is sheafy, suppose we consider a nice $Z$-stably pseudocoherent left $A\widehat{\otimes}Z$ module $M$. Then we have that for any rational localization $A\rightarrow B$ and any rational localization from $B$ to $C$ in the nice setting as in \cite[Definition 1.9.1]{Ked2}. We have that base change along the map $B\rightarrow C$ (note that this is assumed to be nice) will preserve the corresponding stably pseudocoherence for $M$.

\end{corollary}

\begin{lemma} \mbox{\bf{(After Kedlaya \cite[Lemma 1.9.7]{Ked2})}}
Now assume that we are in the corresponding assumption of \cite[1.7.1]{Ked2}. Also consider the corresponding complex in \cite[1.6.15.1]{Ked2}:
\[
\xymatrix@C+0pc@R+0pc{
0 \ar[r] \ar[r] \ar[r] &B \ar[r] \ar[r] \ar[r] &B\left\{\frac{f}{g}\right\}\bigoplus B\left\{\frac{g}{f}\right\} \ar[r] \ar[r] \ar[r] &B\left\{\frac{f}{g},\frac{g}{f}\right\} \ar[r] \ar[r] \ar[r] &0.
}
\]	
And we assume now that this is exact. Now take $g$ to be $1-f$ or $1$. Here the corresponding element $f,g$ come from $B$. Now over $B\widehat{\otimes}Z$ suppose we have a module $M$ which is assumed now to be finitely presented and complete with respect to the natural topology over $B\widehat{\otimes}Z$, and furthermore the corresponding completeness is stable under any corresponding nice rational localization $B\rightarrow C$. Then we have that over then $B\widehat{\otimes}Z$ in our situation tensoring with $M$ will preserve the corresponding exactness of the following:

\[
\xymatrix@C+0pc@R+0pc{
0 \ar[r] \ar[r] \ar[r] &B\widehat{\otimes}Z \ar[r] \ar[r] \ar[r] &B\widehat{\otimes}Z \left\{\frac{f}{g}\right\}\bigoplus B\widehat{\otimes}Z \left\{\frac{g}{f}\right\} \ar[r] \ar[r] \ar[r] &B\widehat{\otimes}Z \left\{\frac{f}{g},\frac{g}{f}\right\} \ar[r] \ar[r] \ar[r] &0.
}
\]

\end{lemma}

\begin{proof}
See \cite[Lemma 1.9.7]{Ked2}.	
\end{proof}

\begin{theorem} \mbox{\bf{(After Kedlaya \cite[Corollary 1.9.8]{Ked2})}}
Working over a uniform analytic Huber pair $(A,A^+)$ which is sheafy, suppose we consider a $Z$-nice-stably pseudocoherent left $A\widehat{\otimes}Z$ module $M$. Then we have that the presheaf $\widetilde{M}$ associated to the $M$ sheafified along the ring $A$ only is acyclic.
	
\end{theorem}

\begin{proof}
See \cite[Corollary 1.9.8]{Ked2}.	
\end{proof}

\indent Now we establish the corresponding analogs of the corresponding results in \cite{Ked2} which are needed in the corresponding descent of pseudocoherent sheaves.

\begin{lemma}\mbox{\bf{(After Kedlaya \cite[Lemma 1.9.9]{Ked2})}}
Again as in \cite[Lemma 1.9.9]{Ked2} we suppose we are in the situation of 1.7.1 of \cite{Ked2} for the adic rings. Now one can actually find a neighbourhood around 0 of the corresponding ring $B\widehat{\otimes}Z\left\{\frac{f}{g},\frac{g}{f}\right\}$. Then one can find a corresponding decomposition for any invertible matrix $M$ where $M-1$ takes coefficients in this neighbourhood into $M_1M_2$ where invertible $M_1$ takes coefficients in $B\widehat{\otimes}Z\left\{\frac{f}{g}\right\}$ and invertible $M_2$ takes coefficients in $B\widehat{\otimes}Z\left\{\frac{g}{f}\right\}$. 
\end{lemma}

\begin{proof}
See \cite[Lemma 1.9.9]{Ked2}.	
\end{proof}


\begin{lemma}\mbox{\bf{(After Kedlaya \cite[Lemma 1.9.10]{Ked2})}} \label{lemma2.12}
Again as in \cite[Lemma 1.9.10]{Ked2} we suppose we are in the situation of 1.7.1 of \cite{Ked2} for the adic rings. Now for a glueing datum $M_1,M_2,M_{12}$ over the rings respectively $B\widehat{\otimes}Z\left\{\frac{f}{g}\right\}$, $B\widehat{\otimes}Z\left\{\frac{g}{f}\right\}$, $B\widehat{\otimes}Z\left\{\frac{f}{g},\frac{g}{f}\right\}$. We assume that they are finitely generated. We recall this means that we have:
\begin{align}
f_1: B\widehat{\otimes}Z\left\{\frac{f}{g},\frac{g}{f}\right\}\otimes_{B\widehat{\otimes}Z\left\{\frac{f}{g}\right\}}M_1\overset{\sim}{\rightarrow}M_{12},\\
f_2: B\widehat{\otimes}Z\left\{\frac{f}{g},\frac{g}{f}\right\}\otimes_{B\widehat{\otimes}Z\left\{\frac{f}{g}\right\}}M_2\overset{\sim}{\rightarrow}M_{12}.
\end{align}
Then we have that the corresponding strictly surjectivity of the map $M_1\bigoplus M_2\rightarrow M_{12}$ given by $f_1(x_1)-f_2(x_2)$. And we have that the corresponding equalizer $M$ will satisfy that $B\widehat{\otimes}Z\left\{\frac{f}{g}\right\} \otimes_{B\widehat{\otimes}Z}M\overset{}{\rightarrow}M_1$, $B\widehat{\otimes}Z\left\{\frac{g}{f}\right\}\otimes_{B\widehat{\otimes}Z}M\overset{}{\rightarrow}M_2$ are strictly surjective.	
\end{lemma}

\begin{proof}
This is a $Z$-relative version of the corresponding result in \cite[Lemma 1.9.10]{Ked2}. We adapt the corresponding argument there to our situation. Choose basis of $M_1$ namely $e_1,...,e_n$, and choose basis of $M_2$ namely $h_1,...,h_n$. Now we set the corresponding matrix of the corresponding $f_2(w_i)$ under the basis $e_1,...,e_n$ to be $P$ while we set the corresponding matrix of the corresponding $f_1(v_i)$ under the basis $h_1,...,h_n$ to be $Q$. Then what is happening is that we can apply the previous lemma to our current situation to extract a neighbourhood such that $P(f^{-k}Q'-Q)$ (for some sufficiently $k$) takes the value in this neighbourhood and we have that the matrix $P(f^{-k}Q'-Q)+1$ will be decomposed into $K_1K_2^{-1}$ with the corresponding conditions on the entries as in the previous lemma. Then what we could do is as in \cite[Lemma 1.9.10]{Ked2} to consider the setting $(x_1,x_2):=(\sum f^kX_{1ij}e_i,\sum Q'X_{2ij}h_i)$ which gives rise to that $f_1(x_1)-f_2(x_2)=0$ by our construction. Then as in \cite[Lemma 1.9.10]{Ked2} the corresponding result in the first statement will follow once one considers as well \cite[Lemma 1.8.1]{Ked2}. Then for the second statement, we consider application of the first statement as in \cite[Lemma 1.9.10]{Ked2}. Briefly recalling, we consider the second part of statement (while the first part could be derived in the same fashion). To be more precise we consider any element $m_2\in M_2$ then send by $f_2$ to $M_{12}$, we will have by the first statement some $y_1\in M_1,y_2\in M_2$ such that $f_1(y_1)-f_2(y_2)=f_2(m_2)$. This will gives $(y_1,y_2+m_2)$ lives in the equalizer, which will proves the result as in \cite[Lemma 1.9.10]{Ked2}.

\end{proof}

\begin{proposition} \mbox{\bf{(After Kedlaya \cite[Lemma 1.9.11]{Ked2})}} \label{lemma2.13}
Consider the same situation of the previous lemma. And assume that the corresponding \cite[1.6.15.1]{Ked2} is exact. And we assume that the corresponding modules $M_1,M_2,M_{12}$ are $Z$-stably pseudocoherent over the rings respectively $B\widehat{\otimes}Z\left\{\frac{f}{g}\right\}$, $B\widehat{\otimes}Z\left\{\frac{g}{f}\right\}$, $B\widehat{\otimes}Z\left\{\frac{f}{g},\frac{g}{f}\right\}$. Then we have that actually the corresponding equalizer $M$ is just pseudocoherent in our current situation.	
\end{proposition}


\begin{proof}
This is a $Z$-relative version of \cite[Lemma 1.9.11]{Ked2}. We adapt the corresponding argument to our situation. First we consider corresponding modules $M_1,M_2,M_{12}$, and choose finite free modules $L_1,L_2,L_{12}$ as in \cite[Lemma 1.9.11]{Ked2} over $B\widehat{\otimes}Z\left\{\frac{f}{g}\right\}$, $B\widehat{\otimes}Z\left\{\frac{g}{f}\right\}$, $B\widehat{\otimes}Z\left\{\frac{f}{g},\frac{g}{f}\right\}$, then we can form the corresponding commutative diagram:
\[
\xymatrix@C+3pc@R+3pc{
 &0 \ar[d] \ar[d] \ar[d] &0 \ar[d] \ar[d] \ar[d]  &0 \ar[d] \ar[d] \ar[d]&\\
0 \ar[r] \ar[r] \ar[r] &F \ar[r] \ar[r] \ar[r] \ar[d] \ar[d] \ar[d] &F_1\bigoplus F_2 \ar[r] \ar[r] \ar[r] \ar[d] \ar[d] \ar[d] &F_{12} \ar[r]^? \ar[r] \ar[r] \ar[d] \ar[d] \ar[d] &0\\
0 \ar[r] \ar[r] \ar[r] &L \ar[r] \ar[r] \ar[r] \ar[d] \ar[d] \ar[d] &L_1\bigoplus L_2 \ar[r] \ar[r] \ar[r] \ar[d] \ar[d] \ar[d] &L_{12} \ar[r] \ar[r] \ar[r] \ar[d] \ar[d] \ar[d] &0\\
0 \ar[r] \ar[r] \ar[r] &M \ar[r] \ar[r] \ar[r] \ar[d]^? \ar[d] \ar[d] &M_1\bigoplus M_2\ar[r] \ar[r] \ar[r] \ar[d] \ar[d] \ar[d] &M_{12} \ar[r] \ar[r] \ar[r] \ar[d] \ar[d] \ar[d] &0\\
&0 &0  &0  &\\
}
\]	
where two arrows marked by $?$ are not a priori known to be exist to make the corresponding parts of the sequences exact. Here the left vertical arrows come from taking snake lemma (recall that one first chooses a corresponding finite free module $L$ mapping to $M$ which will base change to the corresponding module $L_1,L_2,L_{12}$, then $F$ will be taken as in \cite[Lemma 1.9.11]{Ked2} to be the corresponding equalizer of the corresponding horizontal map, the idea is that the map $L\rightarrow M$ is just a map which is not a priori known to be surjective while eventually this is indeed the case). The corresponding modules $F_1,F_2,F_{12}$ are the corresponding kernels of the corresponding coverings by finite free modules. Then what is happening is that the corresponding first horizontal complex is actually within the same situation we are considering, therefore the previous lemma implies that the $?$-marked arrows exists and make the sequences exact. We then consider the following commutative diagram:
\[
\xymatrix@C+0pc@R+0pc{
 &B\widehat{\otimes}Z\left\{\frac{f}{g}\right\}\otimes F \ar[r] \ar[r] \ar[r] \ar[d] \ar[d] \ar[d] &B\widehat{\otimes}Z\left\{\frac{f}{g}\right\}\otimes L \ar[r] \ar[r] \ar[r] \ar[d] \ar[d] \ar[d] &B\widehat{\otimes}Z\left\{\frac{f}{g}\right\}\otimes M \ar[r] \ar[r] \ar[r] \ar[d] \ar[d] \ar[d] &0.\\
0 \ar[r] \ar[r] \ar[r] &F_1 \ar[r] \ar[r] \ar[r] &L_1 \ar[r] \ar[r] \ar[r] &M_1 \ar[r] \ar[r] \ar[r] &0.\\
}
\]
What we know is that the corresponding middle vertical arrow is isomorphism, while the left and right ones are surjective, which implies the right most one is an isomorphism. We then consider the following commutative diagram:
\[
\xymatrix@C+0pc@R+0pc{
 &B\widehat{\otimes}Z\left\{\frac{g}{f}\right\}\otimes F \ar[r] \ar[r] \ar[r] \ar[d] \ar[d] \ar[d] &B\widehat{\otimes}Z\left\{\frac{g}{f}\right\}\otimes L \ar[r] \ar[r] \ar[r] \ar[d] \ar[d] \ar[d] &B\widehat{\otimes}Z\left\{\frac{g}{f}\right\}\otimes M \ar[r] \ar[r] \ar[r] \ar[d] \ar[d] \ar[d] &0.\\
0 \ar[r] \ar[r] \ar[r] &F_2 \ar[r] \ar[r] \ar[r] &L_2 \ar[r] \ar[r] \ar[r] &M_2 \ar[r] \ar[r] \ar[r] &0.\\
}
\]
What we know is that the corresponding middle vertical arrow is isomorphism, while the left and right ones are surjective, which implies the right most one is an isomorphism. Finally to show that the module $M$ is pseudocoherent we consider the corresponding induction on the number of the finite projective modules participated in the resolution. $k=1$ will be obvious, while in general the induction step will achieved from $k$ to $k+1$ by considering as in \cite[Lemma 1.9.11]{Ked2} the module $F$ and apply the corresponding construction in the same fashion.
\end{proof}

\

\begin{proposition} \mbox{\bf{(After Kedlaya \cite[Lemma 1.9.13]{Ked2})}}
Take the corresponding assumption in \cite[1.7.1]{Ked2} and furthermore assume the sheafiness of $A$. Suppose we have that all the modules considered are $Z$-stably complete with respect to nice rational localization. Then consider $C_1$ the category of all such modules which are assumed to be pseudocoherent over $B\widehat{\otimes}Z\left\{\frac{f}{g},\frac{g}{f}\right\}$. Then consider $C_2$ the category of all such modules which are assumed to be pseudocoherent over $B\widehat{\otimes}Z\left\{\frac{f}{g}\right\}$. Then consider $C_3$ the category of all such modules which are assumed to be pseudocoherent over $B\widehat{\otimes}Z\left\{\frac{g}{f}\right\}$. Then consider $C_4$ the category of all such modules which are assumed to be pseudocoherent over $B\widehat{\otimes}Z$. Then we have a corresponding equalizer functor:
\begin{align}
C_2\times_{C_1}C_3\rightarrow C	
\end{align}
which is exact and fully faithful, and we have well-established inverse.
\end{proposition}

\begin{proof}
See \cite[Lemma 1.9.13]{Ked2}.	
\end{proof}

\begin{proposition}  \mbox{\bf{(After Kedlaya \cite[Lemma 1.9.16]{Ked2})}}
Take the corresponding assumption in \cite[1.7.1]{Ked2} and furthermore assume the sheafiness of $A$. Suppose we have that all the modules considered are $Z$-stably complete with respect to nice rational localization. Let $C$ be the category of such $Z$-nice-stably pseudocoherent modules over $A$ and let $C'$ be the category of such $Z$-nice-stably pseudocoherent sheaves over $X$. If we consider the corresponding functor taking the form of the corresponding sheafification of $\widetilde{M}$	with respect to the corresponding nice rational subspaces. Then we have that in fact this realizes an exact equivalence between $C$ and $C'$.
\end{proposition}

\begin{proof}
See \cite[Lemma 1.9.16]{Ked2}.	
\end{proof}

\begin{proposition}  \mbox{\bf{(After Kedlaya \cite[Lemma 1.9.17]{Ked2})}}
Take the corresponding assumption in \cite[1.7.1]{Ked2} and furthermore assume the sheafiness of $A$. Suppose we have that all the modules considered are $Z$-stably complete with respect to nice rational localization. Let $C$ be the category of such $Z$-nice-stably pseudocoherent modules over $A$ and let $C'$ be the category of such $Z$-nice-stably pseudocoherent modules over $B$. Then we have that in fact this realizes an exact functor from $C$ to $C'$.
\end{proposition}

\begin{proof}
See \cite[Lemma 1.9.17]{Ked2}.	
\end{proof}

\begin{proposition}  \mbox{\bf{(After Kedlaya \cite[Corollary 1.9.18]{Ked2})}}
Take the corresponding assumption in \cite[1.7.1]{Ked2} and furthermore assume the sheafiness of $A$. Let $C$ be the category of such $Z$-stably pseudocoherent modules over $A$ and let $C'$ be the category of such $Z$-nice-stably pseudocoherent modules over $A$. Then we have that in fact this realizes an equality from $C\hookrightarrow C'$.
\end{proposition}

\begin{proof}
See \cite[Corollary 1.9.18]{Ked2}.	
\end{proof}

\begin{theorem} \mbox{\bf{(After Kedlaya \cite[Theorem 1.4.14, Theorem 1.4.16, Theorem 1.4.18]{Ked2})}}
Take the corresponding assumption in \cite[1.7.1]{Ked2} and furthermore assume the sheafiness of $A$. Then: I. We have the corresponding stability under rational localization for $Z$-stably pseudocoherent modules. II. We have the corresponding acyclicity of the corresponding presheaf $\widetilde{M}$ attached to any $Z$-stably pseudocoherent module. III. We have the corresponding glueing of the corresponding $Z$-stably pseudocoherent modules as in \cite[Theorem 1.4.18]{Ked2}.
\end{theorem}

\begin{proof}
This is a very parallel $Z$-relative analog of \cite[Theorem 1.4.14, Theorem 1.4.16, Theorem 1.4.18]{Ked2}, we refer the readers to the proof of \cite[Theorem 1.4.14, Theorem 1.4.16, Theorem 1.4.18]{Ked2}. See \cite[After Corollary 1.9.18, the proof of Theorem 1.4.14, the proof of Theorem 1.4.16, the proof of Theorem 1.4.18]{Ked2}.	
\end{proof}

\newpage

\section{Descent over Adic Banach Rings over $\mathbb{Z}_p$}

\subsection{Noncommutative Deformation over Adic Banach Rings in the Analytic Topology}

\noindent We now study the corresponding glueing of stably pseudocoherent sheaves after our previous work \cite{TX3}. We now translate the corresponding discussion to the context which is more related to \cite{KL2} over adic Banach rings. One can see that these are really parallel to the discussion over analytic Huber pair.

\begin{setting}
Now we consider the corresponding adic Banach ring taking the general form of $(A,A^+)$, namely it is uniform analytic. We now assume we are going to work over some base $(V,V^+)=(\mathbb{Z}_p,\mathbb{Z}_p)$. And we will assume that we have another Banach ring $Z$ over $V$ which is assumed to satisfy the following condition: for any exact sequence of uniform analytic adic Banach rings
\[
\xymatrix@C+0pc@R+0pc{
0 \ar[r] \ar[r] \ar[r] &\Gamma_1 \ar[r] \ar[r] \ar[r] &\Gamma_2 \ar[r] \ar[r] \ar[r] &\Gamma_3 \ar[r] \ar[r] \ar[r] &0,
}
\] 
we have that the following is alway exact:
\[
\xymatrix@C+0pc@R+0pc{
0 \ar[r] \ar[r] \ar[r] &\Gamma_1\widehat{\otimes} Z \ar[r] \ar[r] \ar[r] &\Gamma_2\widehat{\otimes} Z \ar[r] \ar[r] \ar[r] &\Gamma_3\widehat{\otimes} Z \ar[r] \ar[r] \ar[r] &0.
}
\] 
This is achieved for instance when we have that the map $C\rightarrow Z$ splits in the category of all the Banach modules.
\end{setting}

\begin{setting}
We will maintain the corresponding assumption in \cite[Hypothesis 1.7.1]{Ked2}. To be more precise we will have a map $(A,A^+)\rightarrow (B,B^+)$ which is a corresponding rational localization. And recall the corresponding complex in \cite{Ked2}:
\[
\xymatrix@C+0pc@R+0pc{
0 \ar[r] \ar[r] \ar[r] &B \ar[r] \ar[r] \ar[r] &B\left\{\frac{f}{g}\right\}\bigoplus B\left\{\frac{g}{f}\right\} \ar[r] \ar[r] \ar[r] &B\left\{\frac{f}{g},\frac{g}{f}\right\} \ar[r] \ar[r] \ar[r] &0.
}
\]	
\end{setting}

\begin{setting}
We will need to consider the corresponding nice rational localizations in the sense of \cite[Definition 1.9.1]{Ked2} where such localizations are defined to be those composites of the corresponding rational localizations in the Laurent or the balanced situation. For the topological modules, we will always assume that the modules are left modules.
\end{setting}

\begin{definition}\mbox{\bf{(After Kedlaya \cite[Definition 1.9.1]{Ked2})}}
Over $A$, we define a corresponding $Z$-stably-pseudocoherent module $M$ to be a pseudocoherent module $M$ over $A\widehat{\otimes}Z$ which is complete with respect to the natural topology. And for any rational localization $A\rightarrow A'$ with respect to $A$, the completeness still holds. Similar we can define the corresponding $m$-$Z$-stably-pseudocoherent modules in the similar way where we just consider the corresponding $m$-pseudocoherent modules on the algebraic level. Over $A$, we define a corresponding $Z$-nice-stably-pseudocoherent module $M$ to be a pseudocoherent module $M$ over $A\widehat{\otimes}Z$ which is complete with respect to the natural topology. And for any nice rational localization $A\rightarrow A'$ with respect to $A$, the completeness still holds. Similar we can define the corresponding $m$-$Z$-nice-stably-pseudocoherent modules in the similar way where we just consider the corresponding $m$-pseudocoherent modules on the algebraic level.
\end{definition}

\begin{lemma} \mbox{\bf{(After Kedlaya \cite[Lemma 1.9.3]{Ked2})}}
Now assume that we are in the corresponding assumption of \cite[1.7.1]{Ked2}. Also consider the corresponding complex in \cite[1.6.15.1]{Ked2}:
\[
\xymatrix@C+0pc@R+0pc{
0 \ar[r] \ar[r] \ar[r] &B \ar[r] \ar[r] \ar[r] &B\left\{\frac{f}{g}\right\}\bigoplus B\left\{\frac{g}{f}\right\} \ar[r] \ar[r] \ar[r] &B\left\{\frac{f}{g},\frac{g}{f}\right\} \ar[r] \ar[r] \ar[r] &0.
}
\]	
And we assume now that this is exact. Now take $g$ to be $1-f$ or $1$. Here the corresponding element $f,g$ come from $B$. Now over $B\widehat{\otimes}Z$ suppose we have a module $M$ which is assumed now to be finitely presented and complete with respect to the natural topology over $B\widehat{\otimes}Z$. Then we have that the group $\mathrm{Tor}_1(B\widehat{\otimes}Z\left\{\frac{g}{f}\right\},M)$ is zero. 	
\end{lemma}

\begin{proof}
See \cref{lemma2.5}.
\end{proof}

\begin{lemma} \mbox{\bf{(After Kedlaya \cite[Lemma 1.9.4]{Ked2})}}
Now assume that we are in the corresponding parallel assumption of \cite[1.7.1]{Ked2}. Also consider the corresponding complex in \cite[1.6.15.1]{Ked2}:
\[
\xymatrix@C+0pc@R+0pc{
0 \ar[r] \ar[r] \ar[r] &B \ar[r] \ar[r] \ar[r] &B\left\{\frac{f}{g}\right\}\bigoplus B\left\{\frac{g}{f}\right\} \ar[r] \ar[r] \ar[r] &B\left\{\frac{f}{g},\frac{g}{f}\right\} \ar[r] \ar[r] \ar[r] &0.
}
\]	
And we assume now that this is exact. Now take $g$ to be $1-f$ or $1$. Here the corresponding element $f,g$ come from $B$. Now over $B\widehat{\otimes}Z$ suppose we have a module $M$ which is assumed now to be finitely presented and complete with respect to the natural topology over $B\widehat{\otimes}Z$, and furthermore the corresponding completeness is stable under any corresponding nice rational localization $B\rightarrow C$. Then we have that the group $\mathrm{Tor}_1(B\widehat{\otimes}Z\left\{\frac{f}{g}\right\},M)$ is zero and the group $\mathrm{Tor}_1(B\widehat{\otimes}Z\left\{\frac{f}{g},\frac{g}{f}\right\},M)$ is zero as well. 	
\end{lemma}

\begin{proof}
See \cref{lemma2.6}.
\end{proof}

\indent Then we have the following corollary which is the corresponding analog of \cite[Corollary 1.9.5]{Ked2}:

\begin{corollary}\mbox{\bf{(After Kedlaya \cite[Corollary 1.9.5]{Ked2})}}
Working over a uniform analytic adic Banach ring $(A,A^+)$ which is sheafy, suppose we consider a finitely generated $A\widehat{\otimes}Z$ module complete with respect to the natural topology. Then we have that for any nice rational localization $A\rightarrow B$ in the sense of \cite[Definition 1.9.1]{Ked2} then tensoring with the corresponding $Z$ we have the vanishing of $\mathrm{Tor}_1(B,M)$.
	
\end{corollary}

\begin{proof}
By the previous lemmas.	
\end{proof}

\begin{corollary}\mbox{\bf{(After Kedlaya \cite[Corollary 1.9.6]{Ked2})}}
Working over a uniform analytic adic Banach ring $(A,A^+)$ which is sheafy, suppose we consider a nice $Z$-stably pseudocoherent left $A\widehat{\otimes}Z$ module $M$. Then we have that for any rational localization $A\rightarrow B$ and any rational localization from $B$ to $C$ in the nice setting as in \cite[Definition 1.9.1]{Ked2}. We have that base change along the map $B\rightarrow C$ (note that this is assumed to be nice) will preserve the corresponding stably pseudocoherence for $M$.

\end{corollary}

\begin{lemma} \mbox{\bf{(After Kedlaya \cite[Lemma 1.9.7]{Ked2})}}
Now assume that we are in the corresponding parallel assumption of \cite[1.7.1]{Ked2}. Also consider the corresponding complex in \cite[1.6.15.1]{Ked2}:
\[
\xymatrix@C+0pc@R+0pc{
0 \ar[r] \ar[r] \ar[r] &B \ar[r] \ar[r] \ar[r] &B\left\{\frac{f}{g}\right\}\bigoplus B\left\{\frac{g}{f}\right\} \ar[r] \ar[r] \ar[r] &B\left\{\frac{f}{g},\frac{g}{f}\right\} \ar[r] \ar[r] \ar[r] &0.
}
\]	
And we assume now that this is exact. Now take $g$ to be $1-f$ or $1$. Here the corresponding element $f,g$ come from $B$. Now over $B\widehat{\otimes}Z$ suppose we have a module $M$ which is assumed now to be finitely presented and complete with respect to the natural topology over $B\widehat{\otimes}Z$, and furthermore the corresponding completeness is stable under any corresponding nice rational localization $B\rightarrow C$. Then we have that over then $B\widehat{\otimes}Z$ in our situation tensoring with $M$ will preserve the corresponding exactness of the following:

\[
\xymatrix@C+0pc@R+0pc{
0 \ar[r] \ar[r] \ar[r] &B\widehat{\otimes}Z \ar[r] \ar[r] \ar[r] &B\widehat{\otimes}Z \left\{\frac{f}{g}\right\}\bigoplus B\widehat{\otimes}Z \left\{\frac{g}{f}\right\} \ar[r] \ar[r] \ar[r] &B\widehat{\otimes}Z \left\{\frac{f}{g},\frac{g}{f}\right\} \ar[r] \ar[r] \ar[r] &0.
}
\]

\end{lemma}

\begin{proof}
See \cite[Lemma 1.9.7]{Ked2}.	
\end{proof}

\begin{theorem} \mbox{\bf{(After Kedlaya \cite[Corollary 1.9.8]{Ked2})}}
Working over a uniform analytic adic Banach ring $(A,A^+)$ which is sheafy, suppose we consider a $Z$-nice-stably pseudocoherent left $A\widehat{\otimes}Z$ module $M$. Then we have that the presheaf $\widetilde{M}$ associated to the $M$ sheafified along the ring $A$ only is acyclic.
	
\end{theorem}

\begin{proof}
See \cite[Corollary 1.9.8]{Ked2}.	
\end{proof}

\indent Now we establish the corresponding analogs of the corresponding results in \cite{Ked2} which are needed in the corresponding descent of pseudocoherent sheaves.

\begin{lemma}\mbox{\bf{(After Kedlaya \cite[Lemma 1.9.9]{Ked2})}}
Again as in \cite[Lemma 1.9.9]{Ked2} we suppose we are in the parallel situation of 1.7.1 of \cite[Lemma 1.9.9]{Ked2} for the adic Banach rings. Now take a sufficiently small constant $\delta>0$. Then one can find a corresponding decomposition for any invertible matrix $M$ where $M-1$ takes norm $\|M-1\|\leq \delta$ into $M_1M_2$ where invertible $M_1$ takes coefficients in $B\widehat{\otimes}Z\left\{\frac{f}{g}\right\}$ and invertible $M_2$ takes coefficients in $B\widehat{\otimes}Z\left\{\frac{g}{f}\right\}$. 
\end{lemma}

\begin{proof}
See \cite[Lemma 1.9.9]{Ked2}.	
\end{proof}

\begin{lemma}\mbox{\bf{(After Kedlaya \cite[Lemma 1.9.10]{Ked2})}}
Again as in \cite[Lemma 1.9.10]{Ked2} we suppose we are in the parallel situation of 1.7.1 of \cite[Lemma 1.9.10]{Ked2} for the adic Banach rings. Now for a glueing datum $M_1,M_2,M_{12}$ over the rings respectively $B\widehat{\otimes}Z\left\{\frac{f}{g}\right\}$, $B\widehat{\otimes}Z\left\{\frac{g}{f}\right\}$, $B\widehat{\otimes}Z\left\{\frac{f}{g},\frac{g}{f}\right\}$. We assume that they are finitely generated. We recall this means that we have:
\begin{align}
f_1: B\widehat{\otimes}Z\left\{\frac{f}{g},\frac{g}{f}\right\}\otimes_{B\widehat{\otimes}Z\left\{\frac{f}{g}\right\}}M_1\overset{\sim}{\rightarrow}M_{12},\\
f_2: B\widehat{\otimes}Z\left\{\frac{f}{g},\frac{g}{f}\right\}\otimes_{B\widehat{\otimes}Z\left\{\frac{f}{g}\right\}}M_2\overset{\sim}{\rightarrow}M_{12}.
\end{align}
Then we have that the corresponding strictly surjectivity of the map $M_1\bigoplus M_2\rightarrow M_{12}$ given by $f_1(x_1)-f_2(x_2)$. And we have that the corresponding equalizer $M$ will satisfy that $B\widehat{\otimes}Z\left\{\frac{f}{g}\right\} \otimes_{B\widehat{\otimes}Z}M\overset{}{\rightarrow}M_1$, $B\widehat{\otimes}Z\left\{\frac{g}{f}\right\}\otimes_{B\widehat{\otimes}Z}M\overset{}{\rightarrow}M_2$ are strictly surjective.	
\end{lemma}

\begin{proof}
See \cref{lemma2.12}.

\end{proof}

\begin{proposition} \mbox{\bf{(After Kedlaya \cite[Lemma 1.9.11]{Ked2})}}
Consider the same situation of the previous lemma. And assume that the corresponding \cite[1.6.15.1]{Ked2} is exact. And we assume that the corresponding modules $M_1,M_2,M_{12}$ are $Z$-stably pseudocoherent over the rings respectively $B\widehat{\otimes}Z\left\{\frac{f}{g}\right\}$, $B\widehat{\otimes}Z\left\{\frac{g}{f}\right\}$, $B\widehat{\otimes}Z\left\{\frac{f}{g},\frac{g}{f}\right\}$. Then we have that actually the corresponding equalizer $M$ is just pseudocoherent in our current situation.	
\end{proposition}

\begin{proof}
This is a $Z$-relative version of \cite[Lemma 1.9.11]{Ked2}. We adapt the corresponding argument to our situation. First we consider corresponding modules $M_1,M_2,M_{12}$, and choose finite free modules $L_1,L_2,L_{12}$ as in \cite[Lemma 1.9.11]{Ked2} over $B\widehat{\otimes}Z\left\{\frac{f}{g}\right\}$, $B\widehat{\otimes}Z\left\{\frac{g}{f}\right\}$, $B\widehat{\otimes}Z\left\{\frac{f}{g},\frac{g}{f}\right\}$, then we can form the corresponding commutative diagram:
\[
\xymatrix@C+3pc@R+3pc{
 &0 \ar[d] \ar[d] \ar[d] &0 \ar[d] \ar[d] \ar[d]  &0 \ar[d] \ar[d] \ar[d]&\\
0 \ar[r] \ar[r] \ar[r] &F \ar[r] \ar[r] \ar[r] \ar[d] \ar[d] \ar[d] &F_1\bigoplus F_2 \ar[r] \ar[r] \ar[r] \ar[d] \ar[d] \ar[d] &F_{12} \ar[r]^? \ar[r] \ar[r] \ar[d] \ar[d] \ar[d] &0\\
0 \ar[r] \ar[r] \ar[r] &L \ar[r] \ar[r] \ar[r] \ar[d] \ar[d] \ar[d] &L_1\bigoplus L_2 \ar[r] \ar[r] \ar[r] \ar[d] \ar[d] \ar[d] &L_{12} \ar[r] \ar[r] \ar[r] \ar[d] \ar[d] \ar[d] &0\\
0 \ar[r] \ar[r] \ar[r] &M \ar[r] \ar[r] \ar[r] \ar[d]^? \ar[d] \ar[d] &M_1\bigoplus M_2\ar[r] \ar[r] \ar[r] \ar[d] \ar[d] \ar[d] &M_{12} \ar[r] \ar[r] \ar[r] \ar[d] \ar[d] \ar[d] &0\\
&0 &0  &0  &\\
}
\]	
where two arrows marked by $?$ are not a priori known to be exist to make the corresponding parts of the sequences exact. Here the left vertical arrows come from taking snake lemma (recall that one first chooses a corresponding finite free module $L$ mapping to $M$ which will base change to the corresponding module $L_1,L_2,L_{12}$, then $F$ will be taken as in \cite[Lemma 1.9.11]{Ked2} to be the corresponding equalizer of the corresponding horizontal map, the idea is that the map $L\rightarrow M$ is just a map which is not a priori known to be surjective while eventually this is indeed the case). The corresponding modules $F_1,F_2,F_{12}$ are the corresponding kernels of the corresponding coverings by finite free modules. Then what is happening is that the corresponding first horizontal complex is actually within the same situation we are considering, therefore the previous lemma implies that the $?$-marked arrows exists and make the sequences exact. We then consider the following commutative diagram:
\[
\xymatrix@C+0pc@R+0pc{
 &B\widehat{\otimes}Z\left\{\frac{f}{g}\right\}\otimes F \ar[r] \ar[r] \ar[r] \ar[d] \ar[d] \ar[d] &B\widehat{\otimes}Z\left\{\frac{f}{g}\right\}\otimes L \ar[r] \ar[r] \ar[r] \ar[d] \ar[d] \ar[d] &B\widehat{\otimes}Z\left\{\frac{f}{g}\right\}\otimes M \ar[r] \ar[r] \ar[r] \ar[d] \ar[d] \ar[d] &0.\\
0 \ar[r] \ar[r] \ar[r] &F_1 \ar[r] \ar[r] \ar[r] &L_1 \ar[r] \ar[r] \ar[r] &M_1 \ar[r] \ar[r] \ar[r] &0.\\
}
\]
What we know is that the corresponding middle vertical arrow is isomorphism, while the left and right ones are surjective, which implies the right most one is an isomorphism. We then consider the following commutative diagram:
\[
\xymatrix@C+0pc@R+0pc{
 &B\widehat{\otimes}Z\left\{\frac{g}{f}\right\}\otimes F \ar[r] \ar[r] \ar[r] \ar[d] \ar[d] \ar[d] &B\widehat{\otimes}Z\left\{\frac{g}{f}\right\}\otimes L \ar[r] \ar[r] \ar[r] \ar[d] \ar[d] \ar[d] &B\widehat{\otimes}Z\left\{\frac{g}{f}\right\}\otimes M \ar[r] \ar[r] \ar[r] \ar[d] \ar[d] \ar[d] &0.\\
0 \ar[r] \ar[r] \ar[r] &F_2 \ar[r] \ar[r] \ar[r] &L_2 \ar[r] \ar[r] \ar[r] &M_2 \ar[r] \ar[r] \ar[r] &0.\\
}
\]
What we know is that the corresponding middle vertical arrow is isomorphism, while the left and right ones are surjective, which implies the right most one is an isomorphism. Finally to show that the module $M$ is pseudocoherent we consider the corresponding induction on the number of the finite projective modules participated in the resolution. $k=1$ will be obvious, while in general the induction step will achieved from $k$ to $k+1$ by considering as in \cite[Lemma 1.9.11]{Ked2} the module $F$ and apply the corresponding construction in the same fashion.
\end{proof}

\

\begin{proposition} \mbox{\bf{(After Kedlaya \cite[Lemma 1.9.13]{Ked2})}}
Take the corresponding parallel assumption in \cite[1.7.1]{Ked2} and furthermore assume the sheafiness of $A$. Suppose we have that all the modules considered are $Z$-stably complete with respect to nice rational localization. Then consider $C_1$ the category of all such modules which are assumed to be pseudocoherent over $B\widehat{\otimes}Z\left\{\frac{f}{g},\frac{g}{f}\right\}$. Then consider $C_2$ the category of all such modules which are assumed to be pseudocoherent over $B\widehat{\otimes}Z\left\{\frac{f}{g}\right\}$. Then consider $C_3$ the category of all such modules which are assumed to be pseudocoherent over $B\widehat{\otimes}Z\left\{\frac{g}{f}\right\}$. Then consider $C_4$ the category of all such modules which are assumed to be pseudocoherent over $B\widehat{\otimes}Z$. Then we have a corresponding equalizer functor:
\begin{align}
C_2\times_{C_1}C_3\rightarrow C	
\end{align}
which is exact and fully faithful, and we have well-established inverse.
\end{proposition}

\begin{proof}
See \cite[Lemma 1.9.13]{Ked2}.	
\end{proof}

\begin{proposition}  \mbox{\bf{(After Kedlaya \cite[Lemma 1.9.16]{Ked2})}}
Take the corresponding parallel assumption in \cite[1.7.1]{Ked2} and furthermore assume the sheafiness of the adic Banach $A$. Suppose we have that all the modules considered are $Z$-stably complete with respect to nice rational localization. Let $C$ be the category of such $Z$-nice-stably pseudocoherent modules over $A$ and let $C'$ be the category of such $Z$-nice-stably pseudocoherent sheaves over $X$. If we consider the corresponding functor taking the form of the corresponding sheafification of $\widetilde{M}$	with respect to the corresponding nice rational subspaces. Then we have that in fact this realizes an exact equivalence between $C$ and $C'$.
\end{proposition}

\begin{proof}
See \cite[Lemma 1.9.16]{Ked2}.	
\end{proof}

\begin{proposition}  \mbox{\bf{(After Kedlaya \cite[Lemma 1.9.17]{Ked2})}}
Take the corresponding parallel assumption in \cite[1.7.1]{Ked2} and furthermore assume the sheafiness of the adic Banach ring $A$. Suppose we have that all the modules considered are $Z$-stably complete with respect to nice rational localization. Let $C$ be the category of such $Z$-nice-stably pseudocoherent modules over $A$ and let $C'$ be the category of such $Z$-nice-stably pseudocoherent modules over $B$. Then we have that in fact this realizes an exact functor from $C$ to $C'$.
\end{proposition}

\begin{proof}
See \cite[Lemma 1.9.17]{Ked2}.	
\end{proof}

\begin{proposition}  \mbox{\bf{(After Kedlaya \cite[Corollary 1.9.18]{Ked2})}}
Take the corresponding parallel assumption in \cite[1.7.1]{Ked2} and furthermore assume the sheafiness of the adic Banach ring $A$. Let $C$ be the category of such $Z$-stably pseudocoherent modules over $A$ and let $C'$ be the category of such $Z$-nice-stably pseudocoherent modules over $A$. Then we have that in fact this realizes an equality from $C\subset C'$.
\end{proposition}

\begin{proof}
See \cite[Corollary 1.9.18]{Ked2}.	
\end{proof}

\begin{theorem} \mbox{\bf{(After Kedlaya \cite[Theorem 1.4.14, Theorem 1.4.16, Theorem 1.4.18]{Ked2})}}
Take the corresponding parallel assumption in \cite[1.7.1]{Ked2} and furthermore assume the sheafiness of the adic Baanch ring $A$. Then: I. we have the corresponding stability under rational localization for $Z$-stably pseudocoherent modules. II. We have the corresponding acyclicity of the corresponding presheaf $\widetilde{M}$ attached to any $Z$-stably pseudocoherent module. III. We have the corresponding glueing result as in \cite[Theorem 1.4.18]{Ked2}.
\end{theorem}

\begin{proof}
This is a very parallel $Z$-relative analog of \cite[Theorem 1.4.14, Theorem 1.4.16, Theorem 1.4.18]{Ked2}, we refer the readers to the proof of \cite[Theorem 1.4.14, Theorem 1.4.16, Theorem 1.4.18]{Ked2}. See \cite[After Corollary 1.9.18, the proof of Theorem 1.4.14, the proof of Theorem 1.4.16, the proof of Theorem 1.4.18]{Ked2}.	
\end{proof}

\

\subsection{Noncommutative Deformation over Quasi-Stein Spaces}

\indent We now consider the corresponding analytic quasi-Stein spaces over $\mathbb{Z}_p$, $Z$ will be an adic Banach ring as above:

\begin{setting} \mbox{\bf{(After Kedlaya-Liu \cite[Definition 2.6.2]{KL2})}}
Fix an adic Banach space $X$ over $\mathbb{Z}_p$ which is assumed to be an injective limit of analytic adic affinoids $X=\varinjlim_i X_i$, where we assume that the corresponding transition map $\mathcal{O}_{X_{i+1}}\rightarrow \mathcal{O}_{X_{i}}$ takes dense images for all $i=1,2,...$.	
\end{setting}

\begin{lemma} \mbox{\bf{(After Kedlaya-Liu \cite[Lemma 2.6.3]{KL2})}} \label{lemma6.24}
We now consider the corresponding rings $H_i:=\mathcal{O}_{X_i}$ for all $i=0,1,...$, and in our current situation we consider the corresponding rings $H_i\widehat{\otimes}Z$ (over $\mathbb{Z}_p$) for all $i=0,1,...$. And in our situation we consider the corresponding modules $M_i$ over $H_i\widehat{\otimes}Z$ for all $i=0,1,...$ with the same requirement as in \cite[Lemma 2.6.3]{KL2} (namely those complete with respect to the natural topology). Suppose that we have bounded surjective map from $f_i:H_{i}\widehat{\otimes}Z\otimes_{H_{i+1}\widehat{\otimes}Z} M_{i+1}\rightarrow M_i,i=0,1,...$. Then we have first the density of the corresponding image of $\varprojlim_i M_i$ in each $M_i$ for any $i=0,1,2,...$. And we have as well the corresponding vanishing of $R^1\varprojlim_i M_i$.
\end{lemma}

\begin{proof}
This is the $Z$-relative version of the result in \cite[Lemma 2.6.3]{KL2}. For the first statement we just choose sequence of Banach norms on all the corresponding modules for all $i=0,1,...$ such that we have $\|f_i(x_{i+1})\|_i\leq 1/2\|x_{i+1}\|_{i+1}$ for any $x_{i+1}\in M_{i+1}$. Then for any $x_i\in M_i$ and any $\delta>0$, we now consider for any $j\geq 1$ the corresponding $x_{i+j}$ such that we have $\|x_{i+j}-f_{i+j+1}(x_{i+j+1})\|_{i+j}\leq \delta$. Then the sequence $x_{i+j+k},k=0,1,...$ will converge to some well-defined $y_{i+j}$ with in our situation the corresponding $y_{i}=f_{i}(y_{i+1})$. We then have $\|x_i-y_i\|_i\leq \delta$. This will prove the first statement. For the second statement as in \cite[Lemma 2.6.3]{KL2} we form the product $M_0\times M_1\times M_2\times...$ and the consider the induced map $F$ from $M_{i+1}\rightarrow M_i$, and consider the corresponding cokernel of the map $1-F$ since this is just the corresponding limit we are considering. Then to show that the cokernel is zero we just look at the corresponding cokernel of the corresponding map on the corresponding completed direct summand which will project to the original one. But then we will have $\|f_i(v)\|_i\leq 1/2 \|v\|_i$, which produces an inverse to $1-F$ which will basically finish the proof for the second statement. 	
\end{proof}

\begin{proposition} \mbox{\bf{(After Kedlaya-Liu \cite[Lemma 2.6.4]{KL2})}} In the same situation as above, suppose we have that the corresponding modules $M_i$ are basically $Z$-stably pseudocoherent over the rings $H_i\widehat{\otimes}Z$ for all $i=0,1,...$. Now we consider the situation where $f_i:H_i\widehat{\otimes}Z\otimes_{H_{i+1}\widehat{\otimes}Z}M_{i+1}\rightarrow M_i$ is an isomorphism. Then the conclusion in our situation is then that the corresponding projection from $\varprojlim M_i$ to $M_i$ for each $i=0,1,2,...$ is an isomorphism.

\end{proposition}

\begin{proof}
This is a $Z$-relative and analytic version of the \cite[Lemma 2.6.4]{KL2}. We adapt the argument to our situation as in the following. First we choose some finite free covering $T$ of the limit $M$ such that we have for each $i$ the corresponding map $T_i\rightarrow M_i$ is surjective. Then we consider the index $j\geq i$ and set the kernel of the map from $T_j$ to $M_j$ to be $S_j$. By the direct analog of \cite[Lemma 2.5.6]{KL2} we have that $H_i\widehat{\otimes}Z\otimes S_j\overset{}{\rightarrow}S_i$ realizes the isomorphism, and we have that the corresponding surjectivity of the corresponding map from $\varprojlim_i S_i$ projecting to the $S_i$. Then one could finish the proof by 5-lemma to the following commutative diagram as in \cite[Lemma 2.6.4]{KL2}:
\[ \tiny
\xymatrix@C+0pc@R+6pc{
 &(H_i\widehat{\otimes}Z)\otimes \varprojlim_i S_i \ar[r]\ar[r]\ar[r] \ar[d]\ar[d]\ar[d] &(H_i\widehat{\otimes}Z)\otimes \varprojlim_i F_i \ar[r]\ar[r]\ar[r] \ar[d]\ar[d]\ar[d] &(H_i\widehat{\otimes}Z)\otimes \varprojlim_i M_i\ar[r]\ar[r]\ar[r]  \ar[d]\ar[d]\ar[d]&0,\\
0 \ar[r]\ar[r]\ar[r] &S_i   \ar[r]\ar[r]\ar[r] &F_i \ar[r]\ar[r]\ar[r] &M_i \ar[r]\ar[r]\ar[r] &0.\\
}
\]
 
\end{proof}

\begin{proposition}  \mbox{\bf{(After Kedlaya-Liu \cite[Theorem 2.6.5]{KL2})}}
For any quasi-compact adic affinoid space of $X$ which is denoted by $Y$, we have that the map $\mathcal{M}(X)\rightarrow \mathcal{M}(Y)$ is surjective for any $Z$-stably pseudocoherent sheaf $\mathcal{M}$ over the sheaf $\mathcal{O}_X\widehat{\otimes}Z$.	
\end{proposition}

\begin{proof}
This is just the corresponding corollary of the previous proposition.	
\end{proof}

\begin{proposition}  \mbox{\bf{(After Kedlaya-Liu \cite[Theorem 2.6.5]{KL2})}}
We have that the stalk $\mathcal{M}_x$ is generated over the stalk $\mathcal{O}_{X,x}$ for any $x\in X$ by $M(X)$, for any $Z$-stably pseudocoherent sheaf $\mathcal{M}$ over the sheaf $\mathcal{O}_X\widehat{\otimes}Z$.	
\end{proposition}

\begin{proof}
This is just the corresponding corollary of the proposition before the previous proposition.	
\end{proof}

\begin{proposition}  \mbox{\bf{(After Kedlaya-Liu \cite[Theorem 2.6.5]{KL2})}} \label{proposition6.28}
For any quasi-compact adic affinoid space of $X$ which is denoted by $Y$, we have that the corresponding vanishing of the corresponding sheaf cohomology groups $H^k(X,\mathcal{M})$ of $\mathcal{M}$ for higher $k>0$, for any $Z$-stably pseudocoherent sheaf $\mathcal{M}$ over the sheaf $\mathcal{O}_X\widehat{\otimes}Z$.	
\end{proposition}

\begin{proof}
We follow the idea of the proof of \cite[Theorem 2.6.5]{KL2} by comparing this to the corresponding \v{C}ech cohomology with some covering $\mathfrak{X}=\{X_1,...,X_N,...\}$:
\begin{align}
\breve{H}^k(X,\mathfrak{X}=\{X_1,...,X_N,...\};\mathcal{M})=H^k(X,\mathcal{M}),\\
\breve{H}^k(X_i,\mathfrak{X}=\{X_1,...,X_i\};\mathcal{M})=H^k(X_i,\mathcal{M})=0,k\geq 1.
\end{align}
Here we have applied the corresponding \cite[Tag 01EW]{SP}. Now we consider the situation where $k>1$:
\begin{align}
\breve{H}^{k-1}(X_{j+1},\mathfrak{X}=\{X_1,...,X_{j+1}\};\mathcal{M})	\rightarrow \breve{H}^{k-1}(X_j,\mathfrak{X}=\{X_1,...,X_{j}\};\mathcal{M})\rightarrow 0,
\end{align}
which induces the following isomorphism by \cite[2.6 Hilfssatz]{Kie1}:
\begin{align}
\varprojlim_{j\rightarrow \infty} \breve{H}^{k-1}(X_j,\mathfrak{X}=\{X_1,...,X_{j}\};\mathcal{M})	\overset{\sim}{\rightarrow} \breve{H}^k(X,\mathfrak{X}=\{X_1,...,X_N,...\};\mathcal{M}). 
\end{align}
Then we have the corresponding results for the index $k>1$. For $k=1$, one can as in \cite[Theorem 2.6.5]{KL2} relate this to the corresponding $R^1\varprojlim_i$, which will finishes the proof in the same fashion.
\end{proof}

\

\begin{corollary}  \mbox{\bf{(After Kedlaya-Liu \cite[Corollary 2.6.6]{KL2})}}
The corresponding functor from the corresponding $Z$-deformed pseudocoherent sheaves over $X$ to the corresponding $Z$-stably by taking the corresponding global section is an exact functor.
\end{corollary}

\begin{corollary}  \mbox{\bf{(After Kedlaya-Liu \cite[Corollary 2.6.8]{KL2})}}
Consider a particular $\mathcal{O}_X\widehat{\otimes}Z$-pseudocoherent sheaf $\mathcal{M}$ which is finite locally free throughout the whole space $X$. Then we have that the global section $\mathcal{M}(X)$ as $\mathcal{O}_X(X)\widehat{\otimes}Z$ left module admits the corresponding structures of finite projective structure if and only if we have the corresponding global section is finitely generated.
\end{corollary}

\begin{proof}
As in \cite[Corollary 2.6.8]{KL2} one could find some global splitting through the local splittings.
\end{proof}

\

\indent We now consider the following $Z$-relative and analytic analog of \cite[Proposition 2.6.17]{KL2}:

\begin{theorem} \mbox{\bf{(After Kedlaya-Liu \cite[Proposition 2.6.17]{KL2})}} \label{theorem6.31}
Consider the following two statements for a particular $\mathcal{O}_X\widehat{\otimes}Z$-pseudocoherent sheaf $\mathcal{M}$. First is that one can find finite many generators (the number is up to some uniform fixed integer $n\geq 0$) for each section of $\mathcal{M}(X_i)$ for each $i=1,2,...$. The second statement is that the global section $\mathcal{M}(X)$ is just finitely generated. Then in our situation the two statement is equivalent if we have that the corresponding space $X$ admits an $m$-uniform covering in the exact same sense of \cite[Proposition 2.6.17]{KL2}. 	
\end{theorem}

\begin{proof}
One direction is obvious, the other direct could be proved in the same way as in \cite[Proposition 2.6.17]{KL2} where essentially the information on $X$ does not change at all. Namely as in \cite[Proposition 2.6.17]{KL2} we could basically consider one single subcovering $\{Y_u\}_{u\in U}$ indexed by $U$ of the covering which is $m$-uniform. For any $u\in U$ we consider the smallest $i$ such that $X_i\bigcap Y_u$, and we denote this by $i(u)\geq 0$. Then we form:
\begin{align}
V_{u}:=X_{i(u)}\bigcup_{v\in U,v\leq u} Y_{v}.	
\end{align}
Then as in \cite[Proposition 2.6.17]{KL2} we can find some $x_u,u\in U$ (which are denoted by the general form of $x_i$ in \cite[Proposition 2.6.17]{KL2}) in the section $\mathcal{O}_X(V_{u})$ such that it restricts to some $0$ onto the space $V_u$ and it restricts to some element $a_1$ such that we have:
\begin{align}
a_1b_1+...+a_kb_k=1	
\end{align}
where $b_1,...,b_k$ are the corresponding topologically nilpotent elements to $Y_u$. By the corresponding density of the corresponding global section $\mathcal{O}_X(X)$ in $\mathcal{O}_X(V_{u})$ we could find some $x_i$ (with the same notation) in the global section $\mathcal{O}_X(X)$ such that it restricts to some topologically nilpotent element onto the space $V_u$ and it restricts to some element $a'_1$ such that we have:
\begin{align}
a'_1b'_1+...+a'_kb'_k=1	
\end{align}
where $b_1,...,b_k$ are the corresponding topologically nilpotent elements to $Y_u$, by simultaneously approximating the corresponding the coefficient $a_1,...,a_k$ and $b_1,...,b_k$ to achieve $a'_1,...,a'_k$ and $b'_1,...,b'_k$ such that we have:
\begin{align}
a'_1b'_1+...+a'_kb'_k=1.	
\end{align}
Then we can build up a corresponding generating set for the global section by the following approximation process just as in \cite[Proposition 2.6.17]{KL2}. To be more precise what we have consider is to modify the corresponding generator for each $Y_u$ (for instance let them be denoted by $y_1,...,y_u$) to be $y_{u,1},...,y_{u,n}$ ($n$ is the corresponding uniform integer in the statement of the theorem) for all $u\in U$. It is achieved through induction, when we have $y_{u,1},...,y_{u,n}$ then we will set that to be $0$ otherwise in the situation where there exists some predecessor $u'$ of $u$ in the corresponding set $U$ we then set this to be $y_{u',1},...,y_{u',n}$. And we set as in \cite[Proposition 2.6.17]{KL2}:
\begin{displaymath}
y_{u,j}:=	y_{u-1,j}+x_u^cy_u
\end{displaymath}
by lifting the corresponding power $c$ to be as large as possible. This will guarantee the convergence of:
\begin{align}
\lim_{u\rightarrow \infty}\{y_{u,1},...,y_{u,n}\},	
\end{align}
which gives the set of the corresponding global generators desired.
\end{proof}

\newpage

\section{Extensions and Applications}

\subsection{Deforming the $(\infty,1)$-Sheaves over Bambozzi-Kremnizer $\infty$-Analytic Spaces}

\noindent Here we follow \cite{BK1} to extend our discussion above partially to the corresponding $\infty$-analytic space attached to a Banach adic uniform algebra $A$ (required to be commutative) over $\mathbb{F}_p((t))$ or $\mathbb{Q}_p$ and we consider furthermore a Banach algebra $B$ over $\mathbb{F}_p((t))$ or $\mathbb{Q}_p$.  

\begin{remark}
We should definitely mention that one can apply the corresponding foundation from Clausen-Scholze \cite{CS} to do this, namely the corresponding animated objects. But we will only apply the corresponding foundation in \cite{BK1}. 
\end{remark}

\indent First we have the following $\infty$-version of Tate-\v{C}ech acyclicity in our context:

\begin{proposition}\mbox{\bf{(After Bambozzi-Kremnizer \cite[Theorem 4.15]{BK1})}} 
The following totalization is strictly exact:
\[
\xymatrix@C+0pc@R+0pc{
\mathrm{Tot}(0  \ar[r]\ar[r]\ar[r] &A^h\widehat{\otimes} B \ar[r]\ar[r]\ar[r] & \prod_i A_i\widehat{\otimes} B \ar[r]\ar[r]\ar[r] &\prod_{i,j} A_i\widehat{\otimes} B \widehat{\otimes}^\mathbb{L} A_j\widehat{\otimes}B\ar[r]\ar[r]\ar[r] &...)
}
\]
for any homotopy Zariski covering $\{\mathrm{Spec}A_i\rightarrow \mathrm{Spec}A^h\}_{i\in I}$ in the sense of \cite[Theorem 4.15]{BK1}. Here we have already attached the corresponding homotopical ring $A^h$ to $A$.
\end{proposition}

\begin{proof}
This is a direct consequence of the corresponding result in \cite[Theorem 4.15]{BK1}.	
\end{proof}

\indent Along this idea one can then build up a sheaf of $\infty$-Banach ring over Bambozzi-Kremnizer's spectrum $\mathrm{Spa}^h(A)$, which we will denote it by $\mathcal{O}_{\mathrm{Spa}^h(A),B}$. Namely we have defined a new $\infty$-ringed space $(\mathrm{Spa}^h(A),\mathcal{O}_{\mathrm{Spa}^h(A),B})$.

\begin{definition} \mbox{\bf{(After Lurie \cite[Definition 2.9.1.1]{Lu1})}} We will call a finite projective module spectrum $M$ over the ring $A^h\widehat{\otimes}B$ finite locally free, namely it will be basically sit in some finite coproduct of copies of $A^h\widehat{\otimes}B$ as a direct summand. Note that one can obviously construct a retract from some finite free module spectrum to such object, which is surjective namely with respect to the homotopy group $\pi_0$.
	
\end{definition}

\indent As in \cite[Proposition 7.2.4.20]{Lu2}, we have that finite locally free sheaves are actually flat and pseudocoherent in the strong $\infty$-sense. This will cause the following to be true:

\begin{proposition} 
The following totalization is strictly exact:
\[
\xymatrix@C+0pc@R+0pc{
\mathrm{Tot}(0  \ar[r]\ar[r]\ar[r] &M \ar[r]\ar[r]\ar[r] & \prod_i A_i\widehat{\otimes} B \widehat{\otimes}^\mathbb{L} M  \ar[r]\ar[r]\ar[r] &\prod_{i,j} (A_i\widehat{\otimes} B \widehat{\otimes}^\mathbb{L} A_j\widehat{\otimes}B) \widehat{\otimes}^\mathbb{L} M\ar[r]\ar[r]\ar[r] &...)
}
\]
for any homotopy Zariski covering $\{\mathrm{Spec}A_i\rightarrow \mathrm{Spec}A^h\}_{i\in I}$ in the sense of \cite[Theorem 2.15, Theorem 4.15]{BK1}. 
\end{proposition}

\begin{theorem} \mbox{\bf{(After Kedlaya-Liu \cite[Theorem 2.7.7]{KL1})}}
Carrying the corresponding coefficient in noncommutative Banach ring $B$, we have that $\infty$-descent for locally free finitely generated module spectra with respect to simple Laurent covering.	
\end{theorem}

\begin{proof} 
\indent Certainly along the idea of \cite[Theorem 2.7.7]{KL1} one can actually try to consider the corresponding glueing of finite projective objects, but this will basically involve a derived or $\infty$ version of \cite[Proposition 2.4.20]{KL1}. Instead we consider the following restricted situation which is also considered restrictively in \cite[Section 5]{T3}. First we consider the corresponding glueing datum along $A$, namely suppose we have the corresponding short exact sequence of Banach adic uniform algebra over $\mathbb{Q}_p$:
\[
\xymatrix@C+0pc@R+0pc{
0   \ar[r]\ar[r]\ar[r] &\Pi \ar[r]\ar[r]\ar[r] &\Pi_1\bigoplus \Pi_2 \ar[r]\ar[r]\ar[r] & \Pi_{12} \ar[r]\ar[r]\ar[r] &0,
}
\]
which is assumed to satisfy conditions $(a)$ and $(b)$ of \cite[Definition 2.7.3]{KL1}, namely this is strictly exact and the map $\Pi_1\rightarrow \Pi_{12}$ is of image that is dense. Now we consider the Bambozzi-Kremnizer spaces and rings:
\begin{align}
&\mathrm{Spa}^h(\Pi),\mathrm{Spa}^h(\Pi_1),\mathrm{Spa}^h(\Pi_2),\mathrm{Spa}^h(\Pi_{12}),\\
&\Pi^h,\Pi_1^h,\Pi_2^h,\Pi_{12}^h.  	
\end{align}
Now consider the following finite locally free bimodule spectra forming the corresponding glueing datum:
\begin{align}
M_1,M_2,M_{12}	
\end{align}
over:
\begin{align}
\Pi_1^h\widehat{\otimes}B,\Pi_2^h\widehat{\otimes}B,\Pi_{12}^h\widehat{\otimes}B.
\end{align}
Then we can regard these as certain sheaves of bimodules:
\begin{align}
\mathcal{M}_1,\mathcal{M}_2,\mathcal{M}_{12}	
\end{align}
over the $\infty$-sheaves:
\begin{align}
\mathcal{O}_{\Pi_1^h}\widehat{\otimes}B,\mathcal{O}_{\Pi_2^h}\widehat{\otimes}B,\mathcal{O}_{\Pi_{12}^h}\widehat{\otimes}B.
\end{align}
Then we can have the chance to take the global section along the space $\mathrm{Spa}^h(\Pi)$ which gives rise to a certain module spectrum $M$ over $\Pi^h$. However the corresponding $\pi_0(M)$ will then be finite projective over $\pi_0(\Pi)$ by \cite[Proposition 5.12]{T3}. This finishes the corresponding glueing of vector bundle process in this specific situation.
\end{proof}

\begin{remark}
Again one can apply the whole machinery from \cite{CS} to achieve this, as long as one would like to work with noncommutative analogs of the animated rings.	
\end{remark}

%

\

\indent Now we consider $\mathbb{Q}_p[[G]]$-adic objects in the situation over $\mathbb{Q}_p$. First we have the following $\infty$-version of Tate-\v{C}ech acyclicity in our context:

\begin{proposition}\mbox{\bf{(After Bambozzi-Kremnizer \cite[Theorem 4.15]{BK1})}} 
The following totalization is strictly exact:
\[
\xymatrix@C+0pc@R+0pc{
\mathrm{Tot}(0  \ar[r]\ar[r]\ar[r] &A^h\widehat{\otimes} B[[G]]/E \ar[r]\ar[r]\ar[r] & \prod_i A_i\widehat{\otimes} B[[G]]/E \ar[r]\ar[r]\ar[r] &\prod_{i,j} A_i\widehat{\otimes} B[[G]]/E \widehat{\otimes} ^\mathbb{L} A_j\widehat{\otimes}B[[G]]/E \ar[r]\ar[r]\ar[r] &...),
}
\]
for any $E\subset G$, and for any homotopy Zariski covering $\{\mathrm{Spec}A_i\rightarrow \mathrm{Spec}A^h\}_{i\in I}$ in the sense of \cite[Theorem 2.15, Theorem 4.15]{BK1}. Here we have already attached the corresponding homotopical ring $A^h$ to $A$.
\end{proposition}

\begin{proof}
This is a direct consequence of the corresponding result in \cite[Theorem 4.15]{BK1}.	
\end{proof}

\indent Along this idea one can then build up a sheaf of $\infty$-Banach ring over Bambozzi-Kremnizer's spectrum $\mathrm{Spa}^h(A)$, which we will denote it by $\mathcal{O}_{\mathrm{Spa}^h(A),B[[G]]}$. Namely we have defined a new $\infty$-ringed space $(\mathrm{Spa}^h(A),\mathcal{O}_{\mathrm{Spa}^h(A),B[[G]]})$.

\begin{definition} \mbox{\bf{(After Lurie \cite[Definition 2.9.1.1]{Lu1})}} We will call a finite projective module spectrum $M$ over the ring $A^h\widehat{\otimes}B[[G]]$ finite locally free, namely it will be basically sit in some finite coproduct of copies of $A^h\widehat{\otimes}B[[G]]$ as a direct summand. Note that one can obviously construct a retract from some finite free module spectrum to such object.
	
\end{definition}

\indent As in \cite[Proposition 7.2.4.20]{Lu2}, we have that finite locally free sheaves are actually flat and pseudocoherent in the strong $\infty$-sense. This will cause the following to be true:

\begin{proposition} 
The following totalization is strictly exact:
\[
\xymatrix@C+0pc@R+0pc{
\mathrm{Tot}(0  \ar[r]\ar[r]\ar[r] &M \ar[r]\ar[r]\ar[r] & \prod_i A_i\widehat{\otimes} B[[G]]/E \widehat{\otimes}^\mathbb{L} M  \ar[r]\ar[r]\ar[r] &\prod_{i,j} (A_i\widehat{\otimes} B[[G]]/E \otimes^\mathbb{L} A_j\widehat{\otimes} B[[G]]/E) \otimes^\mathbb{L} M\ar[r]\ar[r]\ar[r] &...)
}
\]
for any homotopy Zariski covering $\{\mathrm{Spec}A_i\rightarrow \mathrm{Spec}A^h\}_{i\in I}$ in the sense of \cite[Theorem 4.15, Theorem 2.15]{BK1}. 
\end{proposition}

%
%

\

 \indent Now we consider the LF deformation, namely here $B$ could be written as $\varinjlim_h B_h$. First we have the following $\infty$-version of Tate-\v{C}ech acyclicity in our context:

\begin{proposition}\mbox{\bf{(After Bambozzi-Kremnizer \cite[Theorem 4.15]{BK1})}} 
The following totalization is strictly exact:
\[
\xymatrix@C+0pc@R+0pc{
\mathrm{Tot}(0  \ar[r]\ar[r]\ar[r] &A^h\widehat{\otimes} B \ar[r]\ar[r]\ar[r] & \prod_i A_i\widehat{\otimes} B \ar[r]\ar[r]\ar[r] &\prod_{i,j} A_i\widehat{\otimes} B \widehat{\otimes}^\mathbb{L} A_j\widehat{\otimes}B\ar[r]\ar[r]\ar[r] &...)
}
\]
for any homotopy Zariski covering $\{\mathrm{Spec}A_i\rightarrow \mathrm{Spec}A^h\}_{i\in I}$ in the sense of \cite[Theorem 4.15]{BK1}. Here we have already attached the corresponding homotopical ring $A^h$ to $A$.
\end{proposition}

\begin{proof}
This is a direct consequence of the corresponding result in \cite[Theorem 4.15]{BK1}.	
\end{proof}

\indent Along this idea one can then build up a sheaf of $\infty$-Banach ring over Bambozzi-Kremnizer's spectrum $\mathrm{Spa}^h(A)$, which we will denote it by $\mathcal{O}_{\mathrm{Spa}^h(A),B}$. Namely we have defined a new $\infty$-ringed space $(\mathrm{Spa}^h(A),\mathcal{O}_{\mathrm{Spa}^h(A),B})$.

\begin{definition} \mbox{\bf{(After Lurie \cite[Definition 2.9.1.1]{Lu1})}} We will call a finite projective module spectrum $M$ over the ring $A^h\widehat{\otimes}B$ finite locally free, namely it will be basically sit in some finite coproduct of copies of $A^h\widehat{\otimes}B$ as a direct summand. Note that one can obviously construct a retract from some finite free module spectrum to such object.
	
\end{definition}

\indent As in \cite[Proposition 7.2.4.20]{Lu2}, we have that finite locally free sheaves are actually flat and pseudocoherent in the strong $\infty$-sense. This will cause the following to be true:

\begin{proposition} 
The following totalization is strictly exact:
\[
\xymatrix@C+0pc@R+0pc{
\mathrm{Tot}(0  \ar[r]\ar[r]\ar[r] &M \ar[r]\ar[r]\ar[r] & \prod_i A_i\widehat{\otimes} B \widehat{\otimes}^\mathbb{L} M  \ar[r]\ar[r]\ar[r] &\prod_{i,j} (A_i\widehat{\otimes} B \widehat{\otimes}^\mathbb{L} A_j\widehat{\otimes}B) \widehat{\otimes}^\mathbb{L} M\ar[r]\ar[r]\ar[r] &...)
}
\]
for any homotopy Zariski covering $\{\mathrm{Spec}A_i\rightarrow \mathrm{Spec}A^h\}_{i\in I}$ in the sense of \cite[Theorem 2.15, Theorem 4.15]{BK1}. 
\end{proposition}

\begin{theorem} \mbox{\bf{(After Kedlaya-Liu \cite[Theorem 2.7.7]{KL1})}}
Carrying the corresponding coefficient in the noncommutative Limit of Fr\'echet $B$, we have that $\infty$-descent for locally free finitely generated module spectra with respect to simple Laurent covering.	
\end{theorem}

\begin{proof} 
\indent Along the discussion in the Banach situation, we can also make the parallel discussion in the corresponding limit of Fr\'echet situation (LF). Certainly along the idea of \cite[Theorem 2.7.7]{KL1} one can actually try to consider the corresponding glueing of finite projective objects, but this will basically involve a derived or $\infty$ version of \cite[Proposition 2.4.20]{KL1}. Instead we consider the following restricted situation which is also considered restrictively in \cite[Section 5]{T3}. First we consider the corresponding glueing datum along $A$, namely suppose we have the corresponding short exact sequence of Banach adic uniform algebra over $\mathbb{Q}_p$:
\[
\xymatrix@C+0pc@R+0pc{
0   \ar[r]\ar[r]\ar[r] &\Pi \ar[r]\ar[r]\ar[r] &\Pi_1\bigoplus \Pi_2 \ar[r]\ar[r]\ar[r] & \Pi_{12} \ar[r]\ar[r]\ar[r] &0,
}
\]
which is assumed to satisfy conditions $(a)$ and $(b)$ of \cite[Definition 2.7.3]{KL1}, namely this is strictly exact and the map $\Pi_1\rightarrow \Pi_{12}$ is of image that is dense. Now we consider the Bambozzi-Kremnizer spaces and rings:
\begin{align}
&\mathrm{Spa}^h(\Pi),\mathrm{Spa}^h(\Pi_1),\mathrm{Spa}^h(\Pi_2),\mathrm{Spa}^h(\Pi_{12}),\\
&\Pi^h,\Pi_1^h,\Pi_2^h,\Pi_{12}^h.  	
\end{align}
Now consider the following finite locally free bimodule spectra forming the corresponding glueing datum:
\begin{align}
M_1,M_2,M_{12}	
\end{align}
over:
\begin{align}
\Pi_1^h\widehat{\otimes}\varinjlim_h B_h,\Pi_2^h\widehat{\otimes}\varinjlim_h B_h,\Pi_{12}^h\widehat{\otimes}\varinjlim_h B_h.
\end{align}
Then we can regard these as certain sheaves of bimodules:
\begin{align}
\mathcal{M}_1,\mathcal{M}_2,\mathcal{M}_{12}	
\end{align}
over the $\infty$-sheaves:
\begin{align}
\mathcal{O}_{\Pi_1^h}\widehat{\otimes}\varinjlim_h B_h,\mathcal{O}_{\Pi_2^h}\widehat{\otimes}\varinjlim_h B_h,\mathcal{O}_{\Pi_{12}^h}\widehat{\otimes}\varinjlim_h B_h.
\end{align}
Then we can have the chance to take the global section along the space $\mathrm{Spa}^h(\Pi)$ which gives rise to a certain module spectrum $M$ over $\Pi^h$. However the corresponding $\pi_0(M)$ will then be finite projective over $\pi_0(\Pi)$ by \cite[Proposition 5.12]{T3}. This finishes the corresponding glueing of vector bundle process in this specific situation.
\end{proof}

\begin{remark}
Again one can apply the whole machinery from \cite{CS} to achieve this, as long as one would like to work with noncommutative analogs of the animated rings.	
\end{remark}

\
%
%

\indent The above discussion could be made to be more general by looking at some derived rational localization $A\rightarrow D$ of $A$ as in \cite{BK1} which we will denote by $D\in \infty-\mathbf{Comm}$.  For instance one can just take $D$ to be some derived rational localization taking the form such as $A\left<\frac{a_0}{a_0},...,\frac{a_k}{a_0}\right>^h$ (and consider the glueing over $\mathrm{Spec}A\left<\frac{a_0}{a_0},...,\frac{a_k}{a_0}\right>^h\in \infty-\mathbf{Comm}^\mathrm{op}$). We now assume the following assumption:

\begin{assumption}
Without any sheafiness condition on the ring $A$ in the corresponding classical sense, we assume the following totolization is strictly exact:
\[
\xymatrix@C+0pc@R+0pc{
\mathrm{Tot}(0  \ar[r]\ar[r]\ar[r] &D\widehat{\otimes} B \ar[r]\ar[r]\ar[r] & \prod_i D_i\widehat{\otimes} B \ar[r]\ar[r]\ar[r] &\prod_{i,j} D_i\widehat{\otimes} B \widehat{\otimes}^\mathbb{L} D_j\widehat{\otimes}B\ar[r]\ar[r]\ar[r] &...)
}
\]	
for any homotopy Zariski covering $\{\mathrm{Spec}D_i\rightarrow \mathrm{Spec}D\}_{i\in I}$ in the sense of \cite[Theorem 4.15]{BK1}. Here recall the notation $\mathrm{Spec}D$ means the corresponding object in $\infty-\mathbf{Comm}^\mathrm{op}$. Here $B$ is just assumed to be Banach.
\end{assumption}

\begin{remark}
\indent By the corresponding Dold-Kan correspondence we could regard $D$ as an $(\infty,1)$-commutative object in the corresponding category of all the simplicial semi-normed monoids, namely an $E_\infty$-ring which carries the corresponding simplicial Banach structures. However then the corresponding ring $D\widehat{\otimes}B$ will then be a product of an $E_\infty$-ring with a discrete (in the sense of \cite[Chapter 7.2]{Lu2}) $E_1$-ring (or if you wish a discrete $A_\infty$-ring).
\end{remark}

\begin{definition} \mbox{\bf{(After Lurie \cite[Definition 2.9.1.1]{Lu1})}} We will call a finite projective module spectrum $M$ over the ring $D\widehat{\otimes}B$ finite locally free, namely it will be basically sit in some finite coproduct of copies of $D\widehat{\otimes}B$ as a direct summand. Note that one can obviously construct a retract from some finite free module spectrum to such object, which is surjective namely with respect to the homotopy group $\pi_0$.
	
\end{definition}

\begin{definition} \mbox{\bf{(After Lurie \cite[Definition 2.9.1.1]{Lu1})}} We will call a sheaf $M$ over the $\infty$-ringed Grothendieck site:
\begin{align}
(\mathrm{Spec}D,\mathcal{O}_\mathrm{Spec}D\widehat{\otimes}B,\{\mathrm{Spec}D\left<\frac{f_0}{f_i},...,\frac{f_d}{f_i}\right>^h\}_{i=0,...,d})	
\end{align}
locally finite free if it is locally (over some derived rational subset $\mathrm{Spec}D_i\in \infty-\mathbf{Comm}^\mathrm{op}$ and $D_i:=D\left<\frac{f_0}{f_i},...,\frac{f_d}{f_i}\right>^h,i=1,...,d$) sitting in some coproduct of finite many copies of $D_i\widehat{\otimes}B$ as a direct summand. 
	
\end{definition}

\indent As in \cite[Proposition 7.2.4.20]{Lu2}, we have that finite locally free sheaves are actually flat and pseudocoherent in the strong $\infty$-sense. This will cause the following to be true:

\begin{proposition} 
The following totalization is strictly exact:
\[
\xymatrix@C+0pc@R+0pc{
\mathrm{Tot}(0  \ar[r]\ar[r]\ar[r] &M \ar[r]\ar[r]\ar[r] & \prod_i D_i\widehat{\otimes} B \widehat{\otimes}^\mathbb{L} M  \ar[r]\ar[r]\ar[r] &\prod_{i,j} (D_i\widehat{\otimes} B \widehat{\otimes}^\mathbb{L} D_j\widehat{\otimes}B) \widehat{\otimes}^\mathbb{L} M\ar[r]\ar[r]\ar[r] &...)
}
\]
for any homotopy Zariski covering $\{\mathrm{Spec}D_i\rightarrow \mathrm{Spec}D\}_{i\in I}$ in the sense of \cite[Theorem 2.15, Theorem 4.15]{BK1}. 
\end{proposition}

\indent Now we take a corresponding fiber sequence (under the corresponding Dold-Kan correspondence one can also consider the corresponding distinguished triangles 
\[
\xymatrix@C+0pc@R+0pc{
&D \ar[r]\ar[r]\ar[r] &\prod_{i=1,2} D_i \ar[r]\ar[r]\ar[r] & D_{1}\widehat{\otimes}^\mathbb{L} D_{2} \ar[r]\ar[r]\ar[r] & D[1],
}
\]
for the corresponding associated complexes):
\[
\xymatrix@C+0pc@R+0pc{
&D \ar[r]\ar[r]\ar[r] &\prod_{i=1,2} D_i \ar[r]\ar[r]\ar[r] & D_{12},
}
\]
and assume that the following:
\[
\xymatrix@C+0pc@R+0pc{
0   \ar[r]\ar[r]\ar[r] &\pi_0D \ar[r]\ar[r]\ar[r] &\pi_0D_1\bigoplus \pi_0D_2 \ar[r]\ar[r]\ar[r] & \pi_0D_{12} \ar[r]\ar[r]\ar[r] &0
}
\]
to satisfy \cite[Definition 2.7.3, condition (a), (b)]{KL1}. Then we have the effective $\infty$-descentness of the corresponding morphism $D\widehat{\otimes}B\rightarrow D_1\widehat{\otimes}B\oplus D_2\widehat{\otimes}B$.

\

\begin{theorem} \mbox{\bf{(After Kedlaya-Liu \cite[Theorem 2.7.7]{KL1})}}
Carrying the corresponding coefficient in noncommutative Banach ring $B$, we have that $\infty$-descent for locally free finitely generated module spectra with respect to simple Laurent covering on the corresponding $\pi_0$.	
\end{theorem}

\begin{proof} 
\indent Certainly along the idea of \cite[Theorem 2.7.7]{KL1} one can actually try to consider the corresponding glueing of finite projective objects in more general setting, but this will basically involve a derived or $\infty$ version of \cite[Proposition 2.4.20]{KL1}. Instead we consider the following restricted situation which is also considered restrictively in \cite[Section 5]{T3}. First we consider the corresponding glueing datum along $\pi_0D$, namely suppose we have the corresponding short exact sequence of Banach adic uniform algebra over $\mathbb{Q}_p$:
\[
\xymatrix@C+0pc@R+0pc{
0   \ar[r]\ar[r]\ar[r] &\pi_0D \ar[r]\ar[r]\ar[r] &\pi_0D_1\bigoplus \pi_0D_2 \ar[r]\ar[r]\ar[r] & \pi_0D_{12} \ar[r]\ar[r]\ar[r] &0
}
\]
which is assumed to satisfy conditions $(a)$ and $(b)$ of \cite[Definition 2.7.3]{KL1}, namely this is strictly exact and the map $\Pi_1\rightarrow \Pi_{12}$ is of image that is dense. Now consider the following finite locally free bimodule spectra forming the corresponding glueing datum:
\begin{align}
M_1,M_2,M_{12}	
\end{align}
over:
\begin{align}
D_1\widehat{\otimes}B,D_2\widehat{\otimes}B,D_{12}\widehat{\otimes}B.
\end{align}
Then we can regard these as certain sheaves of bimodules:
\begin{align}
\mathcal{M}_1,\mathcal{M}_2,\mathcal{M}_{12}	
\end{align}
over the $\infty$-sheaves:
\begin{align}
\mathcal{O}_{\mathrm{Spec}D_1}\widehat{\otimes}B,\mathcal{O}_{\mathrm{Spec}D_2}\widehat{\otimes}B,\mathcal{O}_{\mathrm{Spec}D_{12}}\widehat{\otimes}B.
\end{align}
Then we can have the chance to take the global section along the space $\mathrm{Spec}D$ which gives rise to a certain module spectrum $M$ over $D$. However the corresponding $\pi_0(M)$ will then be finite projective over $\pi_0(\Pi)$ by \cite[Proposition 5.12]{T3}. This finishes the corresponding glueing of vector bundle process in this specific situation.
\end{proof}

\

\subsection{Application to Equivariant Sheaves over Adic Fargues-Fontaine Curves}

\indent We now apply our study to the first specific situation around the corresponding mixed-type Hodge structures inspired essentially by the work \cite{CKZ}, \cite{PZ} and \cite{Ked1}, which is also considered in \cite{T3}. Let us first describe the corresponding situation:

\begin{setting}
We will use the notations in \cite{T1}, \cite{T2} and \cite{T3}. We now consider the corresponding Fargues-Fontaine curves as our target to relativization. We will use the corresponding $\mathrm{FF}_R$ to denote the adic Fargues-Fontaine curve in our current situation (as in \cite[Theorem 4.6.1]{KL2}) which is defined to be the corresponding suitable quotient of the following adic space:
\begin{align}
\bigcup_{0<r_1<r_2} \mathrm{Spa}(\widetilde{\Pi}^{[r_1,r_2]}_R,\widetilde{\Pi}^{[r_1,r_2],+}_R).	
\end{align}
Here $R$ is specified perfectoid Banach uniform algebra over $\mathbb{F}_{p^h}$.
\end{setting}

\begin{assumption}
In this section we assume that the corresponding base field will be just of mixed-characteristic.	
\end{assumption}

\indent In our current framework we look at the following deformation:

\begin{setting}
Over $\mathbb{Q}_p$ we consider the any Galois perfectoid tower $(H_\bullet,H_\bullet^+)$ (in the sense of \cite[Definition 5.1.1]{KL2}) with Galois group $\mathrm{Gal}_H:=\mathrm{Gal}(\overline{H}_0/H_0)$, and we use the notation $\widetilde{H}_\infty$ to denote the corresponding perfectoid top level of the tower in mixed-characteristic situation, and we consider the corresponding ring $\overline{H}_\infty$ as its tilt. Now we consider the period rings $B_{\mathrm{dR},\overline{H}_\infty}^+$, $B_{\mathrm{dR},\overline{H}_\infty}$ and $B_{e,\overline{H}_\infty}$ (as constructed in \cite[Definition 8.6.5]{KL2}). 	
\end{setting}

\indent We now try to be more specific, where we just consider the situation where the ring $\widetilde{H}_\infty$ is just the ring $\mathbb{C}_p$. We now consider some Banach algebra $B$ over $\mathbb{Q}_p$ as before. Now what is happening is then we consider the following ringed space:

\begin{align}
(\mathrm{FF}_R,\mathcal{O}_{\mathrm{FF}_R}\widehat{\otimes}B\widehat{\otimes}B_{\mathrm{dR},\overline{H}_\infty}^+,\mathcal{O}_{\mathrm{FF}_R}\widehat{\otimes}B\widehat{\otimes}B_{\mathrm{dR},\overline{H}_\infty},\mathcal{O}_{\mathrm{FF}_R}\widehat{\otimes}B\widehat{\otimes}B_{e,\overline{H}_\infty}).	
\end{align}

\indent Now we encode Berger's $B$-pair in some relativization after Kedlaya-Liu (\cite[Definition 9.3.11]{KL1} and \cite[Definition 4.8.2]{KL2}) and Kedlaya-Pottharst (\cite[Definition 2.17]{KP}):

\begin{definition}
We now define a pseudocoherent $\mathrm{FF}_R$-$B$-$\mathrm{Gal}_H$ sheaf to be a triplet $(M^+_\mathrm{dR},M_\mathrm{dR},M_e)$ of $B\widehat{\otimes}?$-stably pseudocoherent sheaves such that the form a corresponding descent triplet with respect to the following diagram by natural base change:
\begin{align}
\mathcal{O}_{\mathrm{FF}_R}\widehat{\otimes}B\widehat{\otimes}B_{\mathrm{dR},\overline{H}_\infty}^+ \longrightarrow \mathcal{O}_{\mathrm{FF}_R}\widehat{\otimes}B\widehat{\otimes}B_{\mathrm{dR},\overline{H}_\infty}\longleftarrow \mathcal{O}_{\mathrm{FF}_R}\widehat{\otimes}B\widehat{\otimes}B_{e,\overline{H}_\infty},	
\end{align}
where $?=B_{\mathrm{dR},\overline{H}_\infty}^+,B_{\mathrm{dR},\overline{H}_\infty},B_{e,\overline{H}_\infty}$. For the rings:
\begin{align}
\mathcal{O}_{\mathrm{FF}_R}\widehat{\otimes}B\widehat{\otimes}B_{e,\overline{H}_\infty}	
\end{align}
the corresponding definition of pseudocoherent sheaves are defined to be locally associated to the $B\widehat{\otimes}B_{e,\overline{H}_\infty}$-stably pseudocoherent modules in the ind-Fr\'echet topology, namely the completeness with respect to the ind-Fr\'echet topology is then stable under the corresponding rational localization concentrated at the part for $\mathrm{FF}_R$. And we assume that the corresponding sheaves admit $\mathrm{Gal}_H$-equivariant action in the semilinear fashion.	
\end{definition}

\indent In the \'etale topology we consider the following ringed space:

\begin{align}
(\mathrm{FF}_R,\mathcal{O}_{\mathrm{FF}_R,\text{\'et}}\widehat{\otimes}B\widehat{\otimes}B_{\mathrm{dR},\overline{H}_\infty}^+,\mathcal{O}_{\mathrm{FF}_R,\text{\'et}}\widehat{\otimes}B\widehat{\otimes}B_{\mathrm{dR},\overline{H}_\infty},\mathcal{O}_{\mathrm{FF}_R,\text{\'et}}\widehat{\otimes}B\widehat{\otimes}B_{e,\overline{H}_\infty}).	
\end{align}

\indent Now we encode Berger's $B$-pair in some relativization after Kedlaya-Liu (\cite[Definition 9.3.11]{KL1} and \cite[Definition 4.8.2]{KL2}) and Kedlaya-Pottharst (\cite[Definition 2.17]{KP}):

\begin{definition}
We now define a pseudocoherent $\mathrm{FF}_R$-$B$-$\mathrm{Gal}_H$ sheaf to be a triplet $(M^+_\mathrm{dR},M_\mathrm{dR},M_e)$ of $B\widehat{\otimes}?$-\'etale-stably pseudocoherent sheaves such that these form a corresponding descent triplet with respect to the following diagram by natural base change:
\begin{align}
\mathcal{O}_{\mathrm{FF}_R,\text{\'et}}\widehat{\otimes}B\widehat{\otimes}B_{\mathrm{dR},\overline{H}_\infty}^+ \longrightarrow \mathcal{O}_{\mathrm{FF}_R,\text{\'et}}\widehat{\otimes}B\widehat{\otimes}B_{\mathrm{dR},\overline{H}_\infty}\longleftarrow \mathcal{O}_{\mathrm{FF}_R,\text{\'et}}\widehat{\otimes}B\widehat{\otimes}B_{e,\overline{H}_\infty},	
\end{align}
where $?=B_{\mathrm{dR},\overline{H}_\infty}^+,B_{\mathrm{dR},\overline{H}_\infty},B_{e,\overline{H}_\infty}$. For the rings:
\begin{align}
\mathcal{O}_{\mathrm{FF}_R,\text{\'et}}\widehat{\otimes}B\widehat{\otimes}B_{e,\overline{H}_\infty}	
\end{align}
the corresponding definition of pseudocoherent sheaves are defined to be locally associated to the $B\widehat{\otimes}B_{e,\overline{H}_\infty}$-\'etale-stably pseudocoherent modules in the ind-Fr\'echet topology, namely the completeness with respect to the ind-Fr\'echet topology is then stable under the corresponding \'etale morphisms concentrated at the part for $\mathrm{FF}_R$. And we assume that the corresponding sheaves admit $\mathrm{Gal}_H$-equivariant action in the semilinear fashion.	
\end{definition}

\begin{definition}
We now define a projective $\mathrm{FF}_R$-$B$-$\mathrm{Gal}_H$ sheaf to be a triplet $(M^+_\mathrm{dR},M_\mathrm{dR},M_e)$ of locally finite projective sheaves such that these form a corresponding descent triplet with respect to the following diagram by natural base change:
\begin{align}
\mathcal{O}_{\mathrm{FF}_R}\widehat{\otimes}B\widehat{\otimes}B_{\mathrm{dR},\overline{H}_\infty}^+ \longrightarrow \mathcal{O}_{\mathrm{FF}_R}\widehat{\otimes}B\widehat{\otimes}B_{\mathrm{dR},\overline{H}_\infty}\longleftarrow \mathcal{O}_{\mathrm{FF}_R}\widehat{\otimes}B\widehat{\otimes}B_{e,\overline{H}_\infty},	
\end{align}
where $?=B_{\mathrm{dR},\overline{H}_\infty}^+,B_{\mathrm{dR},\overline{H}_\infty},B_{e,\overline{H}_\infty}$. For the rings:
\begin{align}
\mathcal{O}_{\mathrm{FF}_R}\widehat{\otimes}B\widehat{\otimes}B_{e,\overline{H}_\infty}	
\end{align}
the corresponding definition of projective sheaves are defined to be locally associated to the finite projective modules in the ind-Fr\'echet topology, namely the completeness with respect to the ind-Fr\'echet topology is then stable under the corresponding rational localization concentrated at the part for $\mathrm{FF}_R$. And we assume that the corresponding sheaves admit $\mathrm{Gal}_H$-equivariant action in the semilinear fashion.	
\end{definition}

\indent In the \'etale topology we consider the following ringed space:

\begin{align}
(\mathrm{FF}_R,\mathcal{O}_{\mathrm{FF}_R,\text{\'et}}\widehat{\otimes}B\widehat{\otimes}B_{\mathrm{dR},\overline{H}_\infty}^+,\mathcal{O}_{\mathrm{FF}_R,\text{\'et}}\widehat{\otimes}B\widehat{\otimes}B_{\mathrm{dR},\overline{H}_\infty},\mathcal{O}_{\mathrm{FF}_R,\text{\'et}}\widehat{\otimes}B\widehat{\otimes}B_{e,\overline{H}_\infty}).	
\end{align}

\indent Now we encode Berger's $B$-pair in some relativization after Kedlaya-Liu (\cite[Definition 9.3.11]{KL1} and \cite[Definition 4.8.2]{KL2}) and Kedlaya-Pottharst (\cite[Definition 2.17]{KP}):

\begin{definition}
We now define a projective $\mathrm{FF}_R$-$B$-$\mathrm{Gal}_H$ sheaf to be a triplet $(M^+_\mathrm{dR},M_\mathrm{dR},M_e)$ of locally finite projective sheaves such that these form a corresponding descent triplet with respect to the following diagram by natural base change:
\begin{align}
\mathcal{O}_{\mathrm{FF}_R,\text{\'et}}\widehat{\otimes}B\widehat{\otimes}B_{\mathrm{dR},\overline{H}_\infty}^+ \longrightarrow \mathcal{O}_{\mathrm{FF}_R,\text{\'et}}\widehat{\otimes}B\widehat{\otimes}B_{\mathrm{dR},\overline{H}_\infty}\longleftarrow \mathcal{O}_{\mathrm{FF}_R,\text{\'et}}\widehat{\otimes}B\widehat{\otimes}B_{e,\overline{H}_\infty},	
\end{align}
where $?=B_{\mathrm{dR},\overline{H}_\infty}^+,B_{\mathrm{dR},\overline{H}_\infty},B_{e,\overline{H}_\infty}$. For the rings:
\begin{align}
\mathcal{O}_{\mathrm{FF}_R,\text{\'et}}\widehat{\otimes}B\widehat{\otimes}B_{e,\overline{H}_\infty}	
\end{align}
the corresponding definition of projective sheaves are defined to be locally associated to the finite projective modules in the ind-Fr\'echet topology, namely the completeness with respect to the ind-Fr\'echet topology is then stable under the corresponding \'etale morphisms concentrated at the part for $\mathrm{FF}_R$. And we assume that the corresponding sheaves admit $\mathrm{Gal}_H$-equivariant action in the semilinear fashion.	
\end{definition}

\begin{remark}
One can basically define the same $B$-$\mathrm{Gal}_H$-sheaves over any adic Banach affinoid algebra or space in the same fashion which we will not repeat again.
\end{remark}

\indent We now prove the following theorems:

\begin{theorem} \mbox{\bf{(After Kedlaya-Liu \cite[Theorem 4.6.1]{KL2})}} \label{theorem5.18}
Consider the following categories:\\
I. The category of all the projective $\mathrm{FF}_R$-$B$-$\mathrm{Gal}_H$ sheaves over the site:
\begin{align}
(\mathrm{FF}_R,\mathcal{O}_{\mathrm{FF}_R}\widehat{\otimes}B\widehat{\otimes}B_{\mathrm{dR},\overline{H}_\infty}^+,\mathcal{O}_{\mathrm{FF}_R}\widehat{\otimes}B\widehat{\otimes}B_{\mathrm{dR},\overline{H}_\infty},\mathcal{O}_{\mathrm{FF}_R}\widehat{\otimes}B\widehat{\otimes}B_{e,\overline{H}_\infty}).	
\end{align}
II. The category of all the compatible families (with locally glueability and local cocycle condition) of all the projective $\{\mathrm{Spa}(\widetilde{\Pi}_R^{[s,r]},\widetilde{\Pi}_R^{[s,r],+})\}_{\{[s,r]\}}$-$B$-$\mathrm{Gal}_H$ sheaves over the sites:
\begin{align}
&(\mathrm{Spa}(\widetilde{\Pi}_R^{[s,r]},\widetilde{\Pi}_R^{[s,r],+}),\mathcal{O}_{\mathrm{Spa}(\widetilde{\Pi}_R^{[s,r]},\widetilde{\Pi}_R^{[s,r],+})}\widehat{\otimes}B\widehat{\otimes}B_{\mathrm{dR},\overline{H}_\infty}^+,\\
&\mathcal{O}_{\mathrm{Spa}(\widetilde{\Pi}_R^{[s,r]},\widetilde{\Pi}_R^{[s,r],+})}\widehat{\otimes}B\widehat{\otimes}B_{\mathrm{dR},\overline{H}_\infty},\mathcal{O}_{\mathrm{Spa}(\widetilde{\Pi}_R^{[s,r]},\widetilde{\Pi}_R^{[s,r],+})}\widehat{\otimes}B\widehat{\otimes}B_{e,\overline{H}_\infty}),
\end{align}	
for any $[s,r]\subset (0,\infty)$, carrying the corresponding semilinear action of the Frobenius coming from just the Robba ring part.\\
III. The category of all the compatible families (with locally glueability and local cocycle condition) of all the projective $\{\widetilde{\Pi}_R^{[s,r]}\}_{\{[s,r]\}}$-$B$-$\mathrm{Gal}_H$ modules over the rings:
\begin{align}
(\mathcal{O}_{\mathrm{Spa}(\widetilde{\Pi}_R^{[s,r]},\widetilde{\Pi}_R^{[s,r],+})}\widehat{\otimes}B\widehat{\otimes}B_{\mathrm{dR},\overline{H}_\infty}^+,\mathcal{O}_{\mathrm{Spa}(\widetilde{\Pi}_R^{[s,r]},\widetilde{\Pi}_R^{[s,r],+})}\widehat{\otimes}B\widehat{\otimes}B_{\mathrm{dR},\overline{H}_\infty},\mathcal{O}_{\mathrm{Spa}(\widetilde{\Pi}_R^{[s,r]},\widetilde{\Pi}_R^{[s,r],+})}\widehat{\otimes}B\widehat{\otimes}B_{e,\overline{H}_\infty})|_{\mathrm{Spa}(\widetilde{\Pi}_R^{[s,r]},\widetilde{\Pi}_R^{[s,r],+})},
\end{align}	
for any $[s,r]\subset (0,\infty)$, carrying the corresponding semilinear action of the Frobenius coming from just the Robba ring part.\\
IV. The category of all the projective $\widetilde{\Pi}_R^{\infty}$-$B$-$\mathrm{Gal}_H$ modules over the rings:
\begin{align}
(\mathcal{O}_{\mathrm{Spa}(\widetilde{\Pi}_R^{\infty},\widetilde{\Pi}_R^{\infty,+})}\widehat{\otimes}B\widehat{\otimes}B_{\mathrm{dR},\overline{H}_\infty}^+,\mathcal{O}_{\mathrm{Spa}(\widetilde{\Pi}_R^{\infty},\widetilde{\Pi}_R^{\infty,+})}\widehat{\otimes}B\widehat{\otimes}B_{\mathrm{dR},\overline{H}_\infty},\mathcal{O}_{\mathrm{Spa}(\widetilde{\Pi}_R^{\infty},\widetilde{\Pi}_R^{\infty,+})}\widehat{\otimes}B\widehat{\otimes}B_{e,\overline{H}_\infty})|_{\mathrm{Spa}(\widetilde{\Pi}_R^{\infty},\widetilde{\Pi}_R^{\infty,+})},
\end{align}	
carrying the corresponding semilinear action of the Frobenius coming from just the Robba ring part.\\
Then we have the corresponding categories above are equivalent.

\end{theorem}

\begin{proof}
The corresponding families of sheaves with respect to the intervals, the families of modules with respect to the intervals and the sheaves over $\mathrm{FF}_R$ could be related equivalently through the functor taking the corresponding local global section, by applying \cref{theorem4.14} and \cref{theorem3.15}. The corresponding families of sheaves with respect to the intervals, the families of modules with respect to the intervals and the modules over $\widetilde{\Pi}^\infty_R$ could be related equivalently through the functor taking the corresponding local projective limit global section, by applying \cref{theorem4.31} and \cref{theorem3.32} and by considering the corresponding Frobenius pullback to control the rank throughout. 	
\end{proof}

\begin{theorem} \mbox{\bf{(After Kedlaya-Liu \cite[Theorem 4.6.1]{KL2})}} \label{theorem5.24}
Consider the following categories:\\
I. The category of all the pseudocoherent sheaves over the site:
\begin{align}
(\mathrm{FF}_R,\mathcal{O}_{\mathrm{FF}_R,\text{\'et}}\widehat{\otimes}B).	
\end{align}
II. The category of all the compatible families (with locally glueability and local cocycle condition) of all the pseudocoherent $\{\mathrm{Spa}(\widetilde{\Pi}_R^{[s,r]},\widetilde{\Pi}_R^{[s,r],+})\}_{\{[s,r]\}}$ sheaves over the sites:
\begin{align}
&(\mathrm{Spa}(\widetilde{\Pi}_R^{[s,r]},\widetilde{\Pi}_R^{[s,r],+}),\mathcal{O}_{\mathrm{Spa}(\widetilde{\Pi}_R^{[s,r]},\widetilde{\Pi}_R^{[s,r],+}),\text{\'et}}\widehat{\otimes}B),
\end{align}	
for any $[s,r]\subset (0,\infty)$, carrying the corresponding semilinear action of the Frobenius coming from just the Robba ring part.\\
III. The category of all the compatible families (with locally glueability and local cocycle condition) of all the stably-pseudocoherent $\{\widetilde{\Pi}_R^{[s,r]}\}_{\{[s,r]\}}$ modules over the rings:
\begin{align}
\mathcal{O}_{\mathrm{Spa}(\widetilde{\Pi}_R^{[s,r]},\widetilde{\Pi}_R^{[s,r],+}),\text{\'et}}\widehat{\otimes}B|_{\mathrm{Spa}(\widetilde{\Pi}_R^{[s,r]},\widetilde{\Pi}_R^{[s,r],+})},
\end{align}	
for any $[s,r]\subset (0,\infty)$, carrying the corresponding semilinear action of the Frobenius coming from just the Robba ring part.\\
IV. The category of all the stably-pseudocoherent $\widetilde{\Pi}_R^{\infty}$ modules over the rings:
\begin{align}
\mathcal{O}_{\mathrm{Spa}(\widetilde{\Pi}_R^{\infty},\widetilde{\Pi}_R^{\infty,+}),\text{\'et}}\widehat{\otimes}B|_{\mathrm{Spa}(\widetilde{\Pi}_R^{\infty},\widetilde{\Pi}_R^{\infty,+})},
\end{align}	
carrying the corresponding semilinear action of the Frobenius coming from just the Robba ring part.\\
Then we have the corresponding categories above are equivalent.

\end{theorem}

\begin{proof}
The corresponding families of sheaves with respect to the intervals, the families of modules with respect to the intervals and the sheaves over $\mathrm{FF}_R$ could be related equivalently through the functor taking the corresponding local global section, by applying \cref{theorem2.22}. The corresponding families of sheaves with respect to the intervals, the families of modules with respect to the intervals and the modules over $\widetilde{\Pi}^\infty_R$ could be related equivalently through the functor taking the corresponding local projective limit global section, by applying \cref{theorem2.31} and by considering the corresponding Frobenius pullback to control the rank throughout. 	
\end{proof}

\newpage

\subsection*{Acknowledgements} 

We would like to thank Professor Kedlaya for helpful discussion on the corresponding extension of the corresponding foundations in \cite{KL2} and the AWS Lecture Notes \cite{Ked2} related to the corresponding contexts in our current presentation, in particular the corresponding glueing pseudocoherent modules and the corresponding glueing vector bundles in the setting of \cite{KL1} which is also beyond the corresponding seminar we participated related to the AWS Lecture Notes \cite{Ked2} although the corresponding earliest version of the material we studied during the seminar contains less.

\newpage

\bibliographystyle{ams}

\end{document}